\documentclass[letterpaper]{amsart} 
\usepackage{amsthm}
\usepackage{mathtools,mathrsfs,amssymb,latexsym,graphics,enumerate}
\usepackage[mathscr]{eucal}
\usepackage{amsmath,amsfonts,amsthm,amssymb,bbm,esint,cite}
\usepackage{color}
\usepackage{hyperref}
\usepackage[left=1in,right=1in,top=1in,bottom=1in]{geometry}
\numberwithin{equation}{section}

\newcommand{\rr}{\mathbb{R}}

\newcommand{\lan}{\langle}
\newcommand{\ran}{\rangle}
\newcommand{\be}{\begin{eqnarray*}}
\newcommand{\bel}{\begin{eqnarray}}
\newcommand{\ee}{\end{eqnarray*}}
\newcommand{\eel}{\end{eqnarray}}
\newcommand{\ba}{\begin{aligned}}
\newcommand{\ea}{\end{aligned}}
\newcommand{\de}{\Delta}

\newcommand{\na}{\nabla}
\newcommand{\ep}{\epsilon}

\newcommand{\rg}{\right}
\newcommand{\lf}{\left}
\newcommand{\pa}{\partial}
\newcommand{\wh}{\widehat}
\newcommand{\wt}{\widetilde}
\newcommand{\nq}{{\neq}}
\newcommand{\Mass}{ {\mathfrak {M}}}
\newcommand{\wwt}[1]{\breve{#1}}

\newcommand{\pal}{{\partial_x^{i}\partial_y^j\partial_z^k}}

\newcommand{\CC}{\mathfrak{B}_{2,\infty}}
\newcommand{\te}{{\zeta}} 

\newcommand{\cz}{{\lan  \mathfrak {C}\ran}}
\newcommand{\nz}{{\lan n\ran}}

\newcommand{\by}{Y}
\newcommand{\cc}{{ \mathfrak {C}}}
\newcommand{\myb}[1]{{#1}}
\newcommand{\myr}[1]{{#1}}
\newcommand{\myc}[1]{}
\newcommand{\Dy}{\mathbb{T}^2}
\newcommand{\mf}{\mathfrak }

\newtheorem{theorem}{Theorem}

\newtheorem{lem}{Lemma}
\newtheorem{pro}{Proposition}
\newtheorem{rmk}{Remark}

\numberwithin{remark}{section}
\numberwithin{lem}{section}

\numberwithin{theorem}{section}
\numberwithin{pro}{section}

\newcommand{\norm}[1]{\left\lVert#1\right\rVert}
\newcommand{\abs}[1]{\left\vert#1\right\vert}

\newcommand\Torus{{\mathbb T}}

\newcommand{\dss}{\displaystyle}

\allowdisplaybreaks

\mathtoolsset{showonlyrefs=true}

\title[Time-dependent Flows and Their Applications]{Time-dependent Flows and Their Applications in Parabolic-parabolic Patlak-Keller-Segel Systems\\ Part II: Shear Flows}
\date{}

\author{Siming He}
\address{
University of South Carolina, Columbia, SC, 29208}
\email{
siming@mailbox.sc.edu
}

\begin{document}
\maketitle
\begin{center}\large\emph{Dedicated to Eitan Tadmor on occasion of his 70th birthday}\end{center}
\begin{abstract}
In this study, we investigate the behavior of three-dimensional parabolic-parabolic Patlak-Keller-Segel (PKS) systems in the presence of ambient shear flows. Our findings demonstrate that when the total mass of the cell density is below a specific threshold, the solution remains globally regular as long as the flow is sufficiently strong. The primary difficulty in our analysis stems from the fast creation of chemical gradients due to strong shear advection.
\end{abstract}
 
\setcounter{tocdepth}{1}{\small\tableofcontents}

\section{Introduction}
We consider the parabolic-parabolic Patlak-Keller-Segel systems (PKS), which model the chemotaxis phenomena in fluid: 
\begin{align}\label{ppPKS_basic}
\lf\{\begin{array}{cc}\ba\pa_t n+&Au(t,y)\pa_x n+\na \cdot( n\na \cc)=\de n,\\
\pa_t  \mathfrak C+&Au(t,y)\pa_x  \mathfrak C=\de  \mathfrak C+n-\overline{n},\\
n(t&=0 )=n_{\mathrm{in}} ,\quad  \mathfrak C(t=0 )=\cc_{\mathrm{in}},\\
(x,&\by)=(x,y,z)\in\mathbb{T}^3=(-\pi,\pi]^3,.\ea\end{array}\right.
\end{align} 
\ifx\textcolor{red}{Today Tarek tells me that the following shear flow (?)
$u(t,y ,z )=(\cos(z +\log(1+t)),0)$ have enhanced dissipation estimate:
\begin{align}
\|\rho_{\neq}(t)\|_2\leq C_{ED}\|\rho_{\neq;\mathrm{in}}\|_2e^{-\delta \nu^{1/3}t}, \quad \forall {t\in[0,\infty)}.
\end{align} Therefore, we might be able to use this time dependent shear flow on the $\Torus^3$ to derive the $8\pi$ critical mass. By alternating these flows, we might get a flow suppressing two dimension $(u(t,y),0,0)$ and $(0,u(t,x),0)$. But the gradient is also growing fast? Is it possible to redo Section \ref{Sec:Linfty_est} with alternating flow? }
\fi
We define $n,  \mathfrak C$ as the cell and chemo-attractant densities, respectively. The evolution of the cell density  $n$ incorporates the chemical-induced aggregation, diffusion, and fluid advection. Throughout the paper, we consider the fluid fields of the form $(Au(t,y),0,0)$, with $A:=\|Au\|_{L_{t,y}^\infty}$ denoting their magnitudes. Next, we focus on the equation for the chemo-attractant density $\mathfrak C$. Without altering the dynamics, one can normalize the equation by subtracting the spatial average $\overline{n}:=\fint_{\Torus^3} n$. Finally, we assume that the initial $\mathfrak C$-density is average-free, i.e., $\overline{\cc_{\mathrm{in}}} =0.$ One can check that the average-free property is preserved in time, i.e., $\overline{\cc}(t)\equiv0$. 

A simple rescaling in time yields the following equivalent system to \eqref{ppPKS_basic}:
\begin{align}
\label{ppPKS}\lf\{\begin{array}{rr}\ba\pa_t n+&u\pa_x n={A}^{-1}(\de n-\na \cdot( n\na \cc)),\\
\pa_t \cc+&u\pa_x \cc={A}^{-1}(\de \cc+n-\overline{n}),\\
n(t&=0)=n_{\mathrm{in}} ,\quad \cc(t=0 )=\cc_{\text{in}}.\ea\end{array}\rg. 
\end{align}
Since this reformulation is more convenient for presentation purposes, we focus on \eqref{ppPKS} from now on. 

The system \eqref{ppPKS_basic} is a simplified model among a large family of the coupled Patlak-Keller-Segel-fluid systems. The literature on the subject is vast, and we refer the interested readers to the following papers, \cite{Lorz10,Lorz12,LiuLorz11,DuanLorzMarkowich10,FrancescoLorzMarkowich10, Winkler12,TaoWinkler,ChaeKangLee13,
KozonoMiuraSugiyama,Tuval05, GongHe20} and the references therein.

If there is no ambient fluid flows, i.e., $A=0$, the equation is the classical parabolic-parabolic PKS equation. The PKS equations were first derived in  \cite{Patlak,KS}. The literature on the the classical parabolic-parabolic PKS equations and their variants  is large, and we refer the interested readers to the papers \cite{CarrapatosoMischler17,EganaMischler16,CalvezCorrias,Schweyer14,Winkler13, BlanchetEJDE06,BlanchetCarrilloMasmoudi08,EganaMischler16, Biler95} and the references within. We summarize the results on the blow-up and global regularity result of the classical parabolic-parabolic PKS equations here. In two-dimension, the total mass of cells $\mathfrak{M}:=\|n\|_{L^1}$ characterizes the long time behavior of the solution. If the total mass is strictly less than $8\pi$, the solutions are globally regular, \cite{CalvezCorrias}, \cite{Mizoguchi13}. On the other hand, if the total mass is large enough, singularities  form in a finite time, see, R. Schweyer \cite{Schweyer14}. In dimension three and higher, the PKS equations become supercritical.  In the paper \cite{Winkler13}, the author showed that there exist solutions,  which have arbitrary small masses, blow up in a finite time. 

In this paper, we study the long-time behavior of the system \eqref{ppPKS_basic}. In work \cite{He}, the author demonstrates that strictly monotone shear flows could suppress the chemotactic blow-up in \eqref{ppPKS_basic} in two-dimension. In a recent development (\cite{ZengZhangZi21}),  the authors extend this result to a coupled Patlak-Keller-Segel-Navier-Stokes system. However, further developments are required to understand the fluid flow-induced suppression of the blow-up mechanism in the 3D-system \eqref{ppPKS_basic}. We initiate this line of study in \cite{ElgindiHe22I}. In that work, we establish the suppression of blow-up results by choosing alternating shear flows. The main challenges in analyzing the parabolic-parabolic system \eqref{ppPKS_basic} stem from the fast chemical gradient $\na\cc$ creation driven by the strong fluid advection. This growth of $\na\cc$ leads to fast nonlinearity inflation, which has the potential to counter-balance the regularization effect induced by the fluid flow. 

Our first main theorem is the following:
\begin{theorem}
\label{thm_2}
Consider the solutions to the equation \eqref{ppPKS_basic} subject to initial data $0<n_{\mathrm{in}}\in   H^{M}(\Torus^3)$, $\cc_{\mathrm{in}}\in  H^{M+1}(\Torus^3)$ with $ M\geq 4$.  Assume that the total mass of cells is bounded \begin{align}\label{mss_cnstr0}\|n_{\mathrm{in}}\|_{L^1(\Torus^3)}<8\pi|\Torus|.\end{align}
Then there exist a time-dependent shear flow $(u_A(t,y ),0,0)$ with $\|u_A\|_{L_{t}^\infty W_{y}^{M+4} }\leq 1$ and a threshold \[A_0(\Mass,M,\|n_{\mathrm{in}}\|_{ H^{M}}, \|\lan \cc\ran_{\mathrm{in}}\|_{  H^{M+1}})\] such that the solutions $(n,\cc)$ to  \eqref{ppPKS_basic} remain globally regular for all time provided that the magnitude $A$ of the shear is greater than $A_0$. 
\end{theorem}\myc{
\footnote{
\myb{The following theorem should be true. 
\begin{theorem}
Let $n_{\text{in}}, \cc_{\text{in}}$ be $C^\infty(\Torus)$ and have support in the strip $y\in[-\frac{\pi}{4}, \frac{\pi}{4}]$. Let $U_A=A\sin(y)$ and $\|n_{\text{in}}\|_{L^1}\leq? 16\pi^2$. There exists a threshold $A_0(n_{\text{in}}, \cc_{\text{in}})$ such that if $A\geq A_0$, the solution is globally smooth. 
\end{theorem}}}}To understand the intuition behind the suppression of chemotactic blow-up through shear flow, we first examine the blow-up mechanism in the classical setting ($A=0$). In the PKS-type systems, there are two competing forces - the nonlinear aggregation ($\na \cdot( n\na \cc)$) and the diffusion ($\de n$). The nonlinear aggregation drives cells to concentrate and form Dirac singularities. On the other hand, diffusion disperses the cells and regularizes the dynamics. In the subcritical regime, diffusion dominates, and the solutions remain smooth for all time. In the supercritical regime, however, aggregation dominates, and singularities form. One natural way to suppress the blow-up is by enhancing the diffusion effect. Diffusion can be enhanced by replacing the linear diffusion operator with nonlinear diffusion,  as seen in \cite{Blanchet09, BRB10, CalvezCarrillo06}. However, we seek to explore another enhanced diffusion phenomenon: the enhanced dissipation effect induced by fluid advection. Here, the fluid advection creates a drastic oscillation in the cell density, amplifying the Laplacian's dissipation effect. Once the dissipation effect dominates the nonlinear aggregation, suppression of blow-up follows. This argument is proven successful in dealing with supercritical parabolic-elliptic PKS equations and other nonlinear equations \cite{KiselevXu15, IyerXuZlatos, BedrossianHe16, FengFengIyerThiffeault20, CotiZelatiDolceFengMazzucato}. The enhanced dissipation phenomena were applied to other important fluid problems. For example, in the study of the nonlinear stability of the shear flows, the mixing-induced enhanced dissipation effect is crucial in deriving a sharp  stability threshold, see, e.g., \cite{BMV14, BVW16, BGM15I, BGM15II, BGM15III, ChenLiWeiZhang18, WeiZhang19,
WeiZhangZhao20,CotiZelatiElgindiWidmayer20}. On the other hand, enhanced dissipation plays a major role in understanding turbulence and anomalous dissipation. See, e.g., \cite{DrivasEtal19anomalous, BedrossianBlumenthalPunshonSmith19,BedrossianBlumenthalPunshonSmith191, BedrossianBlumenthalPunshonSmith192}. 

In this paper, we focuses on the enhanced dissipation effect of the shear flows. To this end, we consider the passive scalar equation associated with the shear flow $u$:
\begin{align}
\pa_t f+ u(t,Y) \pa_x f={A}^{-1}\de f,\quad f(t=0)=f_{\mathrm{in}}.\label{ps_intro_1}
\end{align}
Furthermore, we define the $x$-average $\lan f\ran$ and the remainder $f_{\neq}$ of a function $f\in L^1(\Torus^3)$:
\begin{align}
\lan f\ran(\by)=\frac{1}{|\mathbb{T}|}\int_{\mathbb{T}} f(x,\by)dx,\quad f_{\neq}(x,\by)=f(x,\by)-\lan f\ran(\by),\label{n_0_n_neq}
\end{align}
Most of the estimates introduced below are derived in a two-dimensional setting. However, they can be easily extended to $\Torus^3$.
  
We first consider the \emph{stationary} shear flow $u(t,y)=u(y)$. If the shear flow profile  $u(y)$ has only finitely many nondegenerate critical points ($u'(y_\dagger)=0,\, u''(y_\dagger)\neq0$), then the flow is called nondegenerate shear flow. Direct $L^2$ estimate of the equation \eqref{ps_intro_1} yields that the passive scalar solutions decay on the time scale $\mathcal{O}(A)$, i.e., 
\begin{align*}
\|f(t)-\overline{f}\|_{L^2}\leq \|f_{\text{in}}-\overline{f}\|_{L^2}\exp\{-A^{-1}\  t\}.
\end{align*}The decay is slow if the viscosity $A^{-1}$ is small. However, by focusing on the remainder $f_\nq$ part of the solution, one observes a much faster decay phenomenon. In their paper \cite{BCZ15}, J. Bedrossian and M. Coti Zelati applied the hypocoercivity technique (Villani, \cite{villani2009}) to show that for the nondegenerate shear flows, there exist positive constants $\delta_0, C_\ast $ such that the following estimate holds for a small viscosity $A^{-1}$:
\begin{align*} 
\|f _{\neq}(t)\|_{L^2}\leq C_\ast \|f _{\mathrm{in};\neq}\|_{L^2} e^{-\delta_0  A^{-1/2}|\log A|^{-2} t},\quad \forall t\geq 0.
\end{align*}
If the viscosity is small, the dissipation time scale $\mathcal{O}(A^{1/2}|\log A|^2)$ is much shorter than the heat dissipation time scale $\mathcal{O}(A )$. This fast decay of the remainder $f_\nq$ of the passive scalar solution \eqref{ps_intro_1} is called the \emph{enhanced dissipation}. Heuristically, one can view this phenomenon as a fast homogenization of the density $f$ in the shear direction. Later, D. Wei applied resolvent estimates and a Gearhart-Pr\"{u}ss type theorem (\cite{Wei18}) to derive the sharp estimate
\begin{align}\label{ED_intro_2} 
\|f _{\neq}(t)\|_{L^2}\leq C_\ast \|f _{\mathrm{in};\neq}\|_{L^2} e^{-\delta_0 A^{-1/2}t},\quad \forall t\geq 0.
\end{align}
In the paper \cite{CotiZelatiDrivas19}, M. Coti-Zelati and D. Drivas utilized stochastic methods to confirm that the $A^{-1/2}$-enhanced dissipation rate in \eqref{ED_intro_2} is sharp for nondegenerate stationary shear flows. The enhanced dissipation phenomena are believed to be a consequence of the phase mixing effect induced by shear flows. A quantitative relationship between this mixing effect of fluid flow and the enhanced dissipation is rigorously established in work \cite{ElgindiCotiZelatiDelgadino18}, \cite{FengIyer19}. Recently, D. Albritton, R. Beekie, and M. Novak applied H\"ormander's hypoelliptic technique to prove the sharp estimate \eqref{ED_intro_2} (\cite{AlbrittonBeekieNovack21}).   

Note that the above enhance-dissipation estimate is sharp for  \emph{stationary} shear flows, which leaves open the question of whether one can improve the dissipation rate by relaxing the \emph{stationary} constraint. The authors showed that for a sufficiently general family of time dependent shear flows, the enhanced dissipation rate is $\mathcal{O}(A^{-1/3})$.
\begin{theorem}[\cite{Elgindi20,ElgindiHe22I} ]\label{Thm:Lnr_ED}
Consider the solutions $f_{\neq}$ to the passive scalar equations \eqref{ps_intro_1}.  There exists a shear flow $u_A\in C^\infty_{t,y}$ such that the following enhanced dissipation estimate is satisfied
\begin{align}
\| S_{s}^{s+t}f_{\neq}(s)\|_{L ^2}\leq &C_\ast \| f_{\neq}(s)\|_{L^2}e^{-\delta_0  A^{-1/3}t},\label{ED_tmshr}
\end{align}
for all positive time $s,t\geq 0$. Here $S_{s}^{s+t}$ is the linear propagator associated with the equation $\eqref{ps_intro_1}_{u=u_A}$, i.e., $S_{s}^{s+t}f_\nq(s)=f_\nq(t).$ The constants $\delta_0 $ and $C_\ast $ are independent of $A$ and the solutions. Moreover, the shear flow $u_A$ is bounded, i.e., $\|u_A\|_{L_t^\infty W_y^{M,\infty}}\leq C(M),\quad \forall M\in\mathbb{N}.$
\end{theorem}
Motivated by the results developed in the previous literature, we say that a shear flow $u$ exhibits \emph{$A^{-1/3}$-relaxation enhancing property} if all the solutions $f\in L_t^\infty L^2$ to the passive scalar equation \eqref{ps_intro_1} satisfy the decay estimate
\begin{align}\label{ED1/3}
\|f_\nq (s+t)\|_{L^2}\ \leq\  C_\ast\  \|f_\nq (s)\|_{L^2}\  \exp\lf\{-\delta_0\ {A^{-1/3}\ {t}}\rg\},\quad \forall s, t\geq 0. 
\end{align}
Here $C_\ast\geq 1,\,\delta_0\in(0,1)$ are constants depending only on the shear flow $u.$  

With the enhanced dissipation effect of shear flow introduced, we are ready to lay out the battle plan. First, we observe that the enhanced dissipation estimate  \eqref{ED_intro_2} is only valid for the remainder $f_\nq$ of the passive scalar solution $f$. To better understand this constraint, we consider passive scalar solutions which are homogeneous in the $x$-direction, i.e., $f(x,Y)=\lan f\ran(Y)$. Since these solutions solve the one-dimensional heat equation, the fluid transport effect is vacuous, and no enhanced dissipation is expected. To capture the heterogeneous nature of the shear flow enhanced dissipation effect, we decompose the PKS solution $n$ into the $x$-average part $\lan n\ran$ and the remainder $n_{\neq}$ using \eqref{n_0_n_neq}. The detailed equations for each component will be presented in the next section. 

The main idea of the proof is that we can treat the evolution of the remainder $n_{\neq}$ as a perturbation to the passive scalar equation \eqref{ps_intro_1}. Since the remainder of the passive scalar solution dissipates in a fast time scale, we expect that the remainder $n_{\neq}$  will decay quickly and not alter the dynamics of the whole system. On the other hand, as the $x$-average $\lan n\ran$ is invariant under the shear flow, it solves a modified parabolic-parabolic PKS equation in dimension two. Moreover, since the total mass of the solution is a critical quantity in dimension two, the $8\pi$-mass constraint \eqref{mss_cnstr0} guarantees higher regularity estimation on the solution.   

The above idea works well in the parabolic-elliptic PKS equations. However, new challenges arise when we come to the parabolic-parabolic PKS equations. In the paper \cite{He}, the author observed that the effect of the strong shear on the system \eqref{ppPKS_basic} is two-fold. On the one hand, it enhances the diffusion effect of cell density. However, on the other hand, it triggers a strong growth of the chemical gradient, destabilizing the cell density dynamics through aggregation nonlinearity. Fortunately, some smallness constraints are enough to control this instability.

The next main theorem in this paper is the following:
\begin{theorem}\label{thm_1}
Consider the solutions to the equation \eqref{ppPKS_basic} subject to initial data $n_{\mathrm{in}}\in H^{M}(\Torus^3)$, $\cc_{\mathrm{in}}\in  H^{M+1}(\Torus^3)$ with $M\geq 3$. Assume that the following assumptions hold:

\noindent
a) The flow $u$ is regular, i.e., $\|u\|_{L_t^\infty W_y^{M+3,\infty}}\leq C(M)<\infty$, and exhibits the ${A}^{-1/3}$-relaxation enhancing property \eqref{ED1/3}.

\noindent
b) There exists a universal constant $C$ such that the following smallness condition is satisfied
\begin{align}\label{smll_nq}
\myr{A^{M+1}\|n_{\mathrm{in};\nq}\|_{H^M(\Torus^3)}^2+A^{M+2}\|\cc_{\mathrm{in};\neq}\|_{ H^{M+1} (\mathbb{T}^3)}^2\leq C.}
\end{align}

\noindent
c) The total mass of cells is bounded, 
\begin{align}\label{mss_cnstr}
\|\lan n_{\mathrm{in}}\ran \|_{L^1_{Y}(\Dy)}=\frac{\|n_{\mathrm{in}}\|_{L_{x,Y}^1(\mathbb{T}^3)}}{|\mathbb{T}|}=\frac{\Mass}{2\pi}<\myr{8}\pi.
\end{align} 

\noindent
Then there exists a threshold $A_0(u,\Mass,\|n_{\mathrm{in}}\|_{ H^{M}}, \|\lan \cc\ran_{\mathrm{in}}\|_{  H^{M+1}})$ such that for any $A\geq A_0$, the solutions $(n,\cc)$ to the equation \eqref{ppPKS_basic} remain globally regular in time. 
\end{theorem}

With more effort, we are able to derive the following theorem for a short time:
\begin{theorem} [\cite{ElgindiHe22I}]\label{thm:short_t}
Consider the solutions $(n,\cc)$ to the equation \eqref{ppPKS} initiated from the data $n_{\mathrm{in}}\in  H^M(\Torus^3)$, $\cc_{\mathrm{in}}\in  H^{M+1}(\Torus^3),\ M\geq 5$. Define a parameter 
\begin{align}\label{te_M_chc}
\te(M)= \frac{1}{108(2+M)}.
\end{align} The shear flow $(u_A(t,y),0,0)$ is the one  in Theorem \ref{Thm:Lnr_ED}. There exists a threshold \[A_0=A_0(\|n_{\mathrm{in}}\|_{H^M}, \|\cc_{\mathrm{in}}\|_{H^{M+1}}, \|u_A\|_{L_t^\infty W^{M+3,\infty}}, M)\] such that if $A\geq A_0$, then there exists a universal constant $C$ such that the following estimate holds at time instance $A^{1/3+\te(M)}$,
\begin{align}\label{est_short_t}
\|n_\nq(A^{1/3+\te})\|_{H^{M-1}}^2+\|\cc_\nq(A^{1/3+\te})\|_{H^{M}}^2\leq C \exp\lf\{-\frac{A^{\te}}{C}\rg\}.
\end{align} Moreover, the $x$-average is bounded as follows
\begin{align}\label{est_shrt_t_0}
\|\lan n\ran (A^{1/3+\te}) \|_{H^M}^2+\|\lan \cc\ran(A^{1/3+\te})\|_{H^{M+1}}^2\leq \mathfrak{B}(\|n_{\mathrm{in}}\|_{H^M}, \|\cc_{\mathrm{in}}\|_{H^{M+1}}, \|u_A\|_{L_t^\infty W^{M+3,\infty}}, M). 
\end{align} 
\end{theorem}
As a corollary of Theorem  \ref{Thm:Lnr_ED}, Theorem  \ref{thm_1}, and Theorem \ref{thm:short_t}, we have Theorem \ref{thm_2}.

\myc{\footnote{\myb{Another potential way to give $A^{-1/3}$-R.E. flow is first to consider the strictly monotone shear flow, then extend its drivative to the whole real line through periodic extension. Finally we integrate the derivative to gain the full profile. }}}

In the next section, we lay down the battle plan of the paper. 

\ifx 
We also use the notation $A\lesssim B$ (or $A\gtrsim B$) to denote the fact that there exists a strictly positive constant $C$ such that $A\leq CB$ (or $A\geq \frac{1}{CB}$, respectively). Throughout the paper, the letter $C$ denotes constants which can change from line to line. 
Throughout the paper, we apply Fourier transformation in the $x$ variable or in the  $(x,z)$ variables. The  frequency variables corresponding $x$ and $z $ are denoted by $k$ and $\ell$, respectively. 
The notation $\widehat{(\cdot)}$ is used to denote the Fourier transform and the $(\cdot)^\vee$ is used to denote the inverse transform. If we consider the transformation only in the $x-$variable, the Fourier transform and its inverse has the following form
\begin{equation*}
\wh{f}_k(y ,z ) := \frac{1}{2\pi}\int_{-\pi}^\pi e^{-ikx} f(x,y ,z ) dx , \quad \check{g}(x,z ,z ) = \sum_{k=-\infty}^\infty g_k(y ,z ) e^{ikx}.
\end{equation*}
Similarly, we can define Fourier transform in the $x,z $-variables. In that case, we use $\wh f_{k,\ell}(y )$ to represent the transformed functions. The notation $|\pa_x|^{1/2}$ should be understood as Fourier multipliers with symbol $|k|^{1/2}$, i.e.,
\begin{align} 
 M(\pa_x,\pa_{z })f(x,y ,z )=\left(M(ik,i\ell)\widehat{f}(k,y ,\ell)\right)^\vee.
\end{align}\fi\ifx
We will also use the notation $S_{t_\star,t_\star+t}^k$ to denote the solution semigroup corresponding to the $k$-by-$k$ equation
\begin{align}
\pa_t \wh\rho_k+u(t,y )ik \wh\rho_k =\frac{1}{A} (\de_y-|k|^2) \wh\rho_k, \quad \,\forall y\in\Torus ^2,\label{k-by-k-PS}
\end{align}
which is initiating from $t=t_\star$.
\fi

\ifx
On the domain $\mathbb{T} \times [0,1] \times \mathbb{T}$ subject to the Neumann boundary condition for both chemical and the bacteria density, there should be a way to prove global in time boundedness of the solution using the free
energy method. Here the critical mass should be $4\pi$. Moreover, it might be possible to use the contradiction method to prove the PPPKS case. We just need to set that $\|\na c\|_\infty\lesssim \frac{1}{A^{...}} $ for all time.
\textcolor{red}{My guess: As long as the equation preserve certain kind of symmetry, then we can make the Neumann, Dirichlet boundary condition equivalent to the Periodic boundary condition with double or other times the size of the channel. This guessed equivalent relation might be useful. I hope the equality to hold.}
\fi

\section{Battle Plan}\label{Sec:Btstrp}

\ifx 
Define the $x$-average and remainder of $\cc$ as follows:
\begin{align}
\cz (t,y ,z )=\frac{1}{2\pi}\int_\Torus \cc(t, x, y ,z )dx,\ \ \cc_{\neq}(t,x,y ,z )= \cc(t,x,y ,z )-\cz (t,y ,z ).
\end{align}
\fi 
Following the ideas in the introduction, we write down the equations for the $x$-averages $\nz ,\, \cz $:
\begin{subequations}\label{ppPKS_0md}
\begin{align}
\pa_t \nz +&\frac{1}{A}\na_Y \cdot( \nz \na_Y \cz )+\frac{1}{A}\na_Y  \cdot\lan n_{\neq}\na_Y   \cc_{\neq}\ran=\frac{1}{A}\de \nz ,\label{ppPKS_n}\\
\pa_t \cz =&\frac{1}{A}\de \cz +\frac{1}{A}\nz -\frac{1}{A}\overline{n},\label{ppPKS_c}\\
\nz (t=&0,y ,z )=\lan n_{\mathrm{in}}\ran(y ,z ),\quad \cz (t=0,y ,z )=\lan \cc_{\mathrm{in}}\ran(y ,z ),\quad (y ,z )\in\mathbb{T}^2.\label{ppPKS_BC}
\end{align}
\end{subequations}
On the other hand, the equations for the remainders $n_{\neq},\, \cc_{\neq}$ are the following:
\begin{subequations}\label{ppPKS_nqmd}
\begin{align}
\pa_t n_{\neq}+u_A(t,y ) \pa_x n_{\neq}+&\frac{1}{A}\na_Y \cdot(n_{\neq}\na_Y \cz )+\frac{1}{A}\na \cdot(\nz\na \cc_{\neq} )+\frac{1}{A}\na\cdot(n_{\neq}\na \cc_{\neq})_{\neq}=\frac{1}{A}\de n_{\neq},\label{ppPKS_n_neq}\\
\pa_t \cc_{\neq}+u_A(t,y )\pa_x \cc_{\neq}=&\frac{1}{A}\de \cc_{\neq}+\frac{1}{A}n_{\neq},\label{ppPKS_c_neq}\\
n_{\neq}(t=0,x,y ,z )&=n_{{\mathrm{in}};\neq}(x,y ,z ),\quad \cc_{\neq}(t=0,x,y ,z )=\cc_{{\mathrm{in}};\neq}(x,y ,z ),\quad (x,y ,z )\in\Torus^3.
\end{align}
\end{subequations}

\ifx
Now we  treat the first and second equations in the system \eqref{ppPKS} as a perturbation of the passive scalar equation subject to small viscosity:
\begin{align}\label{PS_introduction} 
\pa_t \rho_{\neq}+u(t,y )\pa_x \rho_{\neq} =\frac{1}{A} \de \rho_{\neq},\quad \int_{\Torus}\rho_{\mathrm{in};\neq}(x,y ,z )dx \equiv 0.
\end{align}
If we do the Fourier transform in the $x$-variable and the $z $-variable, we obtain that
\begin{align}\label{PS_introduction_Fourier} 
\pa_t \wh\rho_{k,\ell}(t,y )+u(t, y )ik\wh \rho_{k,\ell}(t,y ) =\frac{1}{A} (\pa_{y y }-|k|^2-|\ell|^2 )\widehat \rho_{k,\ell}(t,y ).
\end{align}
It is derived in the papers \cite{BCZ15,Wei18} that if the shear flow is non-degenerate and time-independent, the remainder $\rho_{\neq}$ decay as follows:
\begin{align}
\|\rho_{\neq}\|_{L^2}\leq C \|\rho_{\mathrm{in};\neq}\|_{L^2}e^{-\delta \frac{1}{A ^{1/2}}t},\quad \forall t\geq 0.
\end{align}
If the shear flow is time-independent, these results are shown to be sharp in the work \cite{CotiZelatiDrivas19}. However, if we allow time-dependent shear flow, the enhanced dissipation rate can be improved. This is the content of the next lemma.
\myr{Use two lemmas. The first lemma on $[0, \nu^{-1/3}]$, the second lemma on rewinding.}
\begin{lem}\label{lem:ED} Consider solution $\rho_{\neq}$ to the equation  \eqref{PS_introduction} subject to the time varying shear flow $(u(t,y ),0,0)=(\cos(y +\log(1+t)),0,0)$ and the following average-zero constraint on the initial data
\begin{align}
\int_{\mathbb {T}}\rho_{\mathrm{in};\neq}(x,y ,z )dx=&0,\quad \forall y\in \Torus^2.
\end{align}
Then the following enhanced dissipation estimate holds
\begin{align}
\|\wh \rho_{k,\ell}(t)\|_2\leq &C \|\wh \rho_{k,\ell}(0)\|_2e^{-\delta  \frac{1}{A^{1/3}}t-\frac{1}{2A} (|k|^2+|\ell|^2)t},\\
\|\rho_{\neq}(t)\|_2\leq& C \|\rho_{\neq}(0)\|_2e^{-\delta  \frac{1}{A^{1/3}}t-\frac{1}{2A}  t},\quad \forall t\geq 0.%
\end{align}
\end{lem} 
\fi
To prove the suppression of blow-up, we do a bootstrap argument. One has the following local well-posedness result through standard argument. \myc{(? Check!?)}
\begin{theorem}\label{thm:lcl_exst}
Consider solutions $n,\cc$ to the equation \eqref{ppPKS_basic} subject to initial data $n_{\mathrm{in}}\in H^{M}(\Torus^3),\, c_{\mathrm{in}}\in H^{M+1}(\Torus^3),\, M\geq 3$ and regular flow $u\in L_t^\infty W^{M+5,\infty}$. There exists a small constant $T_\varepsilon(\|n_{\mathrm{in}}\|_{H^M},\|c_{\mathrm{in}}\|_{H^{M+1}})$ such that the solution exists on the time interval  $[0,T_\varepsilon]$. 
\end{theorem}

We also have the following blow-up criterion for the equation \eqref{ppPKS_basic}. We refer the interested readers to the appendix of \cite{He} for the proof in $2$-dimension. The 3-dimensional case is similar. \myc{(? Check!?)}
\begin{theorem}\label{thm:blwp_crt}
Consider solutions $(n,\cc)$ to the  equation \eqref{ppPKS_basic} subject to initial data $n_{\mathrm{in}}\in H^{M }(\Torus^3),\, \cc_{\mathrm{in}}\in H^{M+1}(\Torus^3),\, M\geq 3$ and regular flow $u\in L_t^\infty W^{M+5,\infty}$. If the $H^M\times H^{M+1}$-solutions $(n, \cc)$ blow up at time $T_\star<\infty$, then the $L^\infty$-norm of the solutions $n$ must approach infinity at $T_\star$, i.e., 
$\lim_{t\rightarrow T_\star}\sup_{0\leq s\leq t}\|n(s)\|_{L^\infty_{x,y}}=\infty.$
\end{theorem}
To lay down the bootstrap, we first introduce the functional:
\myc{\footnote{\begin{align}
\dot F_{m}[n_\nq, \cc_\nq]:=&\sum_{i+j+k=m}A^{(1/3+\eta)2i+3\eta }\|\pa_x^i\pa_y^j\pa_z^k n_{\neq}\|_2^2+\sum_{i+j+k=m+1}A^{(1/3+\eta)2i}\|\pa_{x}^i\pa_y^j\pa_z^k\cc_\nq\|_2^2,\\
F_{M}=&\sum_{0\leq m\leq M}\dot F_{m}.
\end{align}
}}
\begin{align}
\label{F_M_sect_2} 
\quad F_{M}[n_\nq, \cc_\nq]=&\sum_{i+j+k\leq M}A^{i+\frac {1}{2}} \|\pa_x^i\pa_y^j\pa_z^kn_\nq\|_{L^2}^2+\sum_{i+j+k\leq M+1}A^{i}  \|\pa_x^i\pa_y^j\pa_z^k \cc_\nq\|_{L^2}^2.
\end{align}
We observe that thanks to the smallness constraint \eqref{smll_nq}, the initial value of the functional $F_M$ is bounded, i.e.,\begin{align}
F_M[n_{\text{in};\nq},c_{\text{in};\nq}]\leq C. \label{choice_L_1}
\end{align}

Assume that $[0,T_\star)$ is the largest interval on which the following hypotheses hold:

\noindent
\begin{subequations}\label{Hypotheses}

1) Remainders' enhanced dissipation estimates: 
\begin{align}
F_{M}[n_\nq(t), \cc_\nq(t)]\leq 4e^2F_{M}[n_{\text{in};\nq},\cc_{\text{in};\nq}] \exp\lf\{-\frac{2\delta t}{A^{1/3}}\rg\},\quad \forall t\in[0,T_\star). \label{HypED_FM}
\end{align}
\myr{Here $\delta=\delta(\delta_0, C_*)$ (with $\delta_0,\ C_*$ being chosen in Theorem \ref{Thm:Lnr_ED}) is defined in \eqref{Chc_delta}.}

2) Uniform-in-time estimates of the $x$-averages (zero-mode in $x$):
\begin{align}
\|\lan n\ran \|_{L_t^\infty([0,T_\star);H^M)}\leq& 2\mf B_{\lan n\ran,H^M}(\|\lan n_{\text{in}}\ran\|_{H^M}, \|\lan \cc_{\text{in}}\ran\|_{H^M},\Mass);\label{Hyp_0_nHM}\\
\|\na_Y   \lan \cc\ran \|_{L_t^\infty([0,T_\star);H^{M})}\leq& 2 \mf B_{\na \lan \cc\ran,H^{M}}(\|\lan n_{\text{in}}\ran\|_{H^M}, \|\lan \cc_{\text{in}}\ran\|_{H^M}, \Mass).\label{Hyp0_nac_HM}
\end{align} 
\myr{Here the bounds $\mf B_{\lan n\ran,H^M},\,\,\mf B_{\na \lan \cc\ran,H^{M}}$ are defined in \eqref{chc_B_avg}.}  
\end{subequations}
By the local well-posedness of the equation \eqref{ppPKS} in $H^M,\, M\geq 3$ (Theorem \ref{thm:lcl_exst}), we have that the interval $[0,T_\star)$ is non-empty. Also note that due to the smallness constraint \eqref{smll_nq}, we have that $F_M(0)\leq C$ for some constant $C$. 

The hypotheses \eqref{Hypotheses}, together with the $L^\infty$ Sobolev embedding, imply the uniform-in-time  $L^\infty_tL^\infty_{x,Y}$-bound on the solutions:
\begin{align}
\myr{ A^{1/2} }\|n_\nq(t)\|_{L_{x,Y}^\infty}^2+\myr{ A }\|\na \cc_\nq(t)\|_{L_{x,Y}^\infty}^2\leq C F_M[n_{\nq}(t)&,\cc_{\nq}(t)]\leq C F_M[n_{\text{in};\nq},\cc_{\text{in};\nq}],\quad\forall t\in[0,T_\star),\, M\geq 4;\label{Bd_Linf_nq}\\
\|n\|_{L_t^\infty([0,T_\star);L_{x,Y}^\infty)}+\|\na \cc\|_{L_t^\infty([0,T_\star);L_{x,Y}^\infty)}\leq& \mf B_{\infty}(F_M(0), \mf B_{\lan n \ran; H^M}, \mf B_{\na \lan \cc\ran; H^{M}}).\label{Bound_Linf}
\end{align}

We further define the following quantity
\begin{align}
\CC=&1+\Mass+\|n_{\mathrm{in}}\|_{H_{x,Y}^{M}}+\|\cc_{\mathrm{in}}\|_{ H_{x,Y}^{M+1}}+\sqrt{F_M(0)}+\mf B_{\nz ;H^M}+\mf B_{\na \cz ;H^{M}}+\mf B_{\infty}.\label{CC}
\end{align}

The remaining part of the paper is devoted to the proof of the following proposition:
\begin{pro}\label{Pro:main}
Consider the solution to the equation \eqref{ppPKS} subject to all the constraints in Theorem \ref{thm_1}. Then if the $A$ is chosen large enough depending on the initial data, then the following set of conclusions hold \begin{subequations}\label{Conclusions}
1)  Remainders' enhanced dissipation estimates: 
\begin{align}
F_M[n_\nq(t), \cc_\nq(t)] \leq \myr{2e^2} F_M[n_{\mathrm{in};\nq},\cc_{\mathrm{in};\nq}] \exp\lf\{-\frac{2\delta t}{A^{1/3}}\rg\},\quad \forall t\in[0, T_\star).\label{ConED_FM}
\end{align}
2) Uniform-in-time estimates of the $x$-averages:
\begin{align}
\|\nz \|_{L_t^\infty([0,T_\star);H_Y^M)}\leq& \mf B_{\nz ,H^M};\label{Con0_n_HM}\\
\|\na_Y   \cz \|_{L_t^\infty([0,T_\star);H_Y^M)}\leq&  \mf B_{\na\cz ,H^{M} }.\label{Con0nac_HM+1} 
\end{align}  
\end{subequations}
\end{pro}
\myc{\footnote{Double check which one gives open condition and which one gives the closed condition. }}We prove Theorem \ref{thm_1} by applying Proposition \ref{Pro:main} and Theorem \ref{thm:lcl_exst}, Theorem \ref{thm:blwp_crt}. First note that by the local existence (Theorem \ref{thm:lcl_exst}), the interval $[0,T_\star)$ is nonempty. Now by improving the bootstrap hypotheses \eqref{Hypotheses} to the conclusions \eqref{Conclusions} (Proposition \ref{Pro:main}) and recalling the continuity in time of the Lebesgue and Sobolev norms involved, we show that the estimates \eqref{Conclusions} persist on any finite time interval in $[0,\infty)$. As a result, we derive uniform-in-time $L^\infty$-bound on the solutions $n$. By Theorem \ref{thm:blwp_crt}, we have that the higher $H^s$ Sobolev regularities propagate. Therefore, the solutions are globally regular. This concludes the proof of Theorem \ref{thm_1}.

\subsection{Sketch of the Estimates}
In this subsection, we discuss the main ingredients involved in the proof of the Proposition \ref{Pro:main}. 

\noindent \textbf{A. Enhanced Dissipation Estimates.} The first challenge stems from deriving the nonlinear enhanced dissipation  \eqref{ConED_FM}. In Section \ref{Sec:ED}, we invoke the linear enhanced dissipation assumption  \eqref{ED1/3} to derive the nonlinear enhanced dissipation estimate  \eqref{ConED_FM}. Here we adapt the idea from previous literature, \cite{KiselevXu15, ElgindiCotiZelatiDelgadino18,IyerXuZlatos, HeTadmorZlatos}. The main intuition is that in order to obtain the enhanced dissipation, it is enough to consider the dynamics on a relatively short time interval of length $\mathcal{O}(A^{1/3})$. On this short time interval (compared to the classical heat decay time scale $\mathcal{O}(A)$), the nonlinear effect is perturbative. Hence, it is possible to derive enhanced dissipation by comparing the nonlinear solutions to associated passive scalar solutions. However, even though this idea works well in the parabolic-elliptic PKS equations, it faces challenges in the parabolic-parabolic PKS system \eqref{ppPKS} due to the shear-flow induced nonlinear instabilities. The following diagram describes how strong shear flows induce instability:
\begin{align}
\text{Strong shear } \Rightarrow &\text{ Fast growth of the chemical gradient }\na \cc_{\neq}\\
\Rightarrow &\text{ Fast growth of the nonlinear aggregation }\frac{1}{A}\na \cdot( n\na \cc).
\end{align}
As a result, we observe that the aggregation undergoes fast growth in the small time scale $\mathcal{O}(A^{1/3})$, and this nonlinearity might not be perturbative if no extra control is introduced. 

To overcome this difficulty, we invoke the smallness  \eqref{smll_nq} and design a  functional ${F}_M$ \eqref{F_M_sect_2}, which is bounded independent of the parameter $A$. The $F_M$ records the transient growth of various components of the solutions. By carefully analyzing the time evolution of this functional ${F}_M$, we end up with the proof of \eqref{ConED_FM}. 

\noindent
\textbf{B. The $x$-average Estimates.} Another difficulty comes from estimating the $x$-average quantities $\nz , \,\, \na \cz$. In the previous works on suppression of blow-up through shear flows,  \cite{BedrossianHe16} and \cite{He}, the $x$-average of the solution must be uniformly bounded in time. If this is not guaranteed, then the growing $x$-average part will force the remainder $n_{\neq}, \,\, \na \cc_{\neq}$ to deviate away from the linear passive scalar dynamics and break the argument in the long run. 

For the 3-dimensional PKS equations, the $x$-average part solves the modified 2-dimensional parabolic-parabolic PKS equations. However, to our knowledge, there are no previous results concerning the uniform-in-time boundedness of the solutions on the plane $\rr^2$. The closest result to our goal is the exponential-in-time bound derived by V. Calvez and L. Corrias in  \cite{CalvezCorrias}.  
  
To overcome this difficulty,  we confine our system to the bounded domain $  \mathbb{T}^2$, where it is possible to derive the uniform-in-time estimate on the quantities $\nz , \,\,\na \cz $. Here we use an argument similar to the one in \cite{CalvezCorrias}. The main idea is to compare the parabolic-parabolic PKS solutions to the parabolic-elliptic PKS solutions. Once this is completed, we can apply the techniques from the paper \cite{BedrossianHe16} to derive the bound. 

The remaining part of the paper is organized as follows: in Section \ref{Sec:0-mode}, we estimate the $x$-average quantities $\nz,\ \cz$, \eqref{Con0_n_HM}, \eqref{Con0nac_HM+1}; in Section \ref{Sec:ED}, we prove the enhanced dissipation of the remainder $n_{\neq},\, \na \cc_{\neq}$ \eqref{ConED_FM} and hence concludes the proof of Proposition \ref{Pro:main}.  

\subsection{Notation} In this subsection, we introduce various notations and abbreviations that we use in the sequel. Throughout the paper, the constants $B, C$ change from line to line. However, the constants $C_{(\cdot)}$, e.g., $ C_{\lan n\ran; H^M}$ will be defined and fixed unless otherwise stated. The bootstrap constants will be denoted by $\mathfrak B_{(\cdots)}$.  Since the Sobolev norm of the shear flow $u(t,y)$ is assumed to be uniformly bounded in time, i.e., $\|u\|_{L_t^\infty W^{M+3,\infty}_y}\leq C(M)$, we drop the explicit dependence of the constant $C$ on the Sobolev norm of $u$ and the regularity index $M$. 

The spatial average of a function $f$ on $\Torus^3$ is denoted by $\dss\overline{f}:=\frac{1}{|\Torus|^3}\int_{\Torus^3}f(x,y,z)dxdydz. $

Recall the classical $L^p$ norms and Sobolev $H^s$ norms:
\begin{align}
\|f\|_{L_x^p}=&\|f\|_p=\left(\int |f|^p dx\right)^{1/p};\quad \|f\|_{L_t^q([0,T]; L^p_x)}=\left(\int_0^T\|f(t,x)\|_{L_x^p}^qdt\right)^{1/q};\\
\|f\|_{H^M}=&\left(\sum_{i+j+k\leq M}\|\pa_x^i\pa_y^j\pa_z^k f\|_{L^2}^2\right)^{1/2};\quad
\|f\|_{\dot H^M}=\left(\sum_{i+j+k=M}\|\pa_x^i\pa_y^j\pa_z^k f\|_{L^2}^2\right)^{1/2}.
\end{align}
We use the notation $(i_1,j_1,k_1)\leq (i_2, j_2,k_2)$ to denote the multi-index relation $i_1\leq i_2,\, j_1\leq j_2$ and $k_1\leq k_2. $
Moreover, we use the notation $|i,j,k|:=i+j+k. $

Throughout the paper, we use the notation $S_{t_\star}^{t_\star+t}$ to denote the semigroup corresponding to the following passive scalar equation initiating from time $t_\star$.
\begin{align}\pa_t \rho+u(t,y )\pa_x \rho ={A}^{-1} \de \rho.
\end{align}
Adopting this notation, $S_{t_\star}^{t_\star+t} \rho_{\neq}$ is the solution to the passive scalar subject to the initial data $\rho_{\neq}(t_\star)$ at time $t_\star$. 

We reuse the notation $T_i$ to denote various terms in the proof to avoid cumbersome notations—namely, the meaning of each $T_i$ term changes from lemma to lemma.


\section{The $x$-average Estimates} \label{Sec:0-mode}
In this section, we consider the $x$-average $(\nz ,\, \cz )$'s evolution \eqref{ppPKS_0md} and prove the conclusions \eqref{Con0_n_HM},  \eqref{Con0nac_HM+1}. 

By the divergence structure of the cell density equation \eqref{ppPKS_n}, we obtain that the total mass $\Mass=\|n\|_{L^1}$ is preserved. Standard integration by parts yields that the average chemical concentration is also preserved. To conclude, we have the following conservations:
\begin{align}
\|\nz (t)\|_{L^1(\Torus^2)}=&\|\lan n_{\mathrm{in}}\ran\|_{L^1}=\frac{\Mass}{2\pi}=:\Mass_0;\label{Cllmss_cnsrv}\\
\int_{\mathbb{T}^2}\cz (t,\by)d\by=&\int_{\mathbb{T}^2} \lan \cc_{\mathrm{in}}\ran(\by)d\by=0,\quad \forall t\in[0,T_\star].\label{Chem_cnsrv}
\end{align}  
Recall that the interval $[0,T_\star]$ is the  bootstrap  time horizon.  
We study the following free energy of the $x$-average 
\begin{equation}\label{Free_energy}
E[\nz ,\cz ]:=\int_{\Torus^2} \nz \log \nz dY+\frac{1}{2}\int_{\Torus^2}|\na \cz |^2d\by-\int_{\Torus^2} \cz (\nz -\overline{n})d\by.
\end{equation}
We further define the entropy and the  chemical potential:
\begin{equation}S[\nz]:=\int_{\Torus^2}\nz\log \nz dY,\qquad
P[\nz ,\cz ]:=\frac{1}{2}\int_{\Torus^2} |\na \cz |^2d\by -\int_{\Torus^2} \cz (\nz -\overline{n})d\by.\label{ChemPot}
\end{equation}
Note that 
\begin{equation}
E[\nz ,\cz ]=S[\nz]+P[\nz ,\cz ].
\end{equation}
The free energy \eqref{Free_energy} is bounded and this is the content of the following lemma. 
\begin{lem}
Consider the solution to the $x$-averaged equation \eqref{ppPKS_0md}. The time evolution of the free energy has the form
\begin{align}\label{PKSfrentmev}
\frac{d}{dt}E[\nz ,\cz ]=&-\frac{1}{A}\int \nz \bigg|\na_Y   \log \nz  -\na_Y   \cz \bigg|^2dY-{A}\int|\pa_t\cz |^2dY+\frac{1}{A}\int \na_Y   \log \nz \cdot \lan n_{\neq}\na_Y   \cc_{\neq}\ran dY\\
&-\frac{1}{A}\int \na_Y   \cz \cdot\lan n_{\neq}\na_Y   \cc_{\neq}\ran dY .
\end{align}
Moreover, the following estimate holds
\begin{align}
A\int_0^t\|\pa_t \cz (s)\|_{L_Y^2}^2ds+E[\nz (t),\cz (t)]\leq &E[\lan n_{\mathrm{in}}\ran, \lan\cc_{\mathrm{in}}\ran]+\frac{\myr{1}}{\delta  A^{1/3}} C(\mf B_{2,\infty}),\quad \forall t\in[0,T_\star].\label{frentm_est}
\end{align}
Here $\CC$ are defined in \eqref{CC}.
\end{lem}
\begin{proof}
First we calculate the time evolution of $E$ using the integration by partsas follows 
\begin{align}\label{E_dssp_pf_1}
\frac{d}{dt}E=&\int \pa_t \nz \bigg(\log \nz -\cz \bigg)dY-\int (\nz -\overline{n})\pa_t\cz dY+\frac{d}{dt}\left(\frac{1}{2}\int|\na \cz |^2dY\right) \\
=&-\frac{1}{A}\int \na_Y  \cdot(\nz \na_Y  \log \nz -\nz \na_Y   \cz )\bigg(\log \nz -\cz \bigg)dY-\int (\nz -\overline{n})\pa_t\cz dY\\
&+\frac{d}{dt}\left(\frac{1}{2}\int|\na_Y   \cz |^2dY\right) -\frac{1}{A}\int\na_Y   \cdot \lan n_{\neq}\na_Y   \cc_{\neq}\ran(\log \nz -\cz )dY\\
=&-\frac{1}{A}\int \nz |\na_Y  \log \nz -\na_Y   \cz |^2dY-\int (\nz -\overline{n})\pa_t\cz dY+\frac{1}{2}\frac{d}{dt}\left(\int|\na_Y   \cz |^2dY\right)\\
&+\frac{1}{A}\int \lan n_{\neq}\na_Y   \cc_{\neq}\ran\cdot(\na_Y   \log \nz -\na_Y   \cz )dY.
\end{align}
Next we test the equation \eqref{ppPKS_c} with $\pa_t \cz $, and get 
\begin{equation}
\int | \pa_t\cz |^2dY=-\frac{1}{2A}\frac{d}{dt}\int |\na_Y   \cz |^2dY+\frac{1}{A}\int \pa_t\cz  (\nz -\overline{n})dY.
\end{equation}
Now applying this equation in \eqref{E_dssp_pf_1} yields that
\begin{align}\label{ddtE=ZT}
\frac{d}{dt}E=&-\frac{1}{A}\int \nz |\na_Y  \log \nz -\na_Y   \cz |^2dY-A\int |\pa_t \cz |^2dY\\
&+\frac{1}{A}\int \lan n_{\neq}\na_Y   \cc_{\neq} \ran\cdot \na_Y   \log \nz dY-\frac{1}{A}\int \lan n_{\neq}\na_Y   \cc_{\neq} \ran\cdot \na_Y   \cz dY\\
=:&-\frac{1}{A}\int \nz |\na_Y  \log \nz -\na_Y   \cz |^2dY-A\int |\pa_t \cz |^2dY+T_1+T_2.
\end{align}
To finish the proof of the lemma, we estimate the last two terms in the equation \eqref{ddtE=ZT}. The last term $NZ_2$ in \eqref{ddtE=ZT} is estimated through the H\"older inequality and  Hypotheses \eqref{HypED_FM}, \eqref{Hyp0_nac_HM} as follows
\begin{align}\label{NZ_2}
T_2\leq& \frac{1}{A}\|\na _Y \cc_{\neq}\|_2\|n_{\neq}\|_2\|\na_Y   \cz \|_\infty\leq \frac{1}{A} C F_M[n_{\text{in};\nq},\cc_{\text{in},H^M}]  \mf B_{\na  \cz ,H^M}e^{-\frac{2\delta}{A^{1/3}}t}.
\end{align} For the term $T_1$ in \eqref{ddtE=ZT}, we first prove the following relation, which is a natural consequence of the fact that $n(x,Y)=\nz (Y)+n_{\neq}(x,Y)\geq 0$,
\begin{align}\label{E_dssp_pf_2} 
\frac{\|n_{\neq}(x,Y)\|_{L_x^1}}{\nz (Y)}\leq 2|\Torus|,\quad \forall Y\in\Torus^2.
\end{align}
The proof  is as follows. Since $n(x,Y)\geq 0$ for all $(x,Y)\in\Torus^3$, we have that $n_{\neq}^-(x,Y)\leq \nz (Y),\quad \forall x\in\mathbb{T}$. Now the average zero condition $\int_{\mathbb{T}}n_{\neq}(x,Y)dx=0$ yields that $\|n_{\neq}(\cdot,Y)\|_{L^1_x}=2\|n_{\neq}^-(\cdot,Y)\|_{L_x^1}\leq 2\nz (Y)|\mathbb{T}|$. Moreover, since the initial density is assumed to be strictly positive, i.e., $n_{\text{in}}>0$, parabolic comparison argument yields that $n(t,x,Y)>0$ for arbitrary finite $t<\infty.$ Hence, the $x$-average $\nz(t,Y)$ is nonzero for $\forall t\in[0,\infty)$, $\forall Y\in \Torus^2$. As a result, we have the estimate \eqref{E_dssp_pf_2}. Now 
we estimate the $T_1$ term in \eqref{ddtE=ZT} with Gagliardo-Nirenberg inequality in $1$-dimension as follows:
\begin{align}
\frac{1}{A}\int \frac{\na_Y   \nz }{\nz }\cdot\lan n_{\neq}\na_Y   \cc_{\neq}\ran dY=&\frac{1}{A|\Torus|}
\int  {\na_Y   \nz (Y)}{}\cdot\frac{\int_{\mathbb{T}}n_{\neq}(x,Y)\na_Y   \cc_{\neq}(x,Y)dx}{\nz (Y)}dY\\
\leq &\frac{1}{A|\Torus|}\int|\na_Y   \nz (Y)|\|\na_Y   \cc_{\neq}(\cdot,Y)\|_{L_x^\infty}\frac{\|n_{\neq}(\cdot,Y)\|_{L_x^1}}{\nz (Y)}dY\\
\leq&\frac{2}{A}C\int |\na_Y   \nz (Y)|\|\na_Y   \cc_{\neq}(\cdot,Y)\|_{L_x^2}^{1/2}\|\pa_x \na_Y   \cc_{\neq}(\cdot,Y)\|_{L_x^2}^{1/2}dY.
\end{align}
Now applying Young's inequality, H\"older inequality and   Hypotheses \eqref{HypED_FM} and \eqref{Hyp_0_nHM} yields the following result
\begin{align}
T_1\leq&\frac{C}{A} \|\na_Y   \nz \|_2(\|\na_Y   \cc_{\neq}\|_{L_{x,Y}^2}+\|\pa_x\na_Y   \cc_{\neq}\|_{L_{x,Y}^2}) 
\leq  \frac{C}{A}\CC^3 e^{-\frac{\delta t}{A^{1/3}}}.\label{NZ_1}
\end{align}Combining \eqref{ddtE=ZT}, \eqref{NZ_2},   and \eqref{NZ_1} yields
\begin{align} 
\frac{d}{dt}E[\nz ,\cz ]\leq&-\frac{1}{A}\int \nz \lf|\na_Y   \log \nz  -\na_Y   \cz \rg|^2dY-{A}\int|\pa_t\cz |^2dY +\frac{1}{A}C(\CC)e^{-\frac{\delta t}{A^{1/3}  }}.
\end{align}
Now integrating in time yields the estimate \eqref{frentm_est}.%

\end{proof} 

Regarding the chemical potential $F[\nz ,\cz ]$ \eqref{ChemPot}, the following chemical potential minimization lemma is central in the analysis. The idea behind the lemma is that we can replace the parabolic-parabolic chemical potential by the parabolic-elliptic potential.

\begin{lem}
Consider a function $f\in L^2(\Torus^2)$ \myc{(Check?)} and the solution $e_f$ to the following equation on $\Torus^2$:
\begin{align}
-\de e_f=f-\overline{f},\quad \int_{\Torus^2} e_{f}dY=0.
\end{align}
Then the solution has the following integral representation:
\begin{align}\label{defn_G_T2}
e_f(\by)=(-\de_{\Torus^2})^{-1}(f-\overline{f})(\by)=:\int {G}_{\Torus^2}(Y,Z)(f(Z)-\overline{f})dZ,
\end{align}
where $G_{\Torus^2}$ denotes the Green's function of the negative Laplacian on $ \mathbb{T}^2$. Assume that $e$ is an arbitrary $H^1$-function on the torus \myc{(Check?)}. 
Then the following inequality holds: 
\begin{align}
P[f;e]-P[f;e_f]=\frac{1}{2}\int_{\Torus^2} |\na_Y   (e-e_f)|^2dY\geq0,\label{ChemEngMin}
\end{align}  
\end{lem} 
\ifx
\begin{rmk}
 I think there is an easier way to derive the $\widetilde{B}(y,z)$. First, by the Neumann boundary condition, we see that we can extend the $n$ and $c$ evenly with respect to $\pm 1$. Now due to the even symmetry is preserved in the dynamics, we have that the $\pa_{y }$ is always zero at the two boundaries. Now we can further extend the $n,c$ to the external of the interval periodically with period $4$. Now the $\widetilde{B}$ is the $B$ on the torus of size $4\mathbb{T}\times \mathbb{T}$. Now we could see that $\widetilde{B}(y,z)$ is symmetric in $y,z$.  \end{rmk}
\fi
\begin{rmk}
The function $f$ will be chosen as a multiple of the cell density $n$. 
\end{rmk}
\begin{proof}
\myc{\footnote{Check whether we need this previous notes (?): The technical detail here is that one needs to prove the $L^2$ boundedness of $\na e_f$. The explicit justification comes from the bootstrap hypothesis/argument \eqref{Hyp_zero_mode_n_H1}? The formal proof is as follows:}}
Since $-\de{e_f}=f-\overline{f}$, we obtain the following equality through integration by parts,
\begin{align}\label{nacnabarc}
\int_{\Dy}\na_Y   e\cdot \na_Y   e_fdY=&-\int_{\Dy} e\de_Y  e_f dY=\int_{\Dy} e(f-\overline{f})dY.
\end{align}
Direct application of the above relation with $e=e_f$ yields that 
\begin{align}\label{P_relation}
P[f;e_f]=&\frac{1}{2}\int |\na_Y   e_f|^2dY-\int  e_f (f-\overline{f})dY
=-\frac{1}{2}\int|\na_Y   e_f|^2dY.
\end{align}
We calculate with the equalities \eqref{nacnabarc} and \eqref{P_relation} to obtain that
\begin{align}
P[f;e]-P[f;e_f]=&\frac{1}{2}\int|\na_Y   e|^2dY-\int  e (f- \overline{f})dY+\frac{1}{2}\int|\na_Y   e_f|^2dY\\
=&\frac{1}{2}\int|\na_Y   e|^2dY-\int \na_Y   e\cdot \na_Y   e_f dY+\frac{1}{2}\int|\na_Y   e_f|^2dY\\
=&\frac{1}{2}\int |\na_Y   e -\na_Y   e_f|^2dY\geq 0.
\end{align}
\ifx
\begin{align*}
\frac{1}{2}&\int|\na e-\na e_f|^2dY+\frac{1}{2}\int(e-e_f)^2dY\\=&\frac{1}{2}\int|\na e|^2dY-\int \na e\cdot \na e_fdY+\frac{1}{2}\int|\na e_f{}|^2dY
+\frac{1}{2}\int e^2dY-\int e e_{f}dY+\frac{1}{2}\int e_f^2dY\\
=&\frac{1}{2}\int|\na e|^2dY-\int  e (f- e_f)dY+\frac{1}{2}\int|\na e_f|^2dY+\frac{1}{2}\int e^2dY-\int e e_fdY+\frac{1}{2}\int e_f^2dY\\
=&F[f;e]+\frac{1}{2}\int |\na e_f|^2+\frac{1}{2}\int e_{f}^2\\
=&F[f;e]-F[f;e_f].
\end{align*}
\fi
This finishes the proof.
\end{proof}

Before proving the entropy bound on the solutions, we observe the following logarithmic Hardy-Littlewood-Sobolev inequality:

\begin{lem}\label{lem:log-HLS}
For all non-negative functions $f\in L_+^1(\mathbb{T}^2)$ such that  $\int f(Y)\log f(Y)dY$ is finite, the following inequality holds true:
\begin{align}\label{log-HLS}
\iint_{\Torus^2\times \Torus^2}(f(Y)-\overline{f})G_{\Torus^2}(Y,Z)(f(Z)-\overline{f})dYdZ\leq \frac{ {\|f\|_1}}{4\pi}\int_{\Torus^2} f(Y) \log f(Y)dY+C(\|f\|_1). 
\end{align}
Here $G_{\Torus^2}$ is the Green's function of the negative Laplacian on $ \mathbb{T}^2$ \eqref{defn_G_T2}.
\end{lem} 
To explicitly derive Lemma \ref{lem:log-HLS}, we need the following logarithmic Hardy-Littlewood-Sobolev inequality on a compact manifold and a comparison estimate.  
\begin{theorem}[Logarithmic Hardy-Littlewood-Sobolev inequality, \cite{SW}]
Let $\mathcal{M}$ be a two-dimensional  compact closed  Riemannian manifold. For all $\widetilde{M} > 0$, there exists a constant $C(\widetilde{M})<\infty$ such that for all non-negative functions $f \in L_+^1(\mathcal{M})$ subject to constraints $f \log f \in L^1$ and $\int_{\mathcal{M}} f  =\widetilde M$, the following estimate holds
\begin{align}\label{log HLS T2}
\int_{\mathcal{M}}f \log fdY+\frac{2}{\|f\|_{L^1(\mathcal{M})}}\iint_{\mathcal{M}\times\mathcal{M}} f(Z) f(\by) \log d(\by,Z) dYdZ\geq -C(\|f\|_{L^1(\mathcal{M})}),
\end{align}
where $d(Y,Z)$ is the distance on the Riemannian manifold $\mathcal{M}$.
\end{theorem}

The comparison estimate is included in the following lemma. 
\begin{lem}\label{lemFreEn_T2}
Consider $f\in L^1_+(\Torus^2)$ with average $\overline{f}$, subject to constraint $f\log f\in L^1(\Torus^2)$. There exists a constant $0<C<\infty$, such that the following estimate holds
\begin{align}
- \iint_{\mathbb{T}^2\times \Torus^2}(f(Y)-\overline{f})G_{\mathbb{T}^2}(Y,Z)(f(Z)-\overline{f}) dYdZ
&\geq\frac{1}{2\pi}\iint_{\mathbb{T}^2\times\mathbb{T}^2}f(Y)f(Z)\log d(Y,Z) dYdZ-C\|f\|_{L^1(\Torus^2)}^2. 
\end{align}
Here $G_{\Torus^2}$ is the Green's function of the negative Laplacian on $ \mathbb{T}^2$ \eqref{defn_G_T2}.
\end{lem}

\begin{proof}
The proof of this lemma appears in the paper \cite{BedrossianHe16}, and Lemma 3.2 in \cite{GongHe20}. 
Hence we omit the details for the sake of brevity.\ifx 
\myb{
Let $Y \in \Torus^2$ be fixed. Define the cut-off function $\varphi_Y(Z)\in C^\infty$ such that
\begin{align}
supp(\varphi_Y)=&B(Y,1/4),\\
\varphi_Y(Z)\equiv& 1,\forall Z\in B(Y,1/8),\\
supp(\na \varphi_Y(Z))\subset& \overline{B}(Y,1/4)\backslash B(Y,1/8).
\end{align}
By extending $f(Z)$ and $e_f(Z)$ periodically to $\rr^2$, we can rewrite the equation $-\de e_f=f-\overline{f}$ on $\mathbb{T}^2$ such that it is posed on $\rr^2$:
\begin{align}
-\de_Z (\varphi_Y(Z)e_f(Z))=(f(Z)-\overline{f})\varphi_Y(Z)-2\na_Z \varphi_Y(Z)\cdot\na_Z e_f(Z)-\de_Z\varphi_Y(Z)e_f(Z).
\end{align}
Using the fundamental solution of the Laplacian on $\rr^2$:
\begin{align}
e_f(Y)=&e_f(Y)\varphi_Y(Y)\\
=&-\frac{1}{2\pi}\int_{\rr^2}\log|Y-Z|\bigg((f(Z)-\overline{f})\varphi_Y(Z)-2\na_Z \varphi_Y(Z)\cdot\na_Z e_f(Z)-\de_Z\varphi_Y(Z)e_f(Z)\bigg)dZ\\
=&-\frac{1}{2\pi}\int_{|Y-Z|\leq \frac{1}{4}}\log|Y-Z|(f(Z)-\overline{f})\varphi_Y(Z)dZ-\frac{1}{\pi}\int_{|Y-Z|\leq \frac{1}{4}}\na_Z\cdot(\log|Y-Z|\na_Z \varphi_Y(Z)) e_f(Z)dZ\\
&+\frac{1}{2\pi}\int_{|Y-Z|\leq \frac{1}{4}}\log|Y-Z|\de_Z\varphi_Y(Z)e_f(Z)dZ.
\end{align}
Due to the support of $\varphi_Y$, we can identify the above with an analogous integral on $\mathbb{T}^2$ with $\abs{Y-Z}$ replaced by $d(Y,Z)$.
Therefore, we have the following estimate on the interaction energy,
\begin{align*}
-&\frac{1}{2}\int\displaylimits_{\mathbb{T}^2}(f(Y)-\overline{f})e_f(Y) dY\\
=&\frac{1}{4\pi}\iint\displaylimits_{\substack{\mathbb{T}^2\times\mathbb{T}^2\\d(Y,Z)\leq \frac{1}{4}}}\log d(Y,Z)(f(Y)-\overline{f})(f(Z)-\overline{f})\varphi_Y(Z)dZdY+\frac{1}{2\pi}\iint\displaylimits_{\substack{\mathbb{T}^2\times\mathbb{T}^2\\ \frac{1}{8}\leq d(Y,Z)\leq \frac{1}{4}}}(f(Y)-\overline{f})\na_Z\cdot(\log d(Y,Z)\na_Z \varphi_Y(Z)) e_f(Z)dZdY\\
&-\frac{1}{4\pi}\iint\displaylimits_{\substack{\mathbb{T}^2\times\mathbb{T}^2\\\frac{1}{8}\leq d(Y,Z)\leq \frac{1}{4}}}(f(Y)-\overline{f})\log d(Y,Z)\de_Z\varphi_Y(Z)e_f(Z)dZdY\\
=&\frac{1}{4\pi}\iint\displaylimits_{d(Y,Z)\leq \frac{1}{8}}\log d(Y,Z)(f(Y)-\overline{f})(f(Z)-\overline{f})dZdY+\frac{1}{4\pi}\iint\displaylimits_{\frac{1}{8}\leq d(Y,Z)\leq \frac{1}{4}}\log d(Y,Z)(f(Y)-\overline{f})(f(Z)-\overline{f})\varphi_Y(Z)dZdY\\
&+\frac{1}{2\pi}\iint\displaylimits_{\frac{1}{8}\leq d(Y,Z)\leq \frac{1}{4}}(f(Y)-\overline{f})\na_Z\cdot(\log d(Y,Z)\na_Z \varphi_Y(Z)) e_f(Z)dZdY-\frac{1}{4\pi}\iint\displaylimits_{\frac{1}{8}\leq d(Y,Z)\leq \frac{1}{4}}(f(Y)-\overline{f})\log d(Y,Z)\de_Z\varphi_Y(Z)e_f(Z)dZdY\\
=&\frac{1}{4\pi}\iint\displaylimits_{\mathbb{T}^2\times\mathbb{T}^2}\log d(Y,Z)f(Y)f(Z)dZdY-\frac{1}{4\pi}\iint\displaylimits_{d(Y,Z)> \frac{1}{8}}\log d(Y,Z)f(Y)f(Z)dZdY\\
&-\frac{1}{2\pi}\overline{f}\iint\displaylimits_{d(Y,Z)\leq \frac{1}{8}}\log d(Y,Z)f(Y)dZdY+\frac{1}{4\pi}\overline{f}^2\iint\displaylimits_{d(Y,Z)\leq \frac{1}{8}}\log d(Y,Z)dZdY\\
&+\frac{1}{4\pi}\iint\displaylimits_{\frac{1}{8}\leq d(Y,Z)\leq \frac{1}{4}}\log d(Y,Z)(f(Y)-\overline{f})(f(Z)-\overline{f})\varphi_Y(Z)dZdY\\
&+\frac{1}{2\pi}\iint\displaylimits_{\frac{1}{8}\leq d(Y,Z)\leq \frac{1}{4}}(f(Y)-\overline{f})\na_Z\cdot(\log d(Y,Z)\na_Z \varphi_Y(Z)) e_f(Z)dZdY-\frac{1}{4\pi}\iint\displaylimits_{\frac{1}{8}\leq d(Y,Z)\leq \frac{1}{4}}(f(Y)-\overline{f})\log d(Y,Z)\de_Z\varphi_Y(Z)e_f(Z)dZdY.
\end{align*}
The 2nd, 3rd, 4th, 5th terms in the last line are bounded below by $-C\|f\|_1^2$ for some constant $C > 0$.
The 6th and 7th terms are bounded below by $-C \|f\|_1 \|e_f\|_{L^1}$ for some constant $C > 0$, using the fact that $\na_Z\cdot(\log|Y-Z|\na_Z \varphi_Y(Z))$ and $\log|Y-Z|\de_Z\varphi_Y(Z)$ are bounded in the region $\frac{1}{8}\leq|Y-Z|\leq \frac{1}{4}$.
Denoting $K(Z)$ to be the fundamental solution of the Laplacian on $\mathbb{T}^2$, by Young's inequality, we have
\begin{align}
\|e_f\|_{L^1(\mathbb{T}^2)}=&\|K\ast (f-\overline{f})\|_{L^{1}(\mathbb{T}^2)}\leq\|K\|_{L^1(\mathbb{T}^2)}\|f-\overline{f}\|_{L^{1}(\mathbb{T}^2)}\leq C \|f\|_{L^1(\Torus^2)}.
\end{align}}\fi
\end{proof}

\begin{proof}[Proof of Lemma \ref{lem:log-HLS}] \ifx To prove the estimate \eqref{log-HLS}, we apply a gluing technique. 
First we observe that the torus $2\mathbb{T}\times \mathbb{T}$ can be decomposed into two parts $\mathcal{D}_\by^+$ and $\mathcal{D}_\by^-$, each one of them is a copy of $\mathcal{D}_\by$. Note that we have shift the coordinate so that the domain $\mathcal{D}_\by$ is identical to $\mathcal{D}_\by^-$. Recall that the function $e_f$ can be represented as 
$e_f=\int_{\mathcal{D}_\by} G_M(Y,Z) (f(Z)-\overline{f})dZ$.  Note that the even extension $\wt f$ is symmetric about the $Y $ axis. Now by Lemma \ref{lem:equivalence}, we see that  to derive the estimate \eqref{log-HLS}
\begin{align}
\frac{2\mf M_0}{4\pi}\int_{\mathcal{D}_\by} f \log f dY-\int_{\mathcal{D}_\by\times \mathcal{D}_{\by}}(f(Z)-\overline{f})G_M(Z,y)(f(y)-\overline{f})dYdZ\geq -C(M),
\end{align} 
it is enough to derive the following estimate for the even reflection $\wt f$ and $\wt e_f=\int G_{2\mathbb{T}\times \mathbb{T}}(y,Z)(\wt f(Z)-\overline{f})dZ$
\begin{align}
\frac{2\mf M_0}{4\pi}&\int_{\mathcal{D}_\by^+} f \log fdY-\int_{\mathcal{D}_\by^+}(\wt f(y)-\overline{f})e_{\wt f}(y)dY\\
=&\frac{2\mf M_0}{4\pi}\int_{\mathcal{D}_\by^+} f \log fdY-\int_{\mathcal{D}_\by^+\times (2\mathbb{T}\times\mathbb{T})}(\wt f(y)-\overline{f})G_{2\mathbb{T}\times\mathbb{T}}(Z,y)
(\wt f(Z)-\overline{f})dYdZ\\
\geq& -C(\mf M_0). 
\end{align}
Applying the even symmetry of $\wt f$ and $\wt e_f$ 
with respect to the $y $-axis, we have that the inequality is equivalent to 
\begin{align}
\frac{2\mf M_0}{4\pi}&\int_{\mathcal{D}_\by^-} f \log fdY-\int_{\mathcal{D}_\by^-}(\wt f(y)-\overline{f})e_{\wt f}(y)dY\\
=&\frac{2\mf M_0}{4\pi}\int_{\mathcal{D}_\by^-} f \log fdY-\int_{\mathcal{D}_\by^-\times (2\mathbb{T}\times\mathbb{T})}(\wt f(y)-\overline{f})G_{2\mathbb{T}\times\mathbb{T}}(Z,y)
(\wt f(Z)-\overline{f})dYdZ\\
\geq& -C(\mf M_0). 
\end{align}\fi
Combining the logarithmic HLS inequality \eqref{log HLS T2} and Lemma \ref{lemFreEn_T2} yields that 
\begin{align}
\frac{\|f\|_{L^1(\mathbb{T}^2)}}{4\pi}&\int_{\mathbb{T}^2} f \log fdY-\iint_{\mathbb{T}^2\times \mathbb{T}^{2}}(f(Y)-\overline{f})G_{\mathbb{T}^2}(Y,Z)(f(Z)-\overline{f})dYdZ\\
\geq &\frac{\| f\|_{L^1(\mathbb{T}^2)}}{4\pi}\int_{\mathbb{T}^2} f \log fdY +\frac{1}{2\pi}\iint_{\mathbb{T}^2\times\mathbb{T}^{2}}f(Y)f(Z)\log d(Y,Z)dYdZ-C\|f\|_{L^1(\Torus^2)}^2\\
\geq& -C(\|f\|_{L^1(\Torus^2)})-C\|f\|_{L^1(\Torus^2)}^2\geq -C(\|f\|_{L^1(\Torus^2)}).
\end{align}
This is \eqref{log-HLS}.
\end{proof}

\begin{lem}
Consider the solutions $\nz ,\,\cz $ to the $x$-averaged equations  \eqref{ppPKS_0md}. Assume that $\Mass_0=\|\nz \|_{L_Y^1(\Torus^2)}<8\pi$, and the magnitude $A$ is chosen large enough compared to the constants in Hypotheses \eqref{Hypotheses}. Then the following estimate holds
\begin{align}\label{EntrBound}\\
\int_{\Torus^2} \nz (t,Y)\log^+ \nz (t,Y)dY+\int_{\Torus^2} |\na \cz (t,Y)|^2dY \leq C_{L\log L}(\mf M_0, E[\lan n_{\mathrm{in}}\ran,\lan \cc_{\mathrm{in}}\ran])<\infty,\quad \forall t\in [0,T_\star].
\end{align}
Here $\log^+(\cdot)=\max\{0,\log (\cdot)\}$ is the positive part of the logarithmic function. 
\end{lem}
\begin{proof}
Now we apply the chemical potential comparison \eqref{ChemEngMin} 
by setting the function $f$ to be $\frac{\nz }{1-\zeta}$, $e_f=e_{\frac{\nz }{1-\zeta}}$, and the function $e$ to be $c$. Here $0<\zeta<1$ is a small constant to be chosen later. Now we have the following relation by applying \eqref{P_relation}
\begin{equation}\label{4.5}
P\left[\frac{\nz }{1-\zeta},\cz \right]\geq P\left[\frac{\nz }{1-\zeta},e_{\frac{\nz }{1-\zeta}}\right]=-\frac{1}{2(1-\zeta)}\int (\nz (t,Y)-\overline{n})e_{\frac{\nz }{1-\zeta}}(t,Y)dY.
\end{equation} 
Combining the definitions of $E$ \eqref{Free_energy}, $P$ \eqref{ChemPot}, the lower bound \eqref{4.5}, the logarithmic  Hardy-Littlewood-Sobolev inequality on $\Dy$ \eqref{log-HLS}, the free energy bound \eqref{frentm_est}, and  taking $A$ large enough compared to the constants in the Hypotheses \eqref{Hypotheses}, we have that for $t\in[0,T_\star]$, the following relation holds
\begin{align}\label{Elwbnd}\\
E[ &\lan n_{\mathrm{in}}\ran , \lan \cc_{\mathrm{in}}\ran ]+1\\
\geq &E[\nz (t),\cz (t)]\\
\geq&\int \nz \log \nz dY+(1-\zeta)P\left[\frac{\nz }{1-\zeta};\cz \right]+\frac{1}{2}\zeta\int |\na \cz |^2 dY\nonumber\\
\geq&\int \nz \log \nz dY-\frac{1}{2(1-\zeta)}\iint_{\Dy\times \Dy} (\nz (Y)-\overline{n})G_{\Torus^2}(Y,Z)(\nz (Z)-\overline{n})dYdZ +\frac{1}{2}\zeta\int |\na \cz |^2dY\nonumber\\
\geq&\left(1-\frac{\mf M_0}{8\pi(1-\zeta)}\right)\int \nz \log \nz dY\\&+\frac{\mf M_0}{8\pi(1-\zeta)}\left(\int \nz \log \nz  dY-\frac{4\pi}{\mf M_0}\iint_{\Dy\times\Dy} (\nz (Y)-\overline{n})G_{\Dy}(Y,Z)(\nz (Z)-\overline{n}) dYdZ\right)\\
& +\frac{1}{2}\zeta\int |\na \cz |^2dY\\
\geq&\left(1-\frac{\mf M_0}{8\pi(1-\zeta)}\right)\int \nz \log \nz dY-\frac{1}{1-\zeta}C(\mf M_0) +\frac{1}{2}\zeta\int |\na \cz |^2dY.
\end{align}
Since $\mf M_0<8\pi$, there exists a constant $\zeta=\zeta(\mf M_0)\in(0,1)$ such that $1-\zeta> \frac{\mf M_0}{8\pi}$.  We fix such a $\zeta$ and reorganize the terms to  obtain that 
\begin{align}
\int \nz (t,Y)\log \nz (t,Y)dY +\int |\na \cz (t,Y)|^2dY\leq C_{L\log L}(\mf M_0,\zeta, E[\lan n_{\mathrm{in}}\ran, \lan\cc_{\mathrm{in}}\ran])<\infty,\quad \forall t\in [0,T_\star].  
\end{align}
Since the domain $\Dy$ is bounded and $\nz \log^-\nz $ is bounded from below, the negative part of the entropy $\int n\log^- ndx$ is bounded. Combining this and the fact that $\zeta$ depends only on $\mf M_0<8\pi$, we have that:
\begin{align}
\int \nz (t,Y)\log^+ \nz (t,Y) dY + \int |\na \cz (t,Y)|^2dY\leq C_{L\log L}(\mf M_0, E[\lan n_{\mathrm{in}}\ran, \lan\cc_{\mathrm{in}}\ran])<\infty,\quad \forall t\in [0,T_\star).
\end{align}
\ifx
Since we consider the bounded domain $\Dy$ and assume that the average of $\cz $ is zero \eqref{Chem_cnsrv}, Poincar\'e inequality yields that
\begin{align}
\|\cz (t,\cdot)\|_{L^2(\Dy)}\leq C_P \|\na_Y   \cz (t,\cdot)\|_{L^2(\Dy)}\leq C_{L\log L}<\infty.
\end{align} \fi
This concludes the proof of the lemma.
\end{proof}

\begin{lem}
Assume that the entropy and chemical gradient bound \eqref{EntrBound} holds, then the following estimate of the $L^2 $ norm $\|n\|_2$ holds
\begin{align}
\|\nz(t)\|_{L^2}\leq C_{\nz ;L^2}(\|\lan n_{\mathrm{in}}\ran\|_2, C_{L\log L}(\mf M_0, E[\lan n_{\mathrm{in}}\ran, \lan\cc_{\mathrm{in}}\ran]) )<\infty,\quad \forall t\in [0,T_\star].\label{L2case}
\end{align}

\end{lem} 
\begin{proof}Before proving the $L^2$-bound, we estimate the $\|\pa_t \cz \|_{L_t^2 L_Y^2}.$ Since the negative part of the entropy is bounded, we have that $E[\nz (t),\cz (t)]$ is uniformly bounded from below from \eqref{Elwbnd}. As a result, we have from \eqref{frentm_est} that
\begin{align}\label{PtcTmIntBnd}
A\int_{0}^{T_\star}\|\pa_t \cz \|_2^2dt\leq E[\lan n_{\mathrm{in}}\ran,\lan \cc_{\mathrm{in}}\ran]-\inf_{t\in[0,T_\star]} E[\nz (t),\cz (t)]=:C_{\pa_t \cc}(\mf M_0, E[\lan n_{\mathrm{in}}\ran,\lan \cc_{\mathrm{in}}\ran])<\infty.
\end{align}

We start  the $L^2$ estimate. It is classical to consider the  vertically truncated $L^2$ norm $\|(\nz -K)_+\|_2$ for some positive $K\geq \{1, \overline{n}\}$. The entropy bound \eqref{EntrBound} yields that the $L^1$ norm of the function $(\nz -K)_+$ is small, i.e.,
\begin{align}
\|(\nz -K)_+\|_1\leq \int_{\Torus^2}(\nz -K)_+\frac{\log \nz }{\log K}dY\leq \frac{C_{L\log L}}{\log K}=:\eta_K\ll 1.\label{Chc_eta_K}
\end{align}
The time evolution of the $L^2$-norm, i.e., $\|(\nz -K)_+\|_2^2$, can be estimated with the $x$-averaged  equations of $(\nz ,\cz )$ \eqref{ppPKS_0md}: 
\begin{align}\label{dt(n_0-K)+2}
\frac{d}{dt}&\int(\nz -K)_+^2dY\\
=&-\frac{2}{A}\int |\na(\nz -K)_+|^2dY+\frac{1}{2A}\int(\nz -K)_+^2(\nz -\overline{n})dY-\frac{1}{2A}\int (\nz -K)_+^2(A\pa_t\cz )dY\\
&+\frac{2K}{A}\int(\nz -K)_+(\nz -\overline{n})dY -\frac{2K}{A}\int (\nz -K)_+(A\pa_t\cz )dY+\frac{2}{A}\int \na_Y   (\nz -K)_+\cdot \lan n_{\neq}\na_Y   \cc_{\neq}\ran dY\\
=&:-\frac{2}{A}\int |\na(\nz -K)_+|^2dY+\sum_{\ell=1}^5T_{\ell}.
\end{align}
Applying the Gagliardo-Nirenberg inequality and recalling the choice of $\eta_K$ ($K\geq \max\{1,\overline{n}\}$) \eqref{Chc_eta_K}, we obtain that the $T_{1}$ term is bounded above as follows
\begin{align}\label{T_L2_1}\\
T_{1}=&\frac{1}{2A}\|(\nz -K)_+\|_{L^3(\Torus^2)}^3+\frac{K-\overline{n}}{2A}\norm{ (\nz -K)_+}_{L^2(\Torus^2)}^2 \\
\leq& \frac{C }{A} (\|\na (\nz -K)_+\|_{L^2(\Torus^2)}^2\|(\nz -K)_+\|_{L^1(\Torus^2)}+\|(\nz -K)_+\|_{L^1(\Torus^2)}^3)+\frac{K}{2A}\norm{ (\nz -K)_+}_{L^2(\Torus^2)}^2\\
\leq& \frac{C}{A}\eta_K\|\na (\nz -K)_+\|_{L^2(\Torus^2)}^2+\frac{C}{A}\eta_K^3+\frac{K}{2A}\norm{ (\nz -K)_+}_{L^2(\Torus^2)}^2.
\end{align}
By choosing $K$ large enough in \eqref{Chc_eta_K}, we can make $\eta_K$ small enough such that the leading order term is absorbed by the dissipation. Next we estimate the $T_{2}$ term in \eqref{dt(n_0-K)+2} with the Gagliardo-Nirenberg inequality, H\"older inequality and Young inequality
\begin{align}\label{T_L2_2}
T_{2}\leq&\frac{1}{A}\|(\nz -K)_+\|_{L^4(\Torus^2)}^2\|A\pa_t \cz \|_{L^2(\Torus^2)}\\
\leq&\frac{C}{A}\left(\|\na (\nz -K)_+\|_{L^2(\Torus^2)}\|(\nz -K)_+\|_{L^2(\Torus^2)}+\|(\nz -K)_+\|_{L^1(\Torus^2)}^2\right)\|A\pa_t \cz \|_{L^2(\Torus^2)}\\
\leq& \frac{1}{8A}\|\na(\nz -K)_+\|_{L^2(\Torus^2)}^2+\frac{C}{A}\|(\nz -K)_+\|_{L^2(\Torus^2)}^2\|A\pa_t \cz \|_{L^2(\Torus^2)}^2+\frac{C}{A}\eta_K^2\|A\pa_t \cz \|_{L^2(\Torus^2)}.
\end{align}
For the $T_{3}$ and $T_{4}$ term in \eqref{dt(n_0-K)+2}, we apply the H\"older inequality and the Young's inequality to obtain that
\begin{align}
T_{3}+T_{4}\leq \frac{C(K)}{A}\|(\nz -K)_+\|_{L^2(\Torus)}^2+\frac{C(K)}{A}\|A\pa_t \cz \|_{L^2(\Torus^2)}^2+\frac{C(K)}{A}.\label{T_L2_34}
\end{align}
For the last term in \eqref{dt(n_0-K)+2}, we apply the Hypotheses \eqref{HypED_FM}, \eqref{Bound_Linf} to estimate it as follows
\begin{align}T_{5}\leq&\frac{1}{A}\|\na_Y   (\nz -K)_+\|_{L^2_Y}\|n_{\neq}\|_{L_{x,Y}^\infty}\|\na_Y   \cc_{\neq}\|_{L_{x,Y}^2} 
\leq\frac{1}{8A}\|\na_Y   (\nz -K)_+\|_{L^2_Y}^2+\frac{C}{A}\CC^4e^{-2\delta \frac{t}{A^{1/3}}}.\label{T_L2_5}
\end{align}
Combining the aforementioned estimates \eqref{T_L2_1}, \eqref{T_L2_2}, \eqref{T_L2_34},  \eqref{T_L2_5}, the equation \eqref{dt(n_0-K)+2}, and applying the Nash inequality on the dissipation term, we obtain that
\begin{align}\label{dt(n-K)+2_s2}
\frac{d}{dt}\int(\nz -K)_+^2dY
\leq& -\frac{1}{4AC \mf M_0^2}\left(\int(\nz -K)_+^2dY\right)^2+\frac{1}{A} {C}(K)(1+A^2\|\pa_t \cz  \|_{L_Y^2}^2)\norm{(\nz -K)_+}_{L_Y^2}^2\\
&+\frac{ {C}(K)}{A}(1+A^2\|\pa_t \cz \|_{L_Y^2}^2)+\frac{C}{A}\CC^4e^{-2\delta \frac{t}{A^{1/3} }}.
\end{align}
To simplify notations, we define $\dss \mf Y(t):=\int(\nz (t,Y)-K)_+^2dY$, 
\begin{align}
G_1(t):=&\frac{{C}}{A}\CC^4\int_0^t e^{-2\delta \frac{s}{A^{1/3} }}ds.
\end{align}
By choosing $A$ large enough compared to the constants in the hypotheses \eqref{Hypotheses}, we make $G_1(t)$ smaller than $1/2$, i.e., \begin{align}
G_1(t)\leq & \frac{{C}\CC^4 }{\delta A^{2/3}}\leq \frac{1}{2},\quad \forall t\in[0,T_\star]. 
\end{align}
Then the quantity $\mf Y$ fulfills the following relation
\begin{align}\label{dfn_f}
\frac{d}{dt}\mf Y(t)\leq -\frac{1}{AC_1(\mf M_0)} \mf Y^2(t)+\frac{1}{A}f(t) \mf Y(t)+\frac{1}{A}f(t)+G_1'(t),\quad f(t):={C}_2(K)(1+A^2\|\pa_t \cz \|_2^2).
\end{align}
Since $G_1(t)\leq 1$, we have that
\begin{align}
\frac{d}{dt}(\mf Y(t)-G_1(t))\leq -\frac{1}{AC_1(\mf M_0)} (\mf Y(t)-G_1(t))_+^2+\frac{1}{A}f(t) (\mf Y(t)-G_1(t))+\frac{3}{2A}f(t),\quad \forall t\in[0,T_\star]. 
\end{align}
If $\mf Y(t)\leq G_1(t)\leq 1$, then the $L^2$-bound of $(\nz -K)_+$ is straightforward. \myr{Without loss of generality, we  assume that $\mf Y(t)> G_1(t)$ on the whole time interval $[0,T_\star]$. The proof in the general case is similar. }\myc{\footnote{An explicit argument will be the following. Considering the set of time instances at which $\mf Y(t)> G_1(t)$. The set is open on $\rr$, so it is a countable  collection of disjoint open intervals $\cup_{i=1}^\infty(a_i,b_i)$. On these open intervals, we  apply the comparison argument here. }}
By comparison, we observe that the quantity $\mf Y(t)-G_1(t)$ is bounded above by the solution to the following differential equality 
\begin{align*}
\frac{d}{dt}\mf  Z(t)=&-\frac{1}{C_1(\mf M_0)A}\mf Z(t)^2+\frac{f(t)}{A}\mf Z(t)+\frac{2f(t)}{A},\,\quad
 \mf Z(0)=\max\{2, 2\|(\lan n_{\mathrm{in}}\ran-K)_+\|_2^2, C_1(\mf M_0) C_2(K)+1\}.
\end{align*}
Since $C_1(\mf M_0)C_{2}(K)$ is a stationary solution to the differential equation $\frac{d}{dt}{q}=-\frac{1}{AC_1}q^2+\frac{1}{A}{C}_2q$, by comparison, we have that $\mf Z(t)> {C_1}C_2 $. Recalling the definition \eqref{dfn_f}, we  get an upper bound for $\frac{d}{dt}\mf Z(t)$:
\begin{align*} 
\frac{d}{dt}\mf Z(t)\leq &-\frac{1}{AC_1} \mf Z(t)^2+\frac{1}{A}h(t)\mf Z(t),\quad h(t)=
f(t)\left(2+\frac {2}{C_1(\mf M_0)C_2(K)}\right)=C_3(\mf M_0,K)(1+A^2\|\pa_t \cz \|_2^2). 
\end{align*}
Solving this differential inequality, we have that by \eqref{PtcTmIntBnd}
\begin{align*}
\mf Z^{-1}(t)\geq &\mf Z^{-1}(0)\exp\left\{-\frac{1}{A}C_3t-C_3A\int_0^t \|\pa_t \cz \|_2^2ds\right\}\mathbf{1}_{t\leq \min\{T_\star,A\}}\\
 &+\frac{1}{C_1{C_3}} \exp\left\{-C_3A \int_0^t  \|\pa_t \cz \|_2^2ds\right\}\left(1-\exp\left\{-\frac{C_3}{A}t\right\}\right)\mathbf{1}_{t\geq \min\{T_\star,A\}}\\
 \geq &[C(\|\lan n_{\mathrm{in}}\ran\|_{L^2},C_1, C_2, C_3,C_{\pa_t \cc})]^{-1}>0,\quad \forall t\in[0,T_\star],
\end{align*} 
which yields the $L^2 $-bound for $\forall t\in [0,T_\star]$ as follows:
\begin{align*}
\|\nz (t)\|_2^2\leq& C(\|(\nz (t)-K)_+\|_2^2+\mf M_0 K)\leq C(\mf Z(t)+G_1(t)+\mf M_0K)\\
=: &C_{\nz;L^2}^2(\mf M_0,\| \lan n_{\mathrm{in}}\ran \|_2, E[ \lan n_{\mathrm{in}}\ran ,  \lan \cc_{\mathrm{in}}\ran ])<\infty.\quad 
\end{align*}
This concludes the proof of the lemma.
\end{proof}
\begin{lem} \label{lem:Linf_n_c}
Assume  the entropy and chemical gradient bound \eqref{EntrBound} and the bootstrap hypotheses \eqref{Hypotheses}. If the magnitude $A$ is chosen large enough compared to the thresholds in the bootstrapping hypotheses, then the following $L^\infty$-norm estimates hold
\begin{align}
\|n \|_{L_t^\infty([0, T_\star]; L^\infty)}\leq& C_{n ;L^\infty}(\mf M_0, C_{\lan n\ran;L^2},\|   n_{\mathrm{in}}  \|_{L^\infty})<\infty;\label{Linf_n}\\
\|\na_Y   \cz \|_{L_t^\infty([0,T_\star];L_Y^\infty)}\leq& C_{\na \cz ;L^\infty}(\mf M_0,\| \lan n_{\mathrm{in}}\ran \|_{L_Y^4},\| \na_Y    \lan \cc_{\mathrm{in}}\ran \|_{  L_Y^\infty},E[ \lan n_{\mathrm{in}}\ran ,  \lan \cc_{\mathrm{in}}\ran ])<\infty.\label{Linf_na_cz} 
\end{align}
\end{lem} 
\begin{proof}\myb{
Thanks to the hypothesis \eqref{HypED_FM}, the estimate \eqref{L2case}, and the chemical gradient estimate \eqref{na_c_0_est}, we have that 
\begin{align}\label{nac_L6est}
\|\na \cc\|_{L^6(\Torus^3)}\leq& \|\na \cc_\nq\|_{L^6(\Torus^3)}+C\|\na_Y \cz\|_{L^6(\Torus^2)}\\
\leq& CC_{\nz;L^2} +C\|\na_Y\lan \cc_{\text{in}}\ran\|_{L^6}+\frac{1}{A^{1/3}}{}C(\|n_{\text{in};\nq}\|_{H^M}+\|\cc_{\text{in};\nq}\|_{H^{M+1}}).
\end{align} 
We choose $p=2^j, \quad j\in \mathbb{N}$, and implement the energy estimate with suitably chosen Gagliardo-Nirenberg inequality as follows
\begin{align*}
\frac{1}{2p}\frac{d}{dt}\|n\|_{2p}^{2p}\leq&-\frac{2p-1}{Ap^2}\|\na(n^{p})\|_2^2+\frac{2p-1}{Ap}\|\na(n^p)\|_2\|n^p \na \cc\|_2\\
\leq&-\frac{2p-1}{Ap^2}\|\na(n^{p})\|_2^2+\frac{2p-1}{Ap}\|\na(n^p)\|_2\|n^p\|_3 \| \na \cc\|_{6}\\
\leq&-\frac{2p-1}{2Ap^2}\|\na(n^{p})\|_2^2+\frac{C }{A}\lf(\|\na(n^p)\|_2^{\frac{3}{2 }}\|n^p\|_2^{\frac{1}{2}}+\|\na(n^p)\|_2\|n^p\|_{1}\rg)\| \na \cc\|_6\\
\leq& -\frac{1}{CAp}\|\na(n^{p})\|_2^2+\frac{Cp^3}{A}\|n^p\|_2^2\|\na \cc\|_6^4+\frac{Cp}{A}\|n^p\|_1^2\|\na \cc\|_6^2.
\end{align*}
Recall the Nash inequality on $\Torus^3$
\begin{align}
\|f\|_{L^2}\leq \|\na f\|_{L^2}^{\frac{3}{5}} \|f\|_{L^1}^{\frac{2}{5}}+\|f\|_{L^1},
\end{align}
which implies that 
\begin{align}
-\|\na f\|_{L^2}^2\leq -\frac{1}{C}\|f\|_{L^2}^{\frac{10}{3}}\|f\|_{L^1}^{-\frac{4}{3}}+C\|f\|_{L^1}^2.
\end{align}
\myc{\footnote{(Check the identically zero case? If we have $\|f\|_{L^1}=0$, then $\|f\|_{L^p}\equiv0$. And in the ODE, we don't care about these points? In this approach, we care about the extremal large value case. Well, $\|n^p\|_1\geq c\|n\|_1^p=cM^p>0$?)}}
Combining the Nash inequality and the energy estimate, we obtain
\begin{align}
\frac{1}{2p}\frac{d}{dt}\|n^p\|_{2}^{2}\leq&-\frac{1}{CAp\|n^p\|_1^{4/3}}\lf(\|n^p\|_{2}^{\frac{10}{3}}-Cp^4\|n^p\|_2^{2}\|n^p\|_1^{\frac{4}{3}}\|\na \cc\|_6^4-Cp^2\|n^p\|_1^{\frac{10}{3}}\|\na\cc\|_6^2\rg)\\
\leq &-\frac{1}{CAp\|n^p\|_1^{4/3}}\lf(\|n^p\|_{2}^{\frac{10}{3}}-C p^{10}\|n^p\|_1^{\frac{10}{3}}(\|\na \cc\|_6^2+\|\na\cc\|_6^{10})\rg).
\end{align}
Direct ODE argument yields that for $p=2^{j-1}\geq 4,\quad j\in\mathbb{N}$,
\begin{align}
\sup_{t\in[0,T_\star]}\|n(t)\|_{2^j}^{2^j}\leq \max\lf\{ C p^{6}
\sup_{t\in[0,T_\star]}\|n(t)\|_{2^{j-1}}^{2^j}(1+\|\na \cc(t)\|_{6}^{6}), \|n_{\text{in}}\|_{2^j}^{2^j}\rg\}.
\end{align}
If the second argument in the maximum is used for infinitely many $j$'s, we automatically get \eqref{Linf_n}. So assume that $j_0$ is the largest number such that the first term is dominated by the initial data $\|n_{\text{in}}\|_{2^j}^{2^j}$. 
By taking the logarithm on both side, we have that 
\begin{align}
\log \frac {\sup_{t\in[0,T_\star]}\|n(t)\|_{2^j}}{\sup_{t\in[0,T_\star]}\|n(t)\|_{2^{j-1}}}\leq  2^{-j}\lf(\log C+6j\log 2 +\sup_{t\in[0,T_\star]}\log(1+\|\na \cc(t)\|_{6}^{6})\rg), \, j> j_0.
\end{align}
By recalling the $\na \cc$ estimate \eqref{nac_L6est} and summing in $j$, we obtain the result \eqref{Linf_n}. 
}
The $\na \cz$-estimate \eqref{Linf_na_cz} follows from the $L^\infty$ estimation of $\nz $ and the estimate \eqref{na_c_0_est}. 

\ifx 
By the same type of energy estimate as in the $L^2$-case, we apply the Nash  inequality, the $L^2$-bound \eqref{L2case}, and the boostrap hypotheses \eqref{Hypotheses}  to obtain that \footnote{\textbf{Check!!} (Here I copy from \cite{CalvezCorrias} equation (5.12) )}
\begin{align}
\frac{d}{dt}\|\nz \|_4^4\leq& -\frac{1}{12C_{N}A(C_{\nz ;L^2}^\star)^{4}}\|\nz \|_4^{8}+\frac{{C(\mf M_0, C_{\nz ;L^2}^\star)}}{A}(1+A^2\|\pa_t \cz  \|_2^2)\|\nz \|_4^4\\
&+\frac{ {C(\mf M_0,C_{\nz ;L^2}^\star)}}{A}(1+A^2\|\pa_t \cz  \|_2^2)+\frac{\myr{C(\CC)??}}{A}e^{-\frac{\delta t}{A^{1/3}}}.
\end{align}
Now we define 
\begin{align}
G_2(t):=\int_0^t\frac{\myr{C(\CC)??}}{A}e^{-\frac{\delta s}{A^{1/3}}}ds,\quad\forall t\in[0,T_\star].
\end{align}
By taking $A$ large enough, we see that $G_2(t)\leq 1$. Now apply similar ODE calculation as in the proof of previous lemma, we obtain that $L^4$ norm of the cell density $\nz $ is bounded on the time interval $[0,T_\star]$. 
\fi

\end{proof}
\ifx
\begin{lem}
Assume the entropy and chemical gradient bound \eqref{EntrBound}, and the bootstrapping hypotheses \eqref{Hypotheses}. If the magnitude $A$ is chosen large enough, then the following estimate on the $\dot H_Y^1 $ norm of $\nz $ holds
\begin{align}
\|\na_Y   \nz (t)\|_{L_Y^2}\leq C^\star_{\nz ;\dot H^1}(\CC)<\infty,\quad \forall t\in [0,T_\star].
\end{align}
\end{lem} 
\begin{rmk}
By choosing the constant $C_{\nz ;\dot H^1}$ in the bootstrapping hypothesis \eqref{Hyp_zero_mode_na_c_Linf} large enough compared to $C_{\na \cz }$ in the lemma, we have obtained the improvement \eqref{Con_zero_mode_na_c_Linf}. 
\end{rmk}
\begin{proof}
Now we estimate the $\|\na_Y   \nz \|_2$. Estimating the time evolution of $\|\na_Y   \nz \|_2$ using Gagliardo-Nirenberg-Sobolev inequality, Young's inequality, Minkowski inequality, Lemma \ref{Lem:nacest} and the time integral estimate of $\|\na^2 \cc_{\neq}\|_2^2$ \eqref{Hyp1_c}, we have that 
\begin{align}
\frac{1}{2}\frac{d}{dt}\|\na_Y   \nz \|_2^2
\leq &-\frac{\|\na_Y  ^2 \nz \|_2^2}{2A}+\frac{4\|\de_Y \cz \|_2^2\|\nz \|_\infty^2}{A}+\frac{4\|\na_Y   \cz \|_\infty^2\|\na_Y   \nz \|_2^2}{A}+\frac{4}{A}\|\na_Y  \cdot\lan\na_Y   \cc_{\neq}n_{\neq}\ran \|_2^2\nonumber\\
\leq&-\frac{\|\na_Y  ^2 \nz \|_2^2}{2A}+\frac{4(A\|\pa_t \cz \|_2+\|\nz -\overline{n}\|_2 )^2\|\nz \|_\infty^2}{A}+\frac{C(C_{n,\infty})\|\na_Y   \nz \|_2^{2}}{A}\nonumber\\
&+\frac{4}{A}\|\na_Y  ^2 \cc_{\neq}\|_{L_{x,Y}^2}^2\|n_{\neq}\|_{L_{x,Y}^\infty}^2+\frac{4}{A}\|\na_Y   \cc_{\neq}\|_{L_{x,Y}^\infty}^2\|\na_Y   n_{\neq}\|_{L_{x,Y}^2}^2\nonumber\\
\leq&-\frac{\|\na_Y  ^2 \nz \|_2^2}{2A}+\frac{8\|\nz -\overline{n}\|_2^2\|\nz \|_\infty^2}{A}+\frac{C(C_{n,\infty})\|\na_Y   \nz \|_2^{2}}{A}\nonumber\\
&+\left(12A\|\pa_t \cz \|_2^2\|\nz \|_\infty^2+\frac{4}{A}\|\na_Y  ^2 \cc_{\neq}\|_{L_{x,Y}^2}^2\|n_{\neq}\|_{L_{x,Y}^\infty}^2+\frac{4}{A}\|\na_Y   \cc_{\neq}\|_{L_{x,Y}^\infty}^2\|\na_Y   n_{\neq}\|_{L_{x,Y}^2}^2\right).\label{time evolution of na f L2}
\end{align}
We define the function 
\begin{equation}\label{G3}
G_3(t):=12A\int_0^t\|\pa_t \cz \|_{L_Y^2}^2\|\nz \|_{L_Y^\infty}^2ds+\frac{4}{A}\int_0^{t}\|\na_Y  ^2 \cc_{\neq}\|_{L_{x,Y}^2}^2\|n_{\neq}\|_{L_{x,Y}^\infty}^2+\|\na_Y   \cc_{\neq}\|_{L_{x,Y}^\infty}^2\|\na_Y  n_{\neq}\|_{L_{x,Y}^2}^2ds.
\end{equation}
Note that by the hypotheses \eqref{HypED_L2}, \eqref{Hyp}, \eqref{Hyp_n_Linf}, \eqref{Hyp_na_c_neq_Linf}, we have that
\begin{align}
\frac{4}{A}\int_0^{T_\star}\|\na_Y ^2 \cc_{\neq}\|_2^2\|n_{\neq}\|_{\infty}^2+\|\na_Y   \cc_{\neq}\|_\infty^2\|\na_Y  n_{\neq}\|_{2}^2ds\leq CC_{2,\infty}^4
\end{align}
for all $t\leq T_\star$. Moreover, from the time integral control \eqref{PtcTmIntBnd} and the hypothesis \eqref{Hyp_n_Linf}, we have that
\begin{align*}
A\int_0^{T_\star}\|\pa_t \cz \|_2^2\|\nz \|_\infty^2ds\leq C(C_{n,\infty},\mf M_0,E[ \lan n_{\mathrm{in}}\ran , \lan \cc_{\mathrm{in}}\ran ])\leq C C_{2,\infty}^4.
\end{align*}
Now, we have that the function $G_3$ \eqref{G3} is bounded as follows
\begin{align}
G_3(t)\leq CC_{2,\infty}^4 , \quad \forall t\in[0,T_\star].
\end{align} 
Now we apply Gagliardo-Nirenberg 
inequality, Lemma \ref{Lem:nacest} and definition of $G_3(t)$ \eqref{G3} to rewrite the inequality \eqref{time evolution of na f L2} as follows:
\begin{align}
\frac{d}{dt}\left(\|\na_Y   \nz \|_2^2-G_3(t)\right)\leq &\frac{1}{A}\bigg(-\frac{(\|\na_Y   \nz \|_2^2-G_3)_+^2}{4CC_{\nz ,L^2}^2}
+C(C_{n,\infty})(\|\na_Y   \nz \|_2^2-G_3)+ C(C_{n;\infty})G_3+CC_{2,\infty}^{4}\bigg).\label{f 0 H1 0}
\end{align}
Now because $G_3(t)\leq CC_{2,\infty}^4$, by a similar  ODE argument as in the proof in the $L^2$-case \eqref{L2case}, we obtain that 
\begin{align}\label{f 0 H1}
\|\na_Y   \nz \|_2\leq C^\star_{\nz ,\dot H^1} (C_{2,\infty}).
\end{align}
This concludes the proof of the lemma.
\end{proof}\fi
Now we present the following lemma that provides estimates for the higher Sobolev norms. 
\begin{lem} \label{lem:HM_n_c}
Assume  the entropy and chemical gradient bound \eqref{EntrBound} and the bootstrap hypotheses \eqref{Hypotheses}. If the magnitude $A$ is chosen large enough compared to the constants in the bootstrapping hypothses, then the following estimates hold
\begin{align}
\|\lan n\ran \|_{L_t^\infty([0, T_\star]; H_Y^M)}\leq& C_{\lan n\ran ;H^M}(\mf M_0, C _{\lan n\ran;L^2},\|   \nz_{\mathrm{in}}  \|_{H_Y^M},\|   \na \cz_{\mathrm{in}}  \|_{H_Y^M})<\infty;\label{H_M_n}\\
\|\na_Y   \cz \|_{L_t^\infty([0,T_\star];H_Y^M)}\leq& C_{\na \cz ;H^M}(C_{\nz ;H^M},\| \na_Y    \lan \cc_{\mathrm{in}}\ran \|_{  H_Y^M})<\infty.\label{H_M_na_cz} 
\end{align}
\end{lem} \begin{proof}
The $\na \cz$-estimate \eqref{H_M_na_cz} is a direct consequence of the estimate \eqref{H_M_n} and the estimate \eqref{na_c_0_est}.

We prove the estimate \eqref{H_M_n} through induction. Assume that we obtain the following bound for $m-1,\quad  1\leq m\leq M,\, M\geq 3$:
\begin{align}
\|\nz\|_{L_t^\infty([0, T_\star];H^{m-1})}\leq C_{\nz; H^{m-1}}(C_{\lan n\ran; L^2},\|n_{\text{in}}\|_{H^{m-1}}).\label{Indc_hypo}
\end{align}
We would like to prove that the same estimate holds with `$m-1$'
 replaced by `$m$'. Direct energy estimate of the $\|\cdot\|_{\dot H^m}$-norm yields that
\begin{align}\label{tmevl_Hm_nz}\\
\frac{d}{dt}\frac{1}{2}\sum_{j+k=m}\|\pa_{y}^j\pa_{z}^k\nz\|_{L^2}^2=&-\frac{1}{A}\sum_{j+k=m}\|\na \pa_{y}^j\pa_{z}^k \nz\|_{L^2}^2+\frac{1}{A}\sum_{j+k=m}\int \na_{y,z}\pa_{y}^j\pa_{z}^k\nz \cdot \pa_y^j\pa_z^k( \nz\na\cz)dV\\
&+\frac{1}{A}\sum_{j+k=m}\int \na_{y,z}\pa_y^j\pa_z^k \nz\cdot \pa_y^j\pa_z^k\lan n_\nq \na \cc_{\nq} \ran dV\\
=:&-\frac{1}{A}\sum_{j+k=m}\|\na \pa_{y}^j\pa_{z}^k \nz\|_{L^2}^2+T_{1}+T_{2}.
\end{align}
To estimate $T_{1}$ term, we apply the Sobolev product estimate, the $L^\infty$-bounds \eqref{Linf_n}, \eqref{Linf_na_cz},  and the chemical gradient bound \eqref{na_c_0_est}  
{
\begin{align}\label{T_0_est}
|T_{1}|\leq &\frac{1}{4A}\sum_{j+k=m}\|\na\pa_y^j\pa_z^k \nz\|_{L_Y^2}^2+\frac{C}{A}\|\na \cz\|_{H_{Y}^{m}}^2\|\lan n\ran \|_{L^\infty_{Y}}^2+\frac{C}{A}\|\na \cz\|_{L_Y^\infty}^2\|\nz\|_{H_{Y}^m}^2\\
\leq&\frac{1}{4A}\sum_{j+k=m}\|\na\pa_y^j\pa_z^k \nz\|_{L^2_Y}^2+\frac{C}{A}\sup_{0\leq s\leq T_\star}\|\lan n\ran(s)\|_{H_Y^m}^2C_{ n;L^\infty}^2+\frac{C}{A}C_{\na \cz;L^\infty}^2\|\nz\|_{H_Y^m}^2.
\end{align}}
To estimate the $T_{2}$ term in \eqref{tmevl_Hm_nz}, we invoke the product estimate of Sobolev functions, bootstrap hypothesis (\ref{HypED_FM}) and \eqref{Bound_Linf} to estimate
\begin{align}\label{T_nq_est}
|T_{2}|=&\bigg|\frac{1}{A}\sum_{j+k=m}\int \na_{y,z}\pa_y^j\pa_z^k \nz\cdot \lan \pa_y^j\pa_z^k(\na  \cc_{\nq}n_\nq)\ran dV\bigg|\\
\leq &\frac{1}{8A}\sum_{j+k=m}\|\na\pa_y^j\pa_z^k \nz\|_{L_Y^2}^2+\frac{C}{A}\|\na \cc_\nq\|_{H^m(\Torus^3)}^2\|n_\nq\|_{L^\infty(\Torus^3)}^2\myr{+\frac{C}{A}\|n_\nq\|_{H^m(\Torus^3)}^2\|\na \cc_\nq\|_{L^\infty(\Torus^3)}^2}\\
\leq&\frac{1}{8A}\sum_{j+k=m}\|\na\pa_y^j\pa_z^k \nz\|_{L_Y^2}^2+\frac{C}{A}\CC^4 e^{-\delta t/A^{1/3}}. 
\end{align} \myc{\footnote{Check whether we have lower order terms in the product rule.}}
Combining the estimates \eqref{T_0_est}, \eqref{T_nq_est}, we apply the Gagliardo-Nirenberg inequality to obtain that 
\begin{align*}
\frac{d}{dt}\|\lan n\ran\|_{\dot H^m}^2\leq& -\frac{1}{AC}\frac{\|\nz \|_{\dot H^m}^{\frac{2m+2}{m}}}{\|\nz-\overline{n}\|_{L^2}^{\frac{2}{m}}}+\frac{C}{A}\sup_{0\leq s\leq T_\star}\|\lan n\ran\|_{H^m}^2C_{ n;L^\infty}^2+\frac{C}{A}C_{\na \cz;L^\infty}^2\|\nz\|_{H^m}^2+\frac{C}{A}\CC^4 e^{-\delta t/A^{1/3}}\\
\leq&-\frac{1}{AC}\frac{\|\nz \|_{\dot H^m}^{\frac{2m+2}{m}}}{C_{\nz;L^2}^{\frac{2}{m}}}+\frac{C}{A}\sup_{0\leq s\leq T_\star}\|\lan n\ran\|_{H^m}^2C_{ n;L^\infty}^2+\frac{C}{A}C_{\na \cz;L^\infty}^2\|\nz\|_{H^m}^2+\frac{C}{A}\CC^4 e^{-\delta t/A^{1/3}}.
\end{align*}
\myr{We further choose $A$ large enough and define 
\begin{align}
G_2(t)=\int_0^t\frac{C}{A}\CC^4 e^{-\delta s/A^{1/3}}ds\leq \frac{C}{\delta A^{2/3}}\CC^4\leq 1.
\end{align}
The equation above, together with the induction hypothesis \eqref{Indc_hypo},  yields that there exist constants $C_4,\ C_5$ such that  
\begin{align*}
\frac{d}{dt}\lf(\|\lan n\ran\|_{\dot H^m}^2-G_2\rg)\leq&-\frac{1}{AC}\frac{ \|\nz \|_{\dot H^m}^{\frac{2m+2}{m}}}{C_{\nz;L^2}^{\frac{2}{m}}}+\frac{C_4(C_{ n;L^\infty},C_{\na \cz;L^\infty})}{A}\sup_{0\leq s\leq T_\star}\|\lan n\ran\|_{\dot H^m}^2\\
&+\frac{C_5(C_{n;L^\infty}, C_{\na \cz;L^\infty}, C_{\lan n\ran;H^{m-1}})}{A}.
\end{align*}
We assume that the quantity $\|\nz\|_{\dot H^m}^2-G_2$ reaches the following threshold at the first instance $t_\ast$,
\begin{align}
\lf(\|\nz\|_{\dot H^m}^2-G_2\rg)\big|_{t=t_\star}=C\lf(1+\|\lan n_{\text{in}}\ran\|_{H^m}^2+C^{ {2} }_{\nz;L^2}C_4^{ {m} }+C_{\nz;L^2}^{\frac{2}{m+1}}C_5^{\frac{m}{m+1}}\rg).
\end{align}
Direct analysis yields that $\frac{d}{dt}(\|\nz\|_{\dot H^m}^2-G_2)\big|_{t=t_\star}\geq0.$ However, 
one can check that if the universal constant $C$ is chosen to be large enough, the right hand side of the differential inequality is strictly less than zero, and hence we have a contradiction.  
}
Hence, 
\begin{align}
\sup_{t\in[0,T_\star)}(\|\nz\|_{\dot H^m}^2-G_2) \leq C(C_{\lan n\ran;L^2}, C_{\lan n\ran;H^{m-1}}).
\end{align}This, when combined with the bound $G_2\leq 1,$ yields  \eqref{Indc_hypo} with $m-1$ replaced by  $m$. This concludes the induction.

\ifx{\color{red}Previous:

\begin{align}
|T_{00}|\leq &\frac{1}{4A}\sum_{j+k=M}\|\na\pa_y^j\pa_z^k \nz\|_2^2+\frac{C}{A}\|\na \cz\|_{W_{y,z}^{M,4}}^2\|\lan n\ran \|_{L^4_{y,z}}^2\\
&+\frac{C}{A}\|\na \cz\|_{L_{y,z}^\infty}^2\|\nz\|_{H_{y,z}^M}^2\\
\leq&\frac{1}{4A}\sum_{j+k=M}\|\na\pa_y^j\pa_z^k \nz\|_2^2+\frac{C}{A}\sup_{0\leq s\leq T_\star}\|\lan n\ran\|_{H^M}^2C_{\lan n\ran;L^4}^2+\frac{C}{A}(C^\star_{\na \cz;L^\infty})^2\|\nz\|_{H^M}^2.\label{T_00_est}
\end{align}

\footnote{ Application of the Gagliardo-Nirenberg inequality and the ODE argument yields the result (Add more detail?).}
Now we have that the first term can be absorbed by the dissipation term and the second term has small contribution in the time integral sense.
}\fi
\end{proof} 
\begin{proof}[Proof of \eqref{Con0_n_HM}, \eqref{Con0nac_HM+1}] By picking the constants 
\begin{align}\label{chc_B_avg}\mf B_{\nz ,H^M}\geq 2C_{\nz;H^M}, \mf B_{\nabla \cz; H^M}\geq 2 C_{\na\cz;H^M},
\end{align} we have proven the estimates \eqref{Con0_n_HM} and \eqref{Con0nac_HM+1}. 
\end{proof}
\section{Nonlinear Enhanced Dissipation of the  Remainder} \label{Sec:ED}
   
In this section, we prove the estimate for the enhanced dissipation functional $F_M$ \eqref{F_M_sect_2}, i.e.,   \eqref{ConED_FM}. We divide the proof into several steps. 

\noindent
\textbf{Step \# 1: Preliminaries and Goals.} 
Fix an arbitrary time $t_\star\geq 0$, and we would like to show that on the interval $[t_\star,t_\star+\delta^{-1}A^{1/3}]$, the following two estimates hold:
\begin{align}
F_{M}[t_\star+\tau]\leq &2F_M[t_\star],\quad \forall \tau\in[0,\delta^{-1}A^{1/3}];\label{F_M_Reg_est}\\
F_M[t_\star+\delta^{-1}A^{1/3}]\leq&\frac{1}{\myr{e^2}}F_{M}[t_\star].\label{FM_Decay_est}
\end{align}
Here $\delta$ is a constant depending only on the $\delta_0, \, C_\ast$ in the definition \eqref{ED1/3}. The explicit choice of $\delta$ is provided in \eqref{Chc_delta}. Once the two estimates are obtained, a classical argument yields the result \eqref{ConED_FM}. 

To make the presentation in the next two steps smoother, we provide some technical lemmas. Some of the proofs of these lemmas will be postponed to the end of the section. The first lemma provides estimates for the commutator terms. 
\begin{lem}\label{Lem:comm}Assume that $\mathfrak{H}_{ijk}\in H^1$. Then the following estimate holds for any $0\leq m\leq M+1$,
\begin{align}
\sum_{|i,j,k|=m}A^{i}\lf|\int  \mathfrak{H}_{ijk} [\pa_y^j, u]\pa_x^{i+1}\pa_z^k f dV\rg|\leq &\frac{C}{A^{1/2}}\lf(\sum_{|i,j,k|=m}A^i \|\mf H_{ijk}\|_2^2+\sum_{|i,j,k|\leq m}A^i \|\pa_x^i\pa_y^j\pa_z^k f\|_2^2\rg).
\end{align}
Here the constant $C$ depends only on $M, \, \| u\|_{L_t^\infty W^{m,\infty}}. $

\end{lem}
\begin{proof}
We estimate the commutator as follows
\begin{align*}
&\sum_{|i,j,k|=m}A^{i} \lf|\int  \mathfrak{H}_{ijk} [\pa_y^j, u]\pa_x^{i+1}\pa_z^k f dV\rg|
= \sum_{|i,j,k|=m}A^i \lf|\sum_{\ell=0}^{j-1}\binom{ j}{ \ell}  \int \mf H_{ijk}\  \pa_y^{j-\ell}u\ \pa_x^{i+1}\pa_y^{\ell}\pa_z^k fdV\rg|\\
&\leq\frac{C}{A^{  1/2}}\sum_{|i,j,k|=m}\sum_{\ell=0}^{j -1} A ^{\frac{i}{2}}\|\mf H_{ijk}\|_2 \  A ^{\frac{i+1}{2}}\|\pa_x^{i+1}\pa_y^\ell\pa_z^k f\|_2\\&
\leq \frac{C}{A^{1/2} }\lf(\sum_{|i,j,k|=m}A^i \|\mf H_{ijk}\|_2^2+\sum_{|i,j,k|\leq m}A^i \|\pa_x^i\pa_y^j\pa_z^k f\|_2^2\rg).
\end{align*}
\end{proof}
The second lemma concerns the estimates of the linear terms in the $\cc_\nq$ equation.
\begin{lem}\label{Lem:c_eq}Assume that $\mathfrak{H}_{ijk}\in H^1$. Then the following estimate holds
\begin{align}\label{T_c_est}\\
&\lf|\frac{1}{A}\sum_{|i,j,k|\leq M+1}A^{i} \int \mf H_{ijk}\  \pal n_\nq dV\rg| \leq  \frac{1}{4A}\sum_{|i,j,k|\leq M+1}A^{i} \|\na \mf H_{ijk}\|_2^2+\frac{C}{A^{\frac{1}{2}}}\sum_{|i,j,k|\leq M}A^{i+\frac {1}{2}}\|\pal n_\nq\|_2^2.
\end{align}
\end{lem}
\begin{proof} 
We distinguish among three cases and estimate accordingly\myc{ \footnote{(Double Check!)}}
\begin{align}
&\lf|\frac{1}{A}\sum_{|i,j,k|\leq M+1}A^{i} \int \mf H_{ijk}\  \pal n_\nq dV\rg|\\
 &\leq\bigg|\frac{1}{A^{\frac{1}{2}}}\sum_{\substack{|i,j,k|\leq M+1\\ i\neq 0}} A ^{i-\frac{1}{2}}\int\pa_{x}  \mf H_{ijk}\ \pa_x^{i-1}\pa_y ^j\pa_z^k n_\nq dV\bigg|+\bigg|\frac{1}{A}\sum_{\substack{j+k\leq M+1\\ j\neq 0}} \int  \pa_y \mf H_{ijk}\  \pa_y^{j-1}\pa_z^k n_\nq dV\bigg|\\
 &\quad+\bigg|\frac{1}{A}\sum_{k\leq M+1}  \int \pa_z \mf H_{00k} \pa_z^{k-1} n_\nq dV\bigg|\\
&\leq\frac{1}{4A}\sum_{|i,j,k|\leq M+1}A^i\|\na \mf H_{ijk}\|_2^2 +\frac{C}{A^{\frac{1}{2}}}\sum_{|i,j,k|\leq M} A^{i+\frac{1}{2}}\|\pa_x^i\pa_y^j\pa_z^k n_\nq \|_2^2.
\end{align}
This is  \eqref{T_c_est}.
\end{proof}
The last lemma equips us with suitable estimates for the nonlinear terms in the cell-density equation. 
\begin{lem}\label{Lem:n_eq} Assume that $\mathfrak{H}_{ijk}\in H^1$. Further assume all the bootstrap hypotheses \eqref{Hypotheses}. Then the following estimate holds for $M\geq3$,
\begin{align}\label{T_n_est}
&\lf|\frac{1}{A}\sum_{|i,j,k|\leq M}A^{i+\frac{1}{2}} \int  \mathfrak{H}_{ijk}\ \pal(\na\cdot(n\na \cc))_\nq dV\rg| \\
&\leq \frac{1}{4A}\myr{\sum_{|i,j,k|\leq M}A^{i+\frac{1}{2}}\|\na \mf H_{ijk}\|_2^2}+\frac{C}{A}F_M[n_\nq,\cc_\nq]^2+\frac{C}{A^ {\frac{1}{2}}}F_M[n_\nq, \cc_\nq]\lf(\mf B_\infty^2+\mf B_{\lan n\ran; H^M}^2+\mf B_{\na\lan \cc\ran; H^M}^2\rg).
\end{align}
Here $F_M$ is defined in \eqref{F_M_sect_2}, $\mf B_{\lan n\ran; H^M}, \, \mf B_{\na \cz;H^M}$ are defined in \eqref{Hypotheses},  and $\mf B_\infty$ is defined in \eqref{Bound_Linf}. Finally, the constant $C$ is universal.
\end{lem}
The remaining part of the section is organized as follows. In Step \# 2, we prove \eqref{F_M_Reg_est}. In Step \# 3, we  sketch the proof of \eqref{FM_Decay_est}. In Step \# 4, we carry out a standard argument to prove the enhanced dissipation.  

\noindent
\textbf{Step \#  2: Proof of \eqref{F_M_Reg_est}.}
We recall the definition \eqref{F_M_sect_2} and compute that
\begin{align}\label{ddtF_M}\\
\frac{d}{dt}F_M=&\frac{d}{dt}\sum_{|i,j,k|\leq M+1}A^{i} \|\pal \cc_\nq\|_2^2+\frac{d}{dt}\sum_{|i,j,k|\leq M}A^{i+1/2} \|\pal n_\nq\|_2^2 
=: \frac{d}{dt}T_{\cc} +\frac{d}{dt}T_n .
\end{align}
We first consider the simpler component $\frac{d}{dt}T_\cc$. Direct calculation yields that
\begin{align}\label{T_c_R_12} 
\frac{d}{dt}T_{\mf C}
=&-\frac{2}{A}\sum_{|i,j,k|\leq M+1}A^{i} \|\na \pal \cc_\nq\|_2^2  -2\sum_{|i,j,k|\leq M+1}A^{i}  \int \pal \cc_\nq \ [\pa_y^j,u](\pa_x^{i+1}\pa_z^k \cc_\nq) dV \\
&+\frac{2}{A}\sum_{|i,j,k|\leq M+1}A^{i} \int \pal \cc_\nq\  \pal n_\nq dV\\
=:&-D_\cc  +T_{\cc;1} +T_{\cc;2} .
\end{align}
We estimate  the $T_{\cc;1} $ term with the Lemma  \ref{Lem:comm} as follows
\begin{align}
|T_{\cc;1} |\leq \frac{C}{A^{1/2} }\sum_{|i,j,k|\leq M+1}A^{i} \|\pa_x^i\pa_y^j\pa_z^k \cc_\nq\|_2^2. \label{T_c_R_1} 
\end{align}  
Next we estimate the $T_{\cc;2} $-term in \eqref{T_c_R_12} with $\eqref{T_c_est}_{\mf H_{ijk}=\pal\cc_\nq}$
\begin{align}
|T_{\cc;2} |\leq\frac{1}{2}D_\cc+\frac{C}{A^{1/2} }\sum_{|i,j,k|\leq M}A^{i+1/2}\|\pal n_\nq\|_2^2.
\end{align} 
Therefore, by recalling the definition of the functional $F_M$ \eqref{F_M_sect_2}, we obtain that 
\begin{align}\label{T_c_R_est}
\frac{d}{dt}T_\cc \leq&-\frac{1}{A}\sum_{|i,j,k|\leq M+1}A^{i} \|\na \pal \cc_\nq\|_2^2+\frac{C}{A^{1 /2}}F_M[n_\nq,\mf C_\nq]. 
\end{align}

Next we estimate the term $\frac{d}{dt}T_{n} $ in \eqref{ddtF_M}. The time derivative can be decomposed as follows
\begin{align}\label{T_n_R_12} \\
\frac{d}{dt}T_n
=&-\frac{2}{A}\sum_{|i,j,k|\leq M}A^{i+\frac{1}{2}} \|\na \pal n_\nq\|_2^2 -2\sum_{|i,j,k|\leq M}A^{i+\frac{1}{2}} \int \pal n_\nq\  [\pa_y^j,u ](\pa_x^{i+1}\pa_z^k n_\nq) dV\\
&-\frac{2}{A}\sum_{|i,j,k|\leq M}A^{i+\frac{1}{2}} \int \pal n_\nq \ \pal(\na\cdot(n\na \cc))dV\\
=:&-D_n +T_{n;1}+T_{n;2}.
\end{align} 
First we apply Lemma \ref{Lem:comm} with $\mf H_{ijk}=\pal n_\nq, \, f=n_\nq$ to estimate the commutator term $T_{n;1}$ in \eqref{T_n_R_12} as follows:
\begin{align}
|T_{n;1}|\leq&\frac{C}{A^{1/2}}\lf(\sum_{|i,j,k|\leq M}A^{i+1/2} \|\pa_x^i\pa_y^j\pa_z^k n_\nq\|_2^2\rg).
\end{align}
Next we estimate the $T_{n;2}$-term in \eqref{T_n_R_12} with the estimate $\eqref{T_n_est}_{\mf H_{ijk}=\pal n_\nq}$ as follows
\begin{align}
|T_{n;2}|\leq \frac{1}{4 }D_n+\frac{C }{A}F_M[n_\nq,\cc_\nq]^2+\frac{C}{A^{1 /2}}F_M[n_\nq, \cc_\nq]\lf(\mf B_\infty^2+\mf B_{\lan n\ran; H^M}^2+\mf B_{\na\lan \cc\ran; H^M}^2\rg).
\end{align}
Hence,
\begin{align}\label{T_n_R_est}
\frac{d}{dt}T_n\leq\frac{C }{A}F_M[n_\nq,\cc_\nq]^2+\frac{C}{A^{1 /2}}F_M[n_\nq, \cc_\nq]\lf(1+\mf B_\infty^2+\mf B_{\lan n\ran; H^M}^2+\mf B_{\na\lan \cc\ran; H^M}^2\rg).
\end{align}
Collecting the estimates developed thus far (\eqref{T_c_R_est}, \eqref{T_n_R_est}), we have that
\begin{align}
\frac{d}{dt}F_M\leq\frac{C }{A}F_M^2+\frac{C}{A^{1 /2}}F_M\lf(1+\mf B_\infty^2+\mf B_{\lan n\ran; H^M}^2+\mf B_{\na\lan \cc\ran; H^M}^2\rg).
\end{align}
A standard ODE argument yields that if $A$ is chosen large enough,  
\begin{align}
F_M[n_\nq,\cc_\nq](t_\star+\tau)\leq 2F_M[n_\nq,\cc_\nq](t_\star),\quad \forall\tau\in [0,\delta^{-1}A^{1/3}].
\end{align}

\noindent
\textbf{Step \# 3: Proof of \eqref{FM_Decay_est}.}

The second estimate \eqref{FM_Decay_est} is a natural consequence of the following estimates,
\begin{align}\label{FM_Diff_est}
\sum_{|i,j,k|\leq M}&A^{i+1/2 } \|\pa_x^i\pa_y^j\pa_z^k n_{\neq}-S_{t_\star}^{t_\star+\tau}\pal n_\nq\|_2^2 +\sum_{|i,j,k|\leq M+1}A^{i} \|\pa_{x}^i\pa_y^j\pa_z^k\cc_\nq-S_{t_\star}^{t_\star+\tau}\pal \cc_\nq\|_2^2\\
& \leq\frac{1}{\myr{32}}F_M[t_\star],\quad \forall \tau\in [0,\delta^{-1}A^{1/3}].
\end{align}
Application of Theorem \ref{Thm:Lnr_ED} yields that there exists $0<\delta<\delta_0$ such that 
\begin{align} \label{Chc_delta}
\sum_{|i,j,k|\leq M}&A^{i +1/2} \|S_{t_\star}^{t_\star+\delta^{-1} A^{1/3}}\pal n_\nq\|_2^2 +\sum_{|i,j,k|\leq M+1}A^{i} \|S_{t_\star}^{t_\star+\delta^{-1} A^{1/3}}\pal \cc_\nq \|_2^2,\\
&\leq C_*F_M[t_\star]e^{-\frac{2\delta_0}{\delta}}\leq \frac{1}{\myr{128}}F_M[t_\star], \quad0<\delta\leq \frac {2\delta_0}{\log(\myr{128C_*})}.
\end{align}
Then we observe that the two estimates above yield the bound
\begin{align}
F_M[ t_\star+\delta^{-1}A^{1/3}  ]\leq 2\lf(\frac{1}{32}+\frac{1}{128}\rg)F_M[t_\star] \leq\frac{1}{\myr{e^2}}F_M[t_\star],
\end{align}
which is \eqref{FM_Decay_est}. Hence in the remaining part of the step, we prove the estimate \eqref{FM_Diff_est}.

To simplify the notation, we define
\begin{align}
\mathbb{V}_{ijk}:=\pa_x^i\pa_y^j\pa_z^k n_{\neq}-S_{t_\star}^{t_\star+\tau}\lf[\pal n_\nq(t_\star)\rg], \quad 
\mathbb{W}_{ijk}:=\pa_{x}^i\pa_y^j\pa_z^k\cc_\nq-S_{t_\star}^{t_\star+\tau}\lf[\pal \cc_\nq(t_\star)\rg].
\end{align}
We further define the norm that captures these differences:
\begin{align}
\mathbb{D}_M:=\sum_{|i,j,k|\leq M}A^{i+1/2} \|\mathbb V_{ijk}\|_2^2+\sum_{|i,j,k|\leq M+1}A^{i} \|\mathbb W_{ijk}\|_2^2.
\end{align}
The time evolution of $\mathbb{D}_M$ have two components, i.e.,
\begin{align}
\frac{d}{dt}\mathbb{D}_M=\frac{d}{dt}\sum_{|i,j,k|\leq M}A^{i+1/2} \|\mathbb V_{ijk}\|_2^2+\frac{d}{dt}\sum_{|i,j,k|\leq M+1}A^{i} \|\mathbb W_{ijk}\|_2^2.\label{ddt_D_M}
\end{align}
By invoking Lemma \ref{Lem:comm} and Lemma \ref{Lem:n_eq}, we estimate the first term in \eqref{ddt_D_M} as follows
\begin{align}\label{V_n}
\frac{d}{dt}&\sum_{|i,j,k|\leq M}A^{i+1/2} \|\mathbb V_{ijk}\|_2^2\\ 
=&-\frac{2}{A}\sum_{|i,j,k|\leq M}A^{i+1/2} \|\na \mathbb V_{ijk}\|_2^2 -2\sum_{|i,j,k|\leq M}A^{i+1/2} \int \mathbb V_{ijk}\  [\pa_y^j,u](\pa_x^{i+1}\pa_z^k n_\nq) dV\\
&-\frac{2}{A}\sum_{|i,j,k|\leq M}A^{i+1/2} \int\mathbb V_{ijk}\ \pal(\na\cdot(n\na \cc))dV\\
\leq&-\frac{1}{A}\sum_{|i,j,k|\leq M}A^{i+1/2} \|\na \mathbb V_{ijk}\|_2^2+\frac{C}{A^{1/2}}\sum_{|i,j,k|\leq M}A^{i+1/2}\|\mathbb{V}_{ijk}\|_2^2+\frac{C}{A^{1/2}}F_M[n_\nq,\cc_\nq]\CC^2\\
&+\frac{C}{A}F_M[n_\nq,\cc_\nq]^2. 
\end{align} We apply Lemma \ref{Lem:comm}, the estimate $\eqref{T_c_est}_{\mf H_{ijk}=\mathbb W_{ijk}}$ to estimate the second term in \eqref{ddt_D_M} as follows
\begin{align}\label{V_c}
\frac{d}{dt}&\sum_{|i,j,k|\leq M+1}A^{i} \|\mathbb W_{ijk}\|_2^2\\
=&-\frac{{2}}{A}\sum_{|i,j,k|\leq M+1}A^{i} \|\na \mathbb W_{ijk}\|_2^2  -2\sum_{|i,j,k|\leq M+1}A^{i} \int\mathbb W_{ijk} \ [\pa_y^j,u](\pa_x^{i+1}\pa_z^k \cc_\nq) dV \\
&+\frac{2}{A}\sum_{|i,j,k|\leq M+1}A^{i} \int \mathbb W_{ijk}\  \pal n_\nq dV\\
\leq&-\frac{1}{A}\sum_{|i,j,k|\leq M+1}A^{i} \|\na \mathbb W_{ijk}\|_2^2+\frac{C}{A^{1/2}}\sum_{|i,j,k|\leq M+1}A^{i}\|\mathbb{W}_{ijk}\|_2^2+\frac{C}{A^{1/2}}F_M[n_\nq,\cc_\nq] .
\end{align}  
Recalling the regularity bound \eqref{F_M_Reg_est}, we have that 
\begin{align}
\frac{d}{dt}&\mathbb D_{M} \leq  \frac{C}{A^{1/2}}\mathbb{D}_M+ \frac{C }{A}F_M[t_\star]^2+\frac{C}{A^{1 /2}}F_M[t_\star]\CC^2,\quad \mathbb D_{M}(t=t_\star)=0. 
\end{align}
A standard ODE argument yields that if $A$ is chosen large enough compared to the bootstrap constants, the estimate \eqref{FM_Diff_est} holds. This concludes Step \# 3. 

\noindent
\textbf{Step \# 4: Conclusion.}
The remaining argument to derive the enhanced dissipation estimate \eqref{ConED_FM} is standard. \myr{If $t\in \delta^{-1}A^{1/3}\mathbb{N}$, then we have that
\begin{align}
F_M[n_\nq(t),\cc_\nq(t)]\leq F_M[n_{\text{in};\nq}, \cc_{\text{in};\nq}]\exp\lf\{-\frac{2\delta t}{A^{1/3}}\rg\},\quad t\in \delta^{-1}A^{1/3}\mathbb{N}.
\end{align} For general $ t\geq0$, we identify the integer $N_t\in\mathbb{N}$ so that $t\in \delta^{-1}A^{1/3}[N_t,N_t+1)$. As a consequence of the estimates \eqref{F_M_Reg_est}, \eqref{FM_Decay_est}, we have
\begin{align}
 F_M&[t]\leq 2F_M[N_t\delta^{-1}A^{1/3}]\leq 2e^2 F_M[0]\exp\{-2(N_t+1)\} 
\leq   2e^2 F_M[0] \exp\lf\{ -\frac{2 \delta t}{ A^{1/3}} \rg\},\qquad\forall t\in[0,\infty).
\end{align} Hence we reach the conclusion \eqref{ConED_FM}. }

To conclude this section, we provide the proof of the technical Lemma \ref{Lem:n_eq}. 

\begin{proof}[Proof of Lemma \ref{Lem:n_eq}]
We decompose the left hand side of \eqref{T_n_est} as follows
\begin{align}\label{T_n_123}\\
\frac{1}{A}&\sum_{|i,j,k|\leq M}A^{i+1/2}  \int  \mathfrak{H}_{ijk}\ \pal(\na\cdot(n\na \cc))_\nq dV \\
= &-\frac{1}{A}\sum_{|i,j,k|\leq M}A^{i+1/2} \int  \na\mathfrak{H}_{ijk}\cdot \pal\lf(\nz\na \cc_\nq+ n_\nq\na \cz+ n_\nq\na \cc_\nq\rg)_\nq dV
=:T_{1}+T_{2}+T_{3}.
\end{align}

Before progressing to estimate each term in \eqref{T_n_123}, we recall that from the Gagliardo-Nirenberg inequality, we have the following estimate
\begin{align}\label{est_L4_H1}
\|\pal f\|_{L^4(\Torus^3)}\leq C\|\pal f\|_{L^2(\Torus^3)}^{1/4}\|\na \pal f\|_{L^2(\Torus^3)}^{3/4}\leq C\|\na \pal f\|_{L^2(\Torus^3)},\quad \overline{f}=0.
\end{align}

Now we estimate the $T_{1}$ term in \eqref{T_n_123} as follows
\begin{align}
\label{T_n_1} \\
T_{1}\leq& \frac{1}{12A}\sum_{|i,j,k|\leq M}A^{i+1/2} \|\na \mf H_{ijk}\|_2^2+\frac{C }{A^{1/2}}\sum_{\substack{|i,j,k|\leq M\\ i\neq 0}}A^i \sum_{\substack{(j',k')\leq (j,k)\\ (j',k')\neq  (j,k)}}\|\pa_y^{j-j'}\pa_z^{k-k'}\lan n\ran\|_4^2 \|\pa_x^{i}\pa_y^{j'}\pa_z^{k'}\na \cc_\nq\|_4^2\\
&+\frac{C }{ A^{1/2}}\sum_{|i,j,k|\leq M}A^{i} \|\lan n \ran\|_\infty^2\|\pal \na \cc_\nq\|_2^2+\frac{C }{A^{1/2}}\sum_{j+k\leq M} \|\pa_y^{j}\pa_z^{k} \nz\|_2^2\|\na \cc_\nq\|_\infty^2\\
&+\frac{C }{A^{1/2}}\sum_{j+k\leq M} \sum_{\substack{(j',k')\leq (j,k)\\ (j',k')\neq (0,0), (j,k)}}\|\pa_y^{j-j'}\pa_z^{k-k'} \nz\|_4^2\|\pa_y^{j'}\pa_z^{k'}\na \cc_\nq\|_4^2\\
=:&\sum_{\ell=1}^5 T_{1,\ell}.
\end{align}
We apply the $L^4$-estimate \eqref{est_L4_H1} to obtain that 
\begin{align}
T_{1;2}\leq &\frac{C }{A^{1/2 }}\sum_{\substack{|i,j,k|\leq M\\ i\neq 0}}\sum_{\substack{(j',k')\leq (j,k)\\ (j',k')\neq  (j,k)}} A^i \|\pa_x^{i}\pa_y^{j'}\pa_z^{k'}\na^2 \cc_\nq\|_2^2\|\pa_y^{j-j'}\pa_z^{k-k'}\na\lan n\ran\|_2^2\\
\leq &\frac{C}{A^{1/2} }\sum_{  |i,j,k|\leq M+1 }  A^i \|\pa_x^{i}\pa_y^{j}\pa_z^{k}  \cc_\nq\|_2^2\mf{B}_{\lan n\ran;H^M}^2.
\end{align}
Application of the $L^\infty$-estimates  \eqref{Bound_Linf} yields that
\begin{align}
T_{1,3}\leq&  \frac{C}{A^{1/2 }L^{1/2}}\sum_{|i,j,k|\leq M+1 } A^{i} \|\pal  \cc_\nq\|_2^2\mf B_\infty^2.
\end{align}Now, we apply the bound  \eqref{Bd_Linf_nq} to estimate the $T_{1,4}$-term in \eqref{T_n_1} as follows:
\begin{align}
T_{1,4}\leq \frac{C}{A^{1/2} }\sum_{j+k\leq M} \|\pa_y^{j}\pa_z^{k} \nz\|_2^2F_M[n_\nq,\cc_\nq]\leq \frac{C}{A^{1/2} }\mf{B}^2_{\nz;H^M}F_M[n_\nq,\cc_\nq]. 
\end{align}
We estimate the $T_{1,5}$ with \eqref{est_L4_H1} as follows:
\begin{align}
T_{1,5}\leq &\frac{C }{A^{1/2}}\sum_{j+k\leq M-1} \|\pa_y^{j}\pa_z^{k}\na^2 \cc_\nq\|_2^2\mf{B}_{\lan n\ran; H^M}^2 
\leq  \frac{C}{A^{1/2 } }F_M[n_\nq,\cc_\nq]\mf B_{\lan n\ran;H^M}^2.
\end{align}
To conclude, we have that 
\begin{align}
T_{1}\leq &\frac{1}{12A}\sum_{|i,j,k|\leq M}A^{i+1/2} \|\na \mf H_{ijk}\|_2^2+ \frac{C}{A^{1/ 2}} F_M[n_\nq, \cc_\nq]\CC^2.
\end{align}
Now a similar argument with the $L^\infty$-estimates \eqref{Bd_Linf_nq}, \eqref{Bound_Linf} and the $L^4$-estimate \eqref{est_L4_H1} yields that the $T_{2}$ term in \eqref{T_n_123} is bounded as follows\begin{align}
\label{T_n_2}T_{2}\leq& \frac{1}{12A}\sum_{|i,j,k|\leq M}A^{i+1/2} \|\na \mf H_{ijk}\|_2^2\\
&+\frac{C}{A}\sum_{\substack{|i,j,k|\leq M\\ i\neq 0}}A^{i+1/2} \sum_{\substack{(j',k')\leq(j,k)\\ (j',k')\neq (0,0),  (j,k)}} \|\pa_x^{i}\pa_y^{j'}\pa_z^{k'} n_\nq\|_4^2\|\pa_y^{j-j'}\pa_z^{k-k'}\na \cz\|_4^2\\
&+\frac{C}{ A^{}}\sum_{|i,j,k|\leq M}A^{i+1/2} \|\pal  n_\nq\|_2^2\|\na\lan \cc \ran\|_\infty^2+\frac{C }{A^{1/2}}\sum_{j+k\leq M} \|\pa_y^{j}\pa_z^{k} \na\cz\|_2^2\| n_\nq\|_\infty^2\\
&+\frac{C }{A^{1/2}}\sum_{j+k\leq M} \sum_{\substack{(j',k')\leq (j,k)\\ (j',k')\neq (0,0), (j,k)}}\|\pa_y^{j-j'}\pa_z^{k-k'} \na\cz\|_4^2\|\pa_y^{j'}\pa_z^{k'} n_\nq\|_4^2\\
\leq&\frac{1}{12A}\sum_{|i,j,k|\leq M}A^{i+1/2} \|\na \mf H_{ijk}\|_2^2\\
&+\frac{C}{A}\sum_{\substack{|i,j,k|\leq M\\ i\neq 0}}\sum_{\substack{(j',k')\leq (j,k)\\ (j',k')\neq  (j,k)}} A^{i+1/2} \|\pa_x^{i}\pa_y^{j'}\pa_z^{k'}\na n_\nq\|_2^2\|\pa_y^{j-j'}\pa_z^{k-k'}\na^2\lan \cc\ran\|_2^2\\
&+\frac{C}{A^{}}\sum_{|i,j,k|\leq M} A^{i+1/2 }  \|\pal  n_\nq\|_2^2\mf B_\infty^2+\frac{C}{A}\sum_{j+k\leq M} \|\pa_y^{j}\pa_z^{k} \na\cz\|_2^2F_M[n_\nq,\cc_\nq]\\
&+\frac{C}{A}\sum_{j+k\leq M}A^{1/2}  \sum_{\substack{(j',k')\leq (j,k)\\ (j',k')\neq (0,0), (j,k)}}\|\pa_y^{j-j'}\pa_z^{k-k'}\na \cz\|_{\dot H^1}^2\|\pa_y^{j'}\pa_z^{k'}\na n_\nq\|_2^2\\
\leq &\frac{1}{12A}\sum_{|i,j,k|\leq M}A^{i+1/2} \|\na \mf H_{ijk}\|_2^2+\frac{C}{A }F_M[n_\nq, \cc_\nq]\CC^2. 
\end{align}

Finally, we estimate the $T_{3 }$ term in \eqref{T_n_123}. Here we use the definition of $F_M$ \eqref{F_M_sect_2}, the $L^\infty$-estimates \eqref{Bound_Linf} and the $L^4$-bound  \eqref{est_L4_H1} to estimate it as follows
\begin{align}
\label{T_n_3} T_{3}\leq& \frac{1}{12A}\sum_{|i,j,k|\leq M}A^{i+1/2} \|\na \mf H_{ijk}\|_2^2\\
&+\frac{C}{A}\sum_{|i,j,k|\leq M}A^{i+1/2} \sum_{\substack{(i',j',k')\leq (i,j,k)\\ (i',j',k')\neq (0,0,0), (i,j,k)}}\|\pa_{x}^{i-i'}\pa_y^{j-j'}\pa_z^{k-k'}n_\nq\|_4^2 \|\pa_x^{i'}\pa_y^{j'}\pa_z^{k'}\na \cc_\nq\|_4^2\\
&+\frac{C }{A^{1/2}}\sum_{|i,j,k|\leq M}\|n_\nq\|_\infty^2A^{i} \|\pal \na \cc_\nq\|_2^2+\frac{C}{A}\sum_{|i,j,k|\leq M} A^{i+1/2}  \|\pal n_\nq\|_2^2\|\na \cc_\nq\|_\infty^2\\
\leq& \frac{1}{12A}\sum_{|i,j,k|\leq M}A^{i+1/2} \|\na \mf H_{ijk}\|_2^2\\
&+\frac{ C}{A}\sum_{|i,j,k|\leq M}\sum_{\substack{(i',j',k')\leq (i,j,k)\\ (i',j',k')\neq (0,0,0), (i,j,k)}}A^{(i-i')+\frac 1 2}  \|\pa_x^{i-i'}\pa_y^{j-j'}\pa_z^{k-k'}\na n_\nq\|_2^2A^{i'}  \|\pa_x^{i'}\pa_y^{j'}\pa_z^{k'}\na \cc_\nq\|_{\dot H^1}^2\\
&+\frac{CF_M[n_\nq,\cc_\nq]^2}{A}\\
\leq &\frac{1}{12A}\sum_{|i,j,k|\leq M}A^{i+1/2}\|\na \mf H_{ijk}\|_2^2+\frac{C}{A}F_M[n_\nq, \cc_\nq]^2.
\end{align}
Combining the decomposition \eqref{T_n_123}, the estimates \eqref{T_n_1}, \eqref{T_n_2} and \eqref{T_n_3} yields the result \eqref{T_n_est}. 
\end{proof}

\ifx
The small  initial chemical gradient \eqref{smll_nq} ensures that 
\begin{align}
\mathcal{F}[n_{\mathrm{in};\neq},\na c_{\mathrm{in};\neq}]=\mathcal{O}(A^{1/6}).
\end{align}

With the functional $\mathcal{F}$ introduced, we can explicitly lay out the battle plan. Fix 
\begin{align}\label{Defn_delta}
\delta:=\frac{\delta_0 }{\log (8C_0^2)},
\end{align} where $C_0,\, \delta_0$ are defined in \eqref{ED_tmshr}. The goal of this section is to prove the following enhanced dissipation estimates for $\mathcal{F}$
\begin{align}\label{F_ED}
\mathcal{F}[\pa_x^j n_{\neq}(t), \pa_x^j \cc_{\neq}(t)]\leq C_{ED}\mathcal{F}[\pa_x^j n_{\mathrm{in};\neq},\pa_x^j c_{\mathrm{in};\neq}]e^{-2\delta \frac{t}{A^{1/3} }},\quad j=\{0,1\}, \, \forall t\in[0,T_\star].
\end{align}
Applying the same argument as in \eqref{...}, we have that the estimates \eqref{F_ED} is guaranteed by the decay estimate and regularity estimate of the functional $\mathcal{F}[n_{\neq}(t_\star+t),\cc_{\neq}(t_\star+t)]$:

\noindent
a) $\mathcal{F}$-decay estimate:
\begin{align}
\mathcal{F}[\pa_x^j n_{\neq}(t_\star+\delta^{-1}A^{1/3} ),\pa_x^j \cc_{\neq}(t_\star+\delta^{-1}A^{1/3} )]\leq& \frac{1}{2}\mathcal{F}[\pa_x^j n_{\neq}(t_\star), \pa_x^j \cc_{\neq}(t_\star)],\quad j\in\{0,1\};\label{Decay_estimate_F_nonlinear}
\end{align}
b) $\mathcal{F}$-regularity estimate:
\begin{align}
\mathcal{F}[\pa_x^j n_{\neq}(t_\star+t),\pa_x^j \cc_{\neq}(t_\star+t)]\leq 4\mathcal{F}[\pa_x^jn_{\neq}(t_\star),\pa_x^j \cc_{\neq}(t_\star)],\quad \forall t\in[0, \delta^{-1}A^{1/3} ], \quad j\in\{0,1\}.\label{Regularity_estimate_F_nonlinear}
\end{align}
We first apply Lemma \ref{lem:ED} to show that the quantities $\mathcal{F}[ S_{t_\star,t_\star+t}n_{\neq}(t_\star), S_{t_\star,t_\star+t}|\pa_x|^{1/2} \cc_{\neq}(t_\star),S_{t_\star,t_\star+t}\na \cc_{\neq}(t_\star) ]$ undergoes enhanced dissipation. Recall that the $S_{t_\star,t_\star+t}(\cdot)$ denote the solution operator associated with the passive scalar. Next by showing that the $\mathcal{F}$-induced distance between the linearized solutions $\wwt n_{\neq}, \wwt \cc_{\neq}$ and the the passive scalar solutions is small, we prove the enhanced dissipation of the quantity $\mathcal{F}[\wwt{n}_{\neq},\na \wwt \cc_{\neq}]$ for $A$ chosen large enough. In this step, we address the issue of the shear flow-induced destabilizing effect. Finally, we estimate the $\mathcal{F}$-distance between the linearized solutions $\wwt {n}_{\neq},\, \wwt \cc_{\neq}$ and the exact solutions $n_{\neq},\cc_{\neq}$, and show that it is small on the time interval $[t_\star,t_\star+\delta^{-1}A^{1/3}]$, 
which in turn yields the desired estimates 
\eqref{Decay_estimate_intro_sect_6} and \eqref{Regularity_estimate_intro_sect_6}. The nonlinear self-interaction between remainders are treated in this step. 

In the remaining part of the section , we adopt the following notation conventions:
\begin{align}
B^2:=&\mathcal{F}[n_{\neq}(t_\star), \na \cc_{\neq}(t_\star)]\\
=&A^{1/6}\|n_{\neq}(t_\star)\|_2^2+A^{3/4}\|\pa_x \cc_{\neq}(t_\star)\|_2^2+\sum_{i=1}^2\|\pa_{y_i} \cc_{\neq}(t_\star)\|_2^2+\ep^{-2}\|e^{-\frac{\delta_0 t_\star}{ 4A^{1/3}}}\na^2 c_{\mathrm{in};\neq}\|_2^2. \label{defn_B_B_1} 
\end{align}
Next we consider the linearization of the original problem \eqref{linearized_system}. We do the Fourier transform in the $x$-direction. The corresponding Fourier coefficients are indexed by $k$. Here we omit the $\wh{(\cdot)}$ notation. The resulting system is the following:\begin{subequations}
\label{linearized_system_k_by_k}
\begin{align}
\pa_t \wwt {n}_{k}&=\frac{1}{A}(\de_y-|k|^2) \wwt{n}_{k}-u(t,y )ik \wwt{ n}_{k},\\
\quad \pa_t \wwt {c}_{k} &= \frac{1}{A}(\de_y-|k|^2) \wwt{c}_{k}-u(t,y )ik \wwt{c}_{k};\\
\pa_t  {d}_{k} &= \frac{1}{A}(\de_y-|k|^2) {d}_{k}-u(t,y )ik  {d}_{k};\\
\wwt n_k(t_\star,\cdot)&= n_k(t_\star,\cdot),\quad \wwt c_k(t_\star,\cdot)= c_k(t_\star,\cdot),\quad d_k(t=0,\cdot)=c_{\mathrm{in},k}(\cdot).
\end{align}
\end{subequations}

At the next stage, we compare the solutions $\{\wwt n_{\neq},\, \wwt {c}_{\neq}\}$ to the linearized system \eqref{linearized_system_k_by_k} and the passive scalar solutions $S_{t_\star,t_\star+t}\{n_\nq,\na \cc_{\neq}\}$. This the the content of the next lemma.
\begin{lem} Consider the solutions to the linearized system  \eqref{linearized_system_k_by_k}. There exists universal constant $C$ such that the following enhanced dissipation estimate holds on the time interval $t\in[0,\delta^{-1}A^{-1/3}]$ 
\begin{align}\label{Linearized_soln_F_difference}
\bigg|\mathcal{F}[S_{t_\star,t_\star+t }  n_{\neq}(t_\star),S_{t_\star,t_\star+ t} \na  \cc_{\neq}(t_\star) ]-\mathcal{F}[\wwt n_{\neq}(t_\star+t),\na \wwt \cc_{\neq}(t_\star+t)]\bigg|\leq C \frac{1+\|\pa_{y }u\|_\infty^2}{\delta^2 A^{1/24}} B^2,\\
\mathcal{G}[\wwt n_{\neq}(t_\star+t),\na \wwt \cc_{\neq}(t_\star+t)]\leq \left(1+C \frac{1+\|\pa_{y }u\|_\infty^2}{\delta^2 A^{1/24}}\right) B^2. \label{Linearized_soln_G}
\end{align}
\end{lem}
\begin{proof}
We estimate the difference using the facts that $S_{t_\star,t_\star+t}\{n_\nq, \pa_x \cc_\nq, \pa_{z }\cc_\nq\}=\{\wwt n_\nq,\pa_x \wwt \cc_\nq, \pa_{z } \wwt \cc_\nq\}$, 
\begin{align}
|\mathcal{F} [S_{t_\star,t_\star+t}&   n_{\neq}, S_{t_\star,t_\star+t}  \na \cc_{\neq}]- \mathcal{F}[\wwt n_{\neq},\na\wwt \cc_{\neq}](t_\star+t)|\\
=&A^{1/6}(\|S_{t_\star,t_\star+t} n_{\neq}\|_2^2-\|\wwt n_{\neq}\|^2_2)+(\|S_{t_\star,t_\star+t}  \na_Y    \cc_{\neq}\|_2^2-\|\na_Y   \wwt \cc_{\neq}\|_2^2)+A^{3/4}(\|
S_{t_\star,t_\star+t}\pa_x  \cc_{\neq}\|_2^2-\|\pa_x\wwt \cc_{\neq}\|^2_2)\\
\leq&\left(2\|S_{t_\star,t_\star+t}  \pa_{y }  \cc_{\neq}\|_2+\|S_{t_\star,t_\star+t}  \pa_{y }  \cc_{\neq}-\pa_{y }\wwt \cc_{\neq}\|_2\right)\|S_{t_\star,t_\star+t}  \pa_{y }  \cc_{\neq}-\pa_{y } \wwt \cc_{\neq}\|_2.\label{F_diff_wt_nc_wwt_ nc}
\end{align}
Applying the H\"older inequality and the Young's inequality, we obtain that
\begin{align}
\frac{1}{2}\frac{d}{dt}&\|\pa_{y } \wwt \cc_{\neq}-S _{t_\star,t_\star+t}\pa_{y }  \cc_{\neq}\|_2^2 
\leq \frac{1}{A^{3/8}}\|\pa_{y } \wwt \cc_{\neq}-S _{t_\star,t_\star+t}\pa_{y }  \cc_{\neq}\|_2\|\pa_{y }u\|_{L_{t,y}^\infty}  A^{3/8}\|\pa_x\wwt \cc_{\neq}\|_2 . 
\end{align}
Hence,
\begin{align}
\frac{d}{dt}&\|\pa_{y } \wwt \cc_{\neq}-S _{t_\star,t_\star+t}\pa_{y }  \cc_{\neq}\|_2\leq \frac{1}{A^{3/8}}\|\pa_{y }u\|_{L_{t,y}^\infty}  A^{3/8}\|\pa_x\wwt \cc_{\neq}\|_2 . 
\end{align} Integrating the above differential inequality and applying the definition of $\mathcal{F}$ \eqref{defn_F} and the fact that 
$\|\pa_x \wwt \cc_{\neq}(t_\star+t)\|_2\leq \|\pa_x \wwt \cc_{\neq}(t_\star)\|_2$, we obtain that the following estimate
\begin{align}
\|\pa_{y } &\wwt \cc_{\neq}(t_\star+t)-S_{t_\star,t_\star+t}\pa_{y } \cc_{\neq}(t_\star)\|_2\leq \frac{1}{\delta A^{1/24}}\|\pa_{y }u\|_{L^\infty}\sqrt{\mathcal{F}[ n_{\neq}(t_\star),\na \cc_\nq(t_\star)]},\quad \forall t\in[0,\delta^{-1}A^{1/3}].
\end{align} 
Combining this estimate with  \eqref{F_diff_wt_nc_wwt_ nc} and the fact that $\|S_{t_\star,t_\star+t}\pa_{y }\cc_\nq\|_2^2\leq \|\pa_{y }\cc_\nq(t_\star)\|_2^2\leq\mathcal{F}[n_\nq(t_\star),\cc_\nq(t_\star)]$ yields the result.  
\end{proof}
In the last step, we focus on the actual nonlinear system \eqref{ppPKS_nqmd}. Here we incorporate the nonlinear interaction  term $\na\cdot(n_{\neq}(\na \cc_{\neq}+\na d_{\neq}))$.
Here we will compare the solutions $n_{\neq},\, \cc_{\neq}$ to the solutions $\wwt n_{\neq}, \, \wwt \cc_{\neq}$ of  \eqref{linearized_system}. 
\ifx
\begin{align*}
\pa_t \wt{\wt n}_{\neq}=\frac{1}{A}\de \wt{\wt{n}}_{\neq}-u(y)\pa_x \wt{\wt n}_{\neq}-\frac{1}{A}\na \cdot(\na \wt{\wt c}_{\neq} \nz +\na \cz \wt{\wt{n}}_{\neq}),\quad \pa_t \wt{\wt c}_{\neq} = \text{the same as before}
\end{align*}
and the solution $n_{\neq},\cc_{\neq}$.
\fi
\begin{proof}[Proof of \eqref{ConED_L2}]
We divide the proof into three steps. 

\noindent
\textbf{Step \# 1: Reformulations of \eqref{Decay_estimate_F_nonlinear} and \eqref{Regularity_estimate_F_nonlinear}.} 
First of all, we observe that the passive scalar solutions $\{S_{t_\star,t_\star+t} \{n_\nq,\na \cc_\nq\}$ satisfy the estimates
\begin{align}
\mathcal{F}[S_{t_\star,t_\star+ \delta^{-1} A^{1/3} }  n_{\neq}(t_\star),S_{t_\star,t_\star+ \delta^{-1} A^{1/3} } \na  \cc_{\neq}(t_\star) ]\leq&\frac{1}{8}\mathcal{F}[n_{\neq}(t_\star), \na  \cc_{\neq}(t_\star) ]; \label{F_wt_nc_decay}
\\
\mathcal{F}[S _{t_\star,t_\star+t}  n_{\neq}(t_\star),S_{t_\star,t_\star+t} \na  \cc_{\neq}(t_\star)]\leq&\mathcal{F}[ n_{\neq}(t_\star), \na   \cc_{\neq}(t_\star)], \quad\forall t\in [0,\delta^{-1}A^{1/3}].\label{F_wt_nc_regularity}
\end{align}
 We define the difference
\begin{align}
\mathbb{N}_{\neq}:= n_{\neq}-\wwt {n}_{\neq};\,\,
\mathbb{C}_{\neq}:=\cc_{\neq}-\wwt {c}_{\neq}.
\end{align}
Then the equations for $\mathbb{N}_{\neq}$ and $
\mathbb{C}_{\neq}$ have the following form:
\begin{align}
\pa_t \mathbb{N}_{\neq}
=&\frac{1}{A}\de \mathbb{N}_{\neq}-u(t,y )\pa_x \mathbb{N}_{\neq}-\frac{1}{A}\na\cdot\left((\na \mathbb{C}_{\neq}+\na \wwt \cc_{\neq}) \nz +\na_{y } \cz  (\mathbb{N}_{\neq}+\wwt n_{\neq})+\left((\na \cc_{\neq} +\na d_{\neq})n_{\neq}\right)_{\neq}\right)
\label{N_eq}\\
\pa_t \mathbb{C}_{\neq}=&\myb{\frac{1}{A}\de \mathbb{C}_{\neq}-u(t,y )\pa_x 
\mathbb{C}_{\neq}+\frac{1}{A} n_{\neq}=}\frac{1}{A}\de \mathbb{C}_{\neq}-u(t,y )\pa_x \mathbb{C}_{\neq}+\frac{1}{A} (\mathbb{N}_{\neq}+\wwt n_{\neq}).\label{C_eq}
\end{align}
Since the linearized solutions $\{\wwt n_\nq,\wwt \cc_\nq\}$ are identical to the real solutions $\{n_\nq,\cc_\nq\}$ at instance $t=t_\star$,  the differences $\mathbb{N}_{\neq}$ and $\mathbb{C}_{\neq}$ are zero at $t_\star$, i.e., 
\begin{align}
\mathbb{N}_{\neq}(t_\star)=n_{\neq}(t_\star)-\wwt {n}_{\neq}(t_\star)=0,\quad 
\mathbb{C}_{\neq}(t_\star)=\cc_{\neq}(t_\star)-\wwt {c}_{\neq}(t_\star)=0. \label{Initial_data_for_NC}
\end{align}
To derive the regularity estimate \eqref{Regularity_estimate_F_nonlinear} and the decay estimate \eqref{Decay_estimate_F_nonlinear} for the nonlinear solution $n_{\neq},\, \na \cc_{\neq}$, it is enough to show that
\begin{align}\label{Goal_F_comparison}
 \mathcal{G}[\mathbb{N}_{\neq}(t_\star+t), \na 
\mathbb{C}_{\neq}(t_\star+t)]\leq C\left(\frac{1}{\delta A^{1/12}}+ \ep^2\right)B^2,\quad \forall t\in[0, \delta^{-1}A^{1/3}],
\end{align} where $B$ is defined in \eqref{defn_B_B_1}.
\myc{Can we just use one $\delta$?}
The explicit justification is as follows. To derive \eqref{Regularity_estimate_F_nonlinear}, we decompose the functional as follows
\begin{align*}
\mathcal{F}&[n_\nq (t_\star+t),\na \cc_\nq(t_\star+t)]\\
\leq&\mathcal{F}[S_{t_\star,t_\star+t}n_\nq, S_{t_\star, t_\star+t}\na \cc_\nq]+\big|\mathcal{F}[S_{t_\star,t_\star+t}n_\nq, S_{t_\star, t_\star+t}\na \cc_\nq]-\mathcal{F}[\wwt n_\nq(t_\star+t),\na \wwt c _\nq(t_\star+t)]\big|\\
&+\big|\mathcal{F}[n_\nq(t_\star+t),\na \cc_\nq(t_\star+t)] -\mathcal{F}[\wwt n_\nq(t_\star+t),\na \wwt c _\nq(t_\star+t)]\big|.
\end{align*}
Combining estimates \eqref{F_wt_nc_regularity}, \eqref{Linearized_soln_F_difference} and \eqref{Linearized_soln_G}, one obtains that 
\begin{align*}
\mathcal{F}&[n_\nq (t_\star+t),\na \cc_\nq(t_\star+t)]\\
\leq& B^2+C\frac{1+\|\pa_{y }u\|_\infty^2}{\delta^2A^{1/24}} B^2+\bigg(A^{1/6}\big\\|n_\nq\|_2^2-\|\wwt n_\nq\|_2^2\big|+A^{3/4}\big\\|\pa_x \cc_\nq\|_2^2-\|\pa_x \wwt \cc_{\neq}\|_2^2\big|+\big\|\na_Y   \cc_\nq\|_2^2-\|\na_Y   \cc_\nq\|_2^2\big|
\bigg)\\
\leq &B^2+C\frac{1+\|\pa_{y }u\|_\infty^2}{\delta^2A^{1/24}} B^2+C\left(\mathcal{G}^{1/2}[\wwt n_\nq,\na\wwt \cc_\nq]+\mathcal{G}^{1/2}[\mathbb{N}_\nq,\na \mathbb{C}_\nq]\right)\mathcal{G}^{1/2}[\mathbb{N}_\nq,\na \mathbb{C}_\nq]\\
\leq &2B^2+C\frac{1+\|\pa_{y }u\|_\infty^2}{\delta^2A^{1/24}} B^2+C\left(\frac{1}{\delta A^{1/12}}+ \ep^2\right)B^2.
\end{align*}
By taking $\ep^{-1}$ and $A$ large enough depending on $\delta,\,\|\pa_{y }u\|_\infty$ and constants, we have obtained \eqref{Regularity_estimate_F_nonlinear}. For the estimate \eqref{Decay_estimate_F_nonlinear}, one can proceed in a similar manner with \eqref{F_wt_nc_regularity} replaced by \eqref{F_wt_nc_decay}. We omit further details for the sake of brevity. This concludes step \# 1. 
 
\noindent
\textbf{Step \# 2: Proof of the estimate \eqref{Goal_F_comparison}.} We estimate the size of the deviation $\mathbb{N}_{\neq}$ and $\na \mathbb{C}_{\neq}$. 

Now we start by estimating the time evolution of the $L^2$-norms of  the differences $
\mathbb{N}_{\neq}$  
\begin{align}
\frac{1}{2}\frac{d}{dt}\|
\mathbb{N}_{\neq}\|_2^2\leq& -\frac{1}{8A}\|\na 
\mathbb{N}_{\neq}\|_2^2+ \frac{C}{A}\|\na 
\mathbb{C}_{\neq}\|_2^2\|\nz \|_\infty^2+\frac{C}{A}\|\na_Y   \cz \|_\infty^2\|
\mathbb{N}_{\neq}\|_2^2\\
&+\frac{C}{A}\left(\|\na \cc_{\neq}\|_\infty^2\|n_{\neq}\|_2^2+\|\na \wwt \cc_{\neq}\|_2^2\|\nz \|_\infty^2+\|\na_{y } \cz \|_\infty^2\|\wwt n_{\neq}\|_2^2\right).
\end{align}
Now we recall the estimate of the linearized solutions   \eqref{Linearized_soln_F_difference} and the hypotheses \eqref{Hypotheses}, and estimate the time evolution as follows
\begin{align}
\frac{1}{2}&\frac{d}{dt}\|\mathbb{N}_{\neq}\|_2^2\\
\leq&-\frac{1}{8A}\|\na \mathbb{N}_{\neq}\|_2^2+ \frac{C}{A}\|\na \mathbb{C}_{\neq}\|_2^2\|\nz \|_\infty^2+\frac{C}{A}\|\na_Y   \cz \|_\infty^2\|\mathbb{N}_{\neq}\|_2^2\\
&+ {\frac{C}{A}(\|\na \cc_{\neq}\|_\infty^2+\|\na d_{\neq}\|_\infty^2)(\|n_{\neq}-\wwt n_{\neq}\|_2^2+\|\wwt n_{\neq}\|_2^2})+\frac{C}{A}(\|\nz \|_\infty^2+\|\na_Y   \cz \|_2^2)(\|\wwt n_{\neq}\|_2^2+\|\na \wwt \cc_{\neq}\|_2^2)\\
\leq&-\frac{1}{8A}\|\na \mathbb{N}_{\neq}\|_2^2+ \frac{C(C_{n;L^\infty})}{A}\|\na 
\mathbb{C}_{\neq}\|_2^2 \\
&+\left(\frac{C(C_{\na_Y   \cz ;L^\infty},C_{\na \cc_{\neq};L^\infty})}{A}+\frac{\|\na ^2 c_{\mathrm{in};\neq}\|_2^2}{A}(t_\star+t)^2e^{-\delta_0 \frac{t+t_\star}{A^{1/3}}}\right)\|\mathbb{N}_{\neq}\|_2^2\\
&+\frac{1}{A}C(C_{\na \cc_{\neq};L^\infty},C_{\na \cz ;L^\infty},C_{n;L^\infty} )\|\wwt n_{\neq}(t_\star+t)\|_2^2+\frac{\|\na^2 c_{\mathrm{in};\neq}\|_2^2}{A}(t_\star+t)^2e^{-\delta_0 \frac{t+t_\star}{A^{1/3}}}\|\wwt n_{\neq}(t_\star+t)\|_2^2\\
\leq&-\frac{1}{8A}\|\na \mathbb{N}_{\neq}\|_2^2+ \frac{C(C_{n;L^\infty})}{A}\|\na \mathbb{C}_{\neq}\|_2^2 +\left(\frac{C(C_{\na_Y   \cz ;L^\infty}, C_{\na \cc_{\neq};L^\infty}))}{A}+\frac{C\ep^2}{A}(t_\star+t)^2e^{-\delta_0\frac{t_\star+t}{A^{1/3}}}\right)\|
\mathbb{N}_{\neq}\|_2^2\\
&+\left(\frac{C(C_{\na \cc_{\neq};L^\infty}, C_{\na \cz ;L^\infty}, C_{n;L^\infty})}{A}+\frac{C\ep^2}{A}(t_\star+t)^2e^{-\delta_0\frac{t_\star+t}{A^{1/3}}}\right)\mathcal {F}[n_{\neq}(t_\star),\na \cc_{\neq}(t_\star)]A^{-1/6}.\label{N_eq_est}
\end{align}

Next we estimate the time evolution of each component in $\mathcal{G}[
\mathbb{N}_{\neq},\mathbb{C}_{\neq}]$ involving the difference $\mathbb{C}_{\neq}$ as follows:
\begin{align}
\frac{1}{2}\frac{d}{dt}A^{3/4}\|\pa_x \mathbb{C}_{\neq}\|_2^2\leq &-\frac{1}{4A^{1/4}}\|\na \pa_x \mathbb{C}_{\neq}\|_2^2+\frac{1}{A^{3/8}}CA^{1/8}\|\mathbb{N}_{\neq}\|_2^2+\frac{1}{A^{3/8}}CA^{1/8}\|\wwt{n}_{\neq}\|_2^2;\label{C_eq_est1}\\
\frac{1}{2}\frac{d}{dt}\|\na_Y   \mathbb{C}_{\neq}\|_2^2
\leq &-\frac{1}{4A}\|\na \na_Y   \mathbb{C}_{\neq}\|_2^2+\|\pa_{y }u\|_\infty \frac{1}{A^{3/8}}\|A^{3/8}\pa_x \mathbb{C}_{\neq}\|_2\|\na_Y   \mathbb{C}_{\neq}\|_2\\
 &+\frac{1}{A}C\|\mathbb{N}_{\neq}\|_2^2+\frac{1}{A}C\|\wwt n_{\neq}\|_2^2\\
\leq &-\frac{1}{4A}\|\na \na_Y   \mathbb{C}_{\neq}\|_2^2+\frac{1}{A^{3/8}}\|A^{3/8}\pa_x\mathbb{C}_{\neq}\|_2^2+\frac{1}{A^{3/8}}\|\pa_{y }u\|_\infty^2\|\na_Y  \mathbb{C}_{\neq}\|_2^2\\
&+ \frac{C}{A}\|\mathbb{N}_{\neq}\|_2^2+\frac{C}{A}\|\wwt n_{\neq}\|_2^2.\label{C_eq_est2}
\end{align}
Combining \eqref{N_eq_est}, \eqref{C_eq_est1}, and \eqref{C_eq_est2}, we choose $A$ large enough such that 
\begin{align}
\frac{d}{dt}&\left(A^{1/6}\|\mathbb{N}_{\neq}\|_2^2+A^{3/4}\| \pa_x \mathbb{C}_{\neq}\|_2^2+\|\na_Y   \mathbb{C}_{\neq}\|_2^2\right)\\
\leq& \frac{C}{A^{3/8}}\left(A^{1/6}\|\mathbb{N}_{\neq}\|_2^2+A^{3/4}\|\pa_x \mathbb{C}_{\neq}\|_2^2+\|\na_Y   \mathbb{C}_{\neq}\|_2^2\right)+\frac{C\ep^2}{A}(t_\star+t)^2e^{-\delta\frac{t_\star+t}{A^{1/3}}}(A^{1/6} \|\mathbb{N}_{\neq}\|_2^2)\\
&+\left(\frac{C }{A^{5/12 }}+\frac{C\ep^2}{A}(t_\star+t)^2e^{-\delta\frac{t_\star+t}{A^{1/3}}}\right)B^2.
\end{align}
Integrating the above differential inequality on $[0,\delta^{-1}A^{1/3}]$ yields that
\begin{align}
A^{1/6}\|\mathbb{N}_{\neq}(t)\|_2^2+A^{3/4}\|\pa_x \mathbb{C}_{\neq}(t)\|_2^2+\|\na_Y   \mathbb{C}_{\neq}(t)\|_2^2\leq \exp\left\{\frac{C}{\delta A^{1/ 24}}+C\ep^2\right\}\left(\frac{C}{\delta A^{1/12}}+C\ep^2\right)B^2,\quad \forall t\in[0,\delta^{-1}A^{1/3}].
\end{align}
Now if $A$ and $\ep$ are small enough,  we have \eqref{Goal_F_comparison}. 
Hence the enhanced dissipation of $\mathcal{F}$ is achieved. 

\noindent
\myr{\textbf{Step \# 3: Proof of \eqref{ConED_L2}. (Check!?)}
For $t\in[0,\delta^{-1}A^{ 1/3}|\log A|^2)$
\begin{align}
\frac{d}{dt}&\frac{1}{2}\|n_{\neq}\|_2^2\\
\leq &-\frac{1}{2A}\|\na n_{\neq}\|_2^2+\frac{C}{A}(\|\na \cc_{\neq}\|_\infty^2+\|\na d_\nq\|_\infty^2)\|n_{\neq}\|_2^2+\frac{C}{A}\|\na_Y   \cz \|_\infty^2\|n_{\neq}\|_2^2+\frac{C}{A}(\|\na \cc_{\neq}\|_2^2+\|\na d_\nq\|_2^2)\|\nz \|_\infty^2\\
\leq &\frac{C(\CC)}{A}(\|n_{\neq}\|_2^2+1) +\frac{C( C_{n;L^\infty}) \ep^2 t^2e^{-\delta_0t/A^{1/3}}}{A}(\|n_{\neq}\|_2^2+1).
\end{align}
Hence for $A,\, \ep^{-2}$ chosen large enough compared to $\delta_0^{-1}, \CC$, we have $\|n_{\neq}(t)\|_2^2\leq 2\|n_{\mathrm{in};\nq}\|_2^2+2$ for $\forall t\in[0,\delta^{-1}A^{1/3}|\log A|^2]$.  
For $\|\na \cc_{\neq}\|_2$, we invoke the hypothesis \eqref{Hyp_pa_x_n_Linf} to obtain that 
\begin{align}
\frac{1}{2}\frac{d}{dt} \|\pa_x \cc_{\neq}\|_{2}^{2}\leq \frac{1}{A}\|\pa_x n\|_{2}\|\pa_x \cc_{\neq}\|_{2}\leq \frac{1}{A}C_{\pa_x n,\infty}\|\pa_x \cc_{\neq}\|_{2}.
\end{align}
Hence we have that $\|\pa_x \cc_\nq(t)\|_2\leq {\frac{t}{A}C_{\pa_x n,\infty}}$.
Next we estimate the $\pa_{z } \cc_\nq (t)$,
\begin{align}
\frac{1}{2}\frac{d}{dt}\|\pa_{z } \cc_\nq\|_2^2\leq \frac{1}{A}\|n_\nq\|_2^2.
\end{align} Hence on the time interval $[0, \delta^{-1} A^{1/3}|\log A|^2]$, the $\|\pa_{z } \cc_\nq\|_2^2$ is small as long as $A$ is chosen large enough compared to $C_{n;L^\infty}$. Finally, we estimate $\|\pa_{y }\cc_\nq\|_2^2$. Direct energy estimate yields that 
\begin{align}
\frac{1}{2}\frac{d}{dt}\|\pa_{y }\cc_\nq\|_2^2\leq &-\frac{1}{2A}\|\na \pa_{y } \cc_\nq\|_2^2+\|\pa_{y } u\|_\infty\|\pa_x \cc_\nq\|_2\|\pa_{y }\cc_\nq\|_2+\frac{\|n_{\neq}\|_2^2}{A} \\
\leq&\|\pa_{y }u\|_\infty \frac{t}{A} C_{\pa_x n;L^\infty}\|\pa_{y }\cc_\nq\|_2+\frac{1}{A}C_{n;L^\infty}^2 .
\end{align}  By taking $A$ large compared to $\|\pa_{y }u\|_\infty,\, C_{\pa_x n,\infty},\, C_{n;\infty}$, we have that $\|\pa_{y }\cc_\nq(t)\|_2\leq 2$ for $\forall t\in[0,\delta^{-1}A^{1/3}|\log A|^2]$.   
Now combining the enhanced dissipation with the energy estimate of the $L^2$ norm on $t\in[0,\delta^{-1}A^{ 1/3}|\log A|^2)$, we have \eqref{ConED_L2} with a slightly different $\delta$?. }
\end{proof}
\fi
\appendix
\section{}
\begin{lem}\label{Lem:nacest} 
Consider solution to the heat equation on $\Torus^d,\, d=1,2$,
\begin{align}
\pa_t \mathfrak{c}=\frac{1}{A}\de\mathfrak{c}+\frac{1}{A}\lf(\mathfrak{n}-\overline{\mathfrak n}\rg),\quad \mathfrak{c}(t=0, Y)= \mathfrak{c}_{\mathrm{in}}(Y),\quad \int_{\Torus^d} \mathfrak{c}_{\mathrm{in}}dY=0.\label{Heat_eq}
\end{align}
\myc{(Double check whether the formulation is correct for $\lan \cc\ran, \lan\lan\cc\ran\ran$!)} Then the following estimate holds for$\quad \frac{1}{q}<\frac{1}{d}+\frac{1}{p},\,1\leq q\leq p\leq \infty$,
\begin{align}\label{na_c_0_est}
\|\na^{m+1}  \mathfrak c(t)\|_{L^p(\Torus^d)}\leq&\frac{C}{A}\int_{0}^t\lf (\frac{t-\tau}{A}\rg)^{-\frac{1}{2}-\frac{d}{2}(\frac{1}{q}-\frac{1}{p})} e^{- \frac{t-\tau}{3A}}\|\na^m(\mathfrak n(\tau)-\overline{\mathfrak n})\|_{L^q(\Torus^d)}d\tau+\|\na^{m+1}\mathfrak c_{\mathrm{in}}\|_{L^p(\Torus^d)}\\
\leq & C\min\lf\{1,\lf(\frac{t}{A}\rg)^{\frac{1}{2}-\frac{d}{2}(\frac{1}{q}-\frac{1}{p})}\rg\}\sup_{0\leq\tau\leq t}\|\na^m(\mathfrak n(\tau)-\overline{\mathfrak n})\|_{L^q(\Torus^d)}+\|\na^{m+1}\mathfrak c_{\mathrm{in}}\|_{L^p(\Torus^d)}.
\end{align}
\end{lem}
\begin{proof}
The proof is similar to the proof of heat semigroup bounds in \cite{Winkler10}. 
Since the spacial derivatives commute with the differential equation \eqref{Heat_eq}, it is enough to derive the estimate \eqref{na_c_0_est_1} for $m=0$. Consider the heat equation 
\begin{align}
\pa_t h=\frac{1}{A}\de h, \quad h(t=0,X)=h_{\text{in}}(X),\quad  \overline{h_{\text{in}}}=0,\quad X\in\Torus^d.
\end{align}
Our first goal is to derive the following estimates

\noindent
a) If $d=1,2,\quad 1\leq q\leq p\leq \infty$, then 
\begin{align}\label{Winklersgp1}
\|e^{\frac{t}{A}\de_{\Torus^d }}h_{\text{in}}\|_p\leq C \left(\frac{t}{A}\right)^{-\frac{d}{2}(\frac{1}{q}-\frac{1}{p})} e^{- \frac{t}{2A}}\|h_{\text{in}}\|_q,\quad \forall t>0,
\end{align} 
holds for all $h_{\text{in}}\in L^q$ with $\overline{h_{\text{in}}}=0$;

\noindent
b) If $d=1,2,\quad 1\leq q\leq p \leq \infty$, then 
\begin{align}\label{Winklersgp2}
\|\na e^{\frac{t}{A}\de_{\Torus^d }}h_{\text{in}}\|_p\leq C \left(\frac{t}{A}\right)^{-\frac{1}{2}-\frac{d}{2}(\frac{1}{q}-\frac{1}{p})} e^{-\frac{t}{3A}}\|h_{\text{in}}\|_q,\quad \forall t>0.
\end{align}  

Application of a standard energy estimate yields that 
\begin{align}
\|e^{\frac{t}{A}\de}h_{\text{in}}\|_{L^2(\Torus^d)}\leq e^{-\frac{ t}{A} }\|h_{\text{in}}\|_{L^2(\Torus^d)},\quad \forall t\geq 0.\label{heat_L2L2}
\end{align} 

On the other hand, we recall the explicit formula for the Green's function of the heat semigroup
\begin{align}e^{\frac{t}{A}\de}h_\text{in}(X)=&\sum_{m_1\in \mathbb{Z}}\frac{1}{\sqrt{4\pi t/A} }\int_{X'\in \Torus}e^{-\frac{|X'+2m_1 \pi-X|^2}{4 t/A}}h_{\text{in}}(X')dX',\quad d=1;\\
e^{\frac{t}{A}\de}h_{\text{in}}(X)=&\sum_{(m_1, m_2)\in \mathbb{Z}^2}\frac{1}{4\pi t/A}\int_{X'\in \Torus^2}e^{-\frac{|X'+(2m_1 \pi, 2m_2 \pi)-X|^2}{4 t/A}}h_{\text{in}}(X')dX',\quad d=2.\label{heat_ker}
\end{align}Direct estimate yields that for $1\leq q\leq p\leq \infty$, 
\begin{align}\label{heat_Lp_Lq}
\|e^{\frac{t}{A}\de} h_{\text{in}}\|_{L_X^p(\Torus^d)}\leq& C \left(\frac{t}{A}\right)^{-\frac{d}{2}(\frac{1}{q}-\frac{1}{p})}\| h_{\text{in}}\|_{L_X^q(\Torus^d)},\quad \forall t\in[0, 2A];\\
\label{heatgradLpLq}\|\na e^{\frac{t}{A}\de} h_{\text{in}}\|_{L_X^p(\Torus^d)}\leq & C\left(\frac{t}{A}\right)^{-\frac{1}{2}}\|h_{\text{in}}\|_{L_X^p(\Torus^d)},\quad \forall t\in[0, 2A] .
\end{align}
For $p\in[1,2),\, q\in[1, p]$, application of the H\"older inequality and then $\eqref{heat_Lp_Lq}_{p=2}$ yields that
\begin{align}
\|h(t)\|_{L^p(\Torus^d)}\leq C\|h(t)\|_{L^2(\Torus^d)}\leq C e^{-\frac{t-A}{A}}\|h(A)\|_{L^2}\leq Ce^{-(\frac{t}{A}-1)}\|h_{\text{in}}\|_{L^q},\quad\forall t\geq 2A. 
\end{align}
If $p\in [2,\infty],\,  q\in[1, p]$, we have that by the semigroup property of $e^{\frac{t}{A}\de}$ and the estimates $\eqref{heat_Lp_Lq}_{q=2}$, \eqref{heat_L2L2}, 
\begin{align}
\|h(t)\|_{L^p(\Torus^d)}=&\|e^{\frac{A}{A}\de}h(t-A)\|_{L^p(\Torus^d)}\leq C\|h(t-A)\|_{L^2(\Torus^d)}=C\|e^{\frac{t-2A}{A}\de}h(A)\|_{L^2(\Torus^d)}\\
\leq &Ce^{-\frac{t-2A}{A}}\|h(A)\|_{L^2(\Torus^d)},\quad \forall t\geq 2A.
\end{align}
Now we apply the H\"older inequality ($q\leq 2$) or the estimate $\eqref{heat_Lp_Lq}_{p=2}$ ($q> 2$) to obtain
\begin{align}
\|h(t)\|_{L^p(\Torus^d)}\leq Ce^{- \frac{t}{A} }\|h_{\text{in}}\|_{L^q(\Torus^d)},\quad \forall t\geq 2A.
\end{align} Combining this estimate with \eqref{heat_Lp_Lq}, we have that for $1\leq q\leq p\leq \infty$, 
\begin{align}
\|e^{\frac{t}{A}\de}h_{\text{in}}\|_p\leq C\left(\left(\frac{t}{A}\right)^{-\frac{d}{2}(\frac{1}{q}-\frac{1}{p})}\mathbbm{1}_{t\leq 2A} +e^{- \frac{t}{A}}\right)\|h_{\text{in}}\|_q\leq C\left(\frac{t}{A}\right)^{-\frac{d}{2}(\frac{1}{q}-\frac{1}{p})}e^{- \frac{t}{2A}}\|h_{\text{in}}\|_q,\quad \forall t>0.
\end{align}
Hence the proof of \eqref{Winklersgp1} is completed. 

We proceed to prove \eqref{Winklersgp2}. 
For $t\leq 2A$, we combine the semigroup property and the estimates \eqref{Winklersgp1}, 
\eqref{heatgradLpLq}  to obtain the following gradient bound  
\begin{align} \label{Winklersgp21}
\|\na e^{\frac{t}{A}\de}h_{\text{in}}\|_{L^p(\Torus^d)}=&\|\na h(t)\|_{L^p(\Torus^d)}=\lf\|\na e^{\frac{t}{2A}\de}\lf( h\lf(\frac{t}{2}\rg)\rg)\rg\|_{L^p(\Torus^d)}\leq C\left(\frac{t}{2A}\right)^{-1/2}\lf\|h\lf(\frac{t}{2	}\rg)\rg\|_{L^p(\Torus^d)}\\
\leq& C\left(\frac{t}{2A}\right)^{-\frac{1}{2}}\left(\frac{t}{2A}\right)^{-\frac{d}{2}(\frac{1}{q}-\frac{1}{p})}\|h_{\text{in}}\|_{L^q(\Torus^d)}, \quad 1\leq q\leq p\leq\infty.
\end{align}
For $t\geq 2A$, we estimate the gradient with estimates  \eqref{Winklersgp1},  \eqref{heatgradLpLq} as follows:
\begin{align} \label{Winklersgp22}
\|\na e^{\frac{t}{A}\de}h_{\text{in}}\|_{L^p(\Torus^d)}=&\|\na h(t)\|_{L^p(\Torus^d)}=\|\na e^{\frac{A}{A}\de} (h(t-A))\|_{L^p(\Torus^d)}\leq C\|h({t}-{A})\|_{L^p(\Torus^d)}\\
\leq&C\left(\frac{t}{A}\right)^{-\frac{d}{2}(\frac{1}{q}-\frac{1}{p})}e^{-\frac{1}{2} (\frac{t}{A}-1)}\|h_{\text{in}}\|_{L^q(\Torus^d)}.
\end{align}
Combinbing \eqref{Winklersgp21} and \eqref{Winklersgp22} yields that
\begin{align*}
\|\na e^{\frac{t}{A}\de}h_{\text{in}}\|_{L^p(\Torus^d)}\leq& C\lf(\left(\frac{t}{A}\right)^{-\frac{1}{2}}\mathbbm{1}_{t\leq 2A}+e^{-\frac{1}{2} \frac{t}{A}}\mathbbm{1}_{t>2A}\rg)\left(\frac{t}{A}\right)^{-\frac{d}{2}(\frac{1}{q}-\frac{1}{p})}\|h_{\text{in}}\|_{L^q(\Torus^d)}\\
\leq& C\left(\frac{t}{A}\right)^{-\frac{1}{2}-\frac{d}{2}(\frac{1}{q}-\frac{1}{p})}e^{-\frac{t}{3A}}\|h_{\text{in}}\|_{L^q(\Torus^d)}.
\end{align*}
This completes the proof of \eqref{Winklersgp2}.

Next we observe that from \eqref{Heat_eq}, 
\begin{align}
\pa_t \pa_X^m\mathfrak{c}=\frac{1}{A}\de\pa_X^m\mathfrak{c}+\frac{1}{A}\pa_X^m\lf(\mathfrak{n}-\overline{\mathfrak n}\rg).
\end{align}
The Duhamel formulation yields that
\begin{align*}
\|\pa_X^{m+1}\mathfrak{c}(t)\|_{L^p}\leq& \|e^{\frac{t}{A}\de}\pa_X^{m+1}\mathfrak{c}_{\text{in}}\|_{L^p}+\frac{1}{A}\lf\|\int_0^t e^{\frac{t-s}{A}\de}\lf(\pa_X^{m+1}(\mathfrak{n}(\tau)-\overline{\mathfrak n})\rg)d\tau \rg\|_{L^p}\\
=&\|e^{\frac{t}{A}\de}\pa_X^{m+1}\mathfrak{c}_{\text{in}}\|_{L^p}+\frac{1}{A}\lf\|\int_0^t \pa_X e^{\frac{t-s}{A}\de}\lf(\pa_X^{m}(\mathfrak{n}(\tau)-\overline{\mathfrak n})\rg)d\tau \rg\|_{L^p}.
\end{align*}
Now we invoke the estimates \eqref{Winklersgp1} and \eqref{Winklersgp2} to obtain that
\begin{align}
\|\pa_X^{m+1}\mathfrak{c}(t)\|_{L^p}\leq C\|\pa_X^{m+1}\mathfrak{c}_{\text{in}}\|_{L^p}+\frac{C}{A}\int_0^t\lf (\frac{t-\tau}{A}\rg)^{-\frac{1}{2}-\frac{d}{2}(\frac{1}{q}-\frac{1}{p})} e^{- \frac{t-\tau}{3A}}\|\pa_X^m(\mathfrak n(\tau)-\overline{\mathfrak n})\|_{L^q(\Torus^d)}d\tau
\end{align}
If $ \frac{1}{q}<\frac{1}{d}+\frac{1}{p}$, then,
\begin{align}\label{na_c_0_est_1}\\
\|\pa_X^{m+1}\mathfrak{c}(t)\|_{L^p}\leq& C\|\pa_X^{m+1}\mathfrak{c}_{\text{in}}\|_{L^p}+C\sup_{0\leq\tau\leq t}\|\pa_X^m(\mathfrak n(\tau)-\overline{\mathfrak n})\|_{L^q(\Torus^d)}\int_0^t\lf (\frac{t-\tau}{A}\rg)^{-\frac{1}{2}-\frac{d}{2}(\frac{1}{q}-\frac{1}{p})} e^{- \frac{t-\tau}{3A}}\frac{1}{A}d\tau\\
\leq &C\|\pa_X^{m+1}\mathfrak{c}_{\text{in}}\|_{L^p}+C\sup_{0\leq\tau\leq t}\|\pa_X^m(\mathfrak n(\tau)-\overline{\mathfrak n})\|_{L^q(\Torus^d)}\int_{\max\{0,t-A\}}^t\lf (\frac{t-\tau}{A}\rg)^{-\frac{1}{2}-\frac{d}{2}(\frac{1}{q}-\frac{1}{p})} \frac{1}{A}d\tau\\
&+C\sup_{0\leq\tau\leq t}\|\pa_X^m(\mathfrak n(\tau)-\overline{\mathfrak n})\|_{L^q(\Torus^d)}\int_0^{\max\{0,t-A\}} e^{- \frac{t-\tau}{3A}}\frac{1}{A}d\tau\\
\leq&C\|\pa_X^{m+1}\mathfrak{c}_{\text{in}}\|_{L^p}+C\sup_{0\leq\tau\leq t}\|\pa_X^m(\mathfrak n(\tau)-\overline{\mathfrak n})\|_{L^q(\Torus^d)}.
\end{align}Another way to estimate the $L^p$-norm is that, 
\begin{align}\label{na_c_0_est_2}
\|\pa_X^{m+1}\mathfrak{c}(t)\|_{L^p}\leq& C\|\pa_X^{m+1}\mathfrak{c}_{\text{in}}\|_{L^p}+C\sup_{0\leq\tau\leq t}\|\pa_X^m(\mathfrak n(\tau)-\overline{\mathfrak n})\|_{L^q(\Torus^d)}\int_0^t\lf (\frac{t-\tau}{A}\rg)^{-\frac{1}{2}-\frac{d}{2}(\frac{1}{q}-\frac{1}{p})} \frac{1}{A}d\tau\\
\leq& C\|\pa_X^{m+1}\mathfrak{c}_{\text{in}}\|_{L^p}+C\sup_{0\leq\tau\leq t}\|\pa_X^m(\mathfrak n(\tau)-\overline{\mathfrak n})\|_{L^q(\Torus^d)} \lf(\frac{t}{A}\rg)^{\frac{1}{2}-\frac{d}{2}(\frac{1}{q}-\frac{1}{p})}, \quad\frac{1}{q}<\frac{1}{d}+\frac{1}{p} .
\end{align}
Combining \eqref{na_c_0_est_1} and \eqref{na_c_0_est_2} yields \eqref{na_c_0_est}.
\ifx
Previous argument:
\textcolor{red}{The theorem requires reproving! The idea is that we can periodically extend the $c_0$ to $\rr^2$ and then use the heat semigroup on $\rr^2$ to get the $L^p$ to $L^q$ estimate. Then we can use the spectral gap of the $L^2\rightarrow L^2$-estimate of the $e^{t\de_{\Torus^2}/A}$.}

The estimates of the $L^q$-norm of the chemical gradient $\|\na \lan c\ran^\iota\|_q$ \eqref{na_c_0_est} are natural consequences of the following two estimates
\noindent
a) If $1\leq q\leq p\leq \infty$, then 
\begin{align}\label{Winkler_semigroup_1}
\|e^{\frac{t}{A}\de}h_{\text{in}}\|_p\leq C_1\left(1+\left(\frac{t}{A}\right)^{-\frac{2}{2}(\frac{1}{q}-\frac{1}{p})}\right)e^{-\lambda_1 \frac{t}{A}}\|h_{\text{in}}\|_q,\quad \forall t>0,
\end{align} 
holds for all $w\in L^q$ with $\int w=0$;

\noindent
b) If $1\leq q\leq p \leq \infty$, then 
\begin{align}\label{Winkler_semigroup_2}
\|\na e^{\frac{t}{A}\de}h_{\text{in}}\|_p\leq C_2\left(1+\left(\frac{t}{A}\right)^{-\frac{1}{2}-\frac{n}{2}(\frac{1}{q}-\frac{1}{p})}\right)e^{-\lambda_1 t/A}\|h_{\text{in}}\|_q,\quad \forall t>0.
\end{align} 
Here $\lambda_1>0$ is the first nonzero eigenvalue of the negative Laplacian $-\de$. 

Combining the Duhamel formulation and the estimates \eqref{Winkler_semigroup_1}, \eqref{Winkler_semigroup_2}, one can integrate in time to derive \eqref{na_c_0_est}. 
The proof of the two inequalities can be found in  Lemma 1.3 in \cite{Winkler10}. In the paper, the estimates are proved that for bounded domain with Neumann boundary condition. The Green's function estimate is crucial. However, on the torus, we can explicitly write down the Green' function of the heat kernel and check all the conditions in the proof of Winkler. A standard bookkeeping yields  \eqref{Winkler_semigroup_1}, \eqref{Winkler_semigroup_2}.

{\color{blue}Consider the heat equation 
\begin{align}
\pa_t h=\frac{1}{A}\de h, \quad h(t=0,X)=h_{\text{in}}(X),\, \overline{h_{\text{in}}}=0,\,\, X\in\Torus^2.
\end{align} Through standard energy estimate, we have that 
\begin{align}
\|h(t)\|_{L^2(\Torus^2)}\leq e^{-\frac{\lambda_1t}{A} }\|h_{\text{in}}\|_{L^2(\Torus^2)},\quad \forall t\geq 0.
\end{align}
Here $\lambda_1>0$ is the first nonzero eigenvalue of the negative Laplacian $-\de$. 

On the other hand, the Green's function of the heat equation on the torus can be explicitly spelled out
\begin{align}\label{heat_kernel}
h(t,X)=\sum_{(m, n)\in \mathbb{Z}^2}\frac{1}{4\pi t/A}\int_{Y\in \Torus^2}e^{-\frac{|X'+(m L, n L)-X|^2}{4 t/A}}h_{\text{in}}(X')dX'.
\end{align}
Direct estimate of the Green's function yields that 
\begin{align}\label{heat_LpLq_decay}
\|h(t)\|_{L_X^p(\Torus^2)}\leq C \left(\frac{t}{A}\right)^{-\frac{2}{2}(\frac{1}{q}-\frac{1}{p})}\|h_{\text{in}}\|_{L_X^q(\Torus^2)},\quad ( \forall t\in[0, 2A]??).
\end{align}\textcolor{red}{(Check!)}

For $p\in[1,2)$, application of the H\"older inequality yields that
\begin{align}
\|h(t)\|_{L^p(\Torus^2)}\leq |\Torus|^{\frac{4}{2-p}}\|h(t)\|_{L^2(\Torus^2)}\leq C e^{-\lambda_1\frac{(t-1)}{A}}\|h(A)\|_{L^2}\leq Ce^{-\lambda_1\frac{(t-A)}{A}}\|h_{\text{in}}\|_{L^q},\quad\forall t\geq 2A. 
\end{align}
If $p\in [2,\infty]$, we have
\begin{align}
\|h(t)\|_{L^p(\Torus^2)}\leq C\|h(t-A)\|_{L^2(\Torus^2)}\leq Ce^{-\lambda_1(\frac{t-2A}{A})}\|h(A)\|_{L^2(\Torus^2)}\leq Ce^{-\lambda_1\left(\frac{t}{A}-2\right)}\|h_{\text{in}}\|_{L^q(\Torus^2)},\quad \forall t\geq 2A.
\end{align}
For $t\leq 2A$, the estimate  \eqref{Winkler_semigroup_1} is direct from \eqref{heat_LpLq_decay}. Hence the proof of \eqref{Winkler_semigroup_1} is completed. 

In order to prove \eqref{Winkler_semigroup_2}, we first note that {\bf(Check!)}
\begin{align}
\|\na e^{\frac{t}{A}\de} h_{\text{in}}\|_{L^p(\Torus^2)}\leq C\left(\frac{t}{A}\right)^{-1/2}\|h_{\text{in}}\|_{L^p(\Torus^2)},\quad \forall t\leq A.
\end{align} 
For $t\leq 2A$, we can estimate the gradient with \eqref{Winkler_semigroup_1} as follows
\begin{align}
\|\na e^{\frac{t}{A}\de}h_{\text{in}}\|_{L^p(\Torus^2)}=&\|(\na e^{\frac{t}{2A}\de})e^{\frac{t}{2A}\de} h_{\text{in}}\|_{L^p(\Torus^2)}\leq C\left(\frac{t}{2A}\right)^{-1/2}\|e^{\frac{t}{2	A}\de}h_{\text{in}}\|_{L^p(\Torus^2)}\\
\leq& C\left(\frac{t}{2A}\right)^{-1/2}C_1\left(1+\left(\frac{t}{2A}\right)^{-\frac{2}{2}(\frac{1}{q}-\frac{1}{p})}\right)\|h_{\text{in}}\|_{L^q(\Torus^2)}.
\end{align}
{\bf I am not sure whether $\na e^{\frac{t}{A}\de}=(\na e^{\frac{t}{2A}\de})e^{\frac{t}{2A}\de}$ is justified. It feels like there is a $`2'$ factor?}
For $t\geq 2A$, we estimate it as follows:
\begin{align}
\|\na e^{\frac{t}{A}\de}h_{\text{in}}\|_{L^p(\Torus^2)}=&\|(\na e^{\de})e^{\frac{t-A}{A}\de} h_{\text{in}}\|_{L^p(\Torus^2)}\leq C\|h({t}-{A})\|_{L^p(\Torus^2)}\\
\leq&C\left(1+\left(\frac{t}{A}\right)^{-\frac{2}{2}(\frac{1}{q}-\frac{1}{p})}\right)e^{-\lambda_1 (\frac{t}{A}-1)}\|h_{\text{in}}\|_{L^q(\Torus^2)}.
\end{align}
This completes the proof.
 }\fi

\end{proof}

\noindent
\thanks{\textbf{Acknowledgment.} SH was supported in part by NSF grants DMS 2006660, DMS 2304392, DMS 2006372. The author would like to thank Tarek Elgindi for providing the key construction for the time-dependent shear flows and many crucial suggestions,  discussions and guidance.  }

\small
\bibliographystyle{abbrv}
\bibliography{nonlocal_eqns,JacobBib,SimingBib}

\end{document}